\documentclass[reqno,10pt,centertags]{amsart}
\usepackage{amsmath,amsthm,amscd,amssymb,latexsym,upref,enumerate,stmaryrd}
\usepackage{amsfonts,mathrsfs}
\usepackage{esint}
\usepackage{color}
\usepackage{hyperref} 

\newcommand{\R}{{\bbR}}

\newcommand{\bbC}{{\mathbb{C}}}

\newcommand{\bbN}{{\mathbb{N}}}

\newcommand{\bbR}{{\mathbb{R}}}

\newcommand{\cB}{{\mathcal B}}
\newcommand{\cC}{{\mathcal C}}
\newcommand{\cD}{{\mathcal D}}
\newcommand{\cE}{{\mathcal E}}

\newcommand{\cH}{{\mathcal H}}

\newcommand{\cK}{{\mathcal K}}

\newcommand{\cM}{{\mathcal M}}

\newcommand{\cR}{{\mathcal R}}
\newcommand{\cS}{{\mathcal S}}

\newcommand{\cV}{{\mathcal V}}
\newcommand{\cW}{{\mathcal W}}
\newcommand{\cX}{{\mathcal X}}

\newcommand{\gR}{\mathfrak R}

\newcommand{\ga}{\mathfrak a}
\newcommand{\gb}{\mathfrak b}

\newcommand{\gq}{\mathfrak q} 

\newcommand{\B}{\vec{B}}

\newcommand{\Q}{\widetilde{Q}}

\newcommand{\uu}{\overline{u}}

\newcommand{\e}{{\varepsilon}}

\newcommand{\U}{\Upsilon}
\newcommand{\Ut}{\widetilde{\Upsilon}}

\newcommand{\dott}{\,\cdot\,}

\newcommand{\eps}{\varepsilon}

\newcommand{\Sect}{\text{\rm Sect}}

\DeclareMathOperator{\diam}{diam}

\DeclareMathOperator{\dist}{dist}
\DeclareMathOperator{\supp}{supp}

\DeclareMathOperator{\ran}{ran}
\DeclareMathOperator{\dom}{dom}

\DeclareMathOperator{\tr}{tr}

\renewcommand{\Re}{\text{\rm Re}}
\renewcommand{\Im}{\text{\rm Im}}
\renewcommand{\ln}{\text{\rm ln}}

\newcommand{\beq}{\begin{equation}}
\newcommand{\enq}{\end{equation}}

\newcommand{\no}{\notag}
\newcommand{\lb}{\label}
\newcommand{\f}{\frac}

\newcommand{\ol}{\overline}

\newcommand{\wti}{\widetilde}
\newcommand{\Oh}{O}

\newcommand{\hatt}{\widehat}
\newcommand{\bi}{\bibitem}
\newcommand{\vp}{\varphi}
\newcommand{\vpt}{\tilde{\varphi}}

\renewcommand{\ge}{\geqslant}
\renewcommand{\le}{\leqslant}

\let\geq\geqslant
\let\leq\leqslant

\newcommand{\dv}{\operatorname{div}}

\makeatletter
\def\theequation{\@arabic\c@equation}

\newcommand*{\mailto}[1]{\href{mailto:#1}{\nolinkurl{#1}}}
\newcommand{\arxiv}[1]{\href{http://arxiv.org/abs/#1}{arXiv:#1}}

\allowdisplaybreaks 
\numberwithin{equation}{section}

\newtheorem{theorem}{Theorem}[section]

\newtheorem{lemma}[theorem]{Lemma}
\newtheorem{corollary}[theorem]{Corollary}

\newtheorem{hypothesis}[theorem]{Hypothesis}

\theoremstyle{remark}
\newtheorem{remark}[theorem]{Remark}

\begin{document}

\title[Stability of Square Root Domains]{On Stability of Square Root Domains for Non-Self-Adjoint Operators Under Additive Perturbations}

\author[F.\ Gesztesy]{Fritz Gesztesy}
\address{Department of Mathematics,
University of Missouri, Columbia, MO 65211, USA}
\email{\mailto{gesztesyf@missouri.edu}}
\urladdr{\url{http://www.math.missouri.edu/personnel/faculty/gesztesyf.html}}

\author[S.\ Hofmann]{Steve Hofmann}
\address{Department of Mathematics,
University of Missouri, Columbia, MO 65211, USA}
\email{\mailto{hofmanns@missouri.edu}}
\urladdr{\url{http://www.math.missouri.edu/$\sim$hofmann/}}

\author[R.\ Nichols]{Roger Nichols}
\address{Mathematics Department, The University of Tennessee at Chattanooga, 
415 EMCS Building, Dept. 6956, 615 McCallie Ave, Chattanooga, TN 37403, USA}
\email{\mailto{Roger-Nichols@utc.edu}}
\urladdr{\url{http://www.utc.edu/faculty/roger-nichols/index.php}}

\thanks{To appear in {\it Mathematika.}}
\date{\today}

\begin{abstract} 
Assuming $T_0$ to be an m-accretive operator in the complex Hilbert space $\cH$, 
we use a resolvent method due to Kato to appropriately define 
the additive perturbation $T = T_0 + W$ and prove 
stability of square root domains, that is, 
$$
\dom\big((T_0 + W)^{1/2}\big) = \dom\big(T_0^{1/2}\big). 
$$
Moreover, assuming in addition that $\dom\big(T_0^{1/2}\big) = \dom\big((T_0^*)^{1/2}\big)$, 
we prove stability of square root domains in the form 
$$
\dom\big((T_0 + W)^{1/2}\big) = \dom\big(T_0^{1/2}\big) = \dom\big((T_0^*)^{1/2}\big) 
= \dom\big(((T_0 + W)^*)^{1/2}\big), 
$$
which is most suitable for PDE applications. 

We apply this approach to elliptic second-order partial differential operators of the 
form 
$$
- {\rm div}(a\nabla \, \cdot \,) + \big(\B_1\cdot \nabla \cdot \big) + \dv \big(\B_2 \cdot \big) + V    
$$
in $L^2(\Omega)$ on certain open sets $\Omega \subseteq \bbR^n$, $n \in \bbN$, with 
Dirichlet, Neumann, and mixed boundary conditions on $\partial \Omega$, under general 
hypotheses on the (typically, nonsmooth, unbounded) coefficients and on $\partial\Omega$. 
\end{abstract}

\maketitle

\section{Introduction}  \lb{s1}

The principal aim of this paper is to discuss an abstract approach to the problem of stability 
of square root domains of a class of non-self-adjoint operators with respect to additive 
perturbations, and to apply it to the concrete case of uniformly elliptic second-order partial 
differential operators of the form 
\begin{equation} 
- {\rm div}(a\nabla \, \cdot \,) + V   \lb{1.0} 
\end{equation} 
in $L^2(\Omega)$ on certain open sets $\Omega \subseteq \bbR^n$, $n \in \bbN$, assuming 
$a \in L^\infty (\Omega)$, under  
various boundary conditions on $\partial\Omega$, and under fairly general conditions on the 
(typically, nonsmooth, unbounded) coefficients $\B_j$, $j=1,2$, and $V$. In addition, we discuss in 
great detail the operator 
\begin{equation} 
- {\rm div}(a\nabla \, \cdot \,) + \big(\B_1\cdot \nabla \cdot \big) + \dv \big(\B_2 \cdot \big) + V   \lb{1.0A} 
\end{equation} 
in $L^2(\bbR^n)$, $n\in\bbN$, $n\geq 2$, with optimal $L^p$-conditions on $a$, $\B_j$, $j=1,2$, 
and $V$ of the type 
\begin{align}
a \in L^\infty (\bbR^n)^{n \times n}, \quad 
\B_1,\B_2 \in L^n(\bbR^n)^n + L^\infty(\bbR^n)^n, \quad V\in L^{n/2}(\bbR^n) + L^\infty(\bbR^n). 
\end{align}
   
In a nutshell, we are interested in the solution of the following abstract problem: if $T_0$ is an  
appropriate non-self-adjoint operator, for which non-self-adjoint additive perturbations $W$ of 
$T_0$ can one conclude that  
\begin{equation}
\dom\big((T_0 + W)^{1/2}\big) = \dom\big(T_0^{1/2}\big),    \lb{1.0a}
\end{equation}
and hence conclude stability of square root domains under additive perturbations? 
More precisely, driven by applications to PDEs, we are particularly interested in the following variant of 
this stability problem for square root domains: if $T_0$ is an  
appropriate non-self-adjoint operator for which it is known that 
\begin{equation}
\dom\big(T_0^{1/2}\big) = \dom\big((T_0^*)^{1/2}\big)    \lb{1.1} 
\end{equation}
is valid, for which non-self-adjoint additive perturbations $W$ of $T_0$ can one conclude that also 
\begin{equation}
\dom\big((T_0 + W)^{1/2}\big) = \dom\big(T_0^{1/2}\big) = \dom\big((T_0^*)^{1/2}\big) 
= \dom\big(((T_0 + W)^*)^{1/2}\big).     \lb{1.2}
\end{equation}
Of course, in 
answering such a question one first needs to be able to define the square root of a 
non-self-adjoint operator and it also necessitates a detailed discussion of how to interpret the sum of 
$T_0$ and $W$. We postpone a more detailed account of these issues for now, and instead describe 
a bit of the motivation behind posing such a problem (see also Remark \ref{r2.13} which provides 
additional motivations in terms of trace formulas and Fredholm determinants). 

First of all, when discussing variational approaches to PDE (eigenvalue) problems, formulations in 
terms of quadratic (resp., sesquilinear) forms have become a standard tool and hence the quadratic 
form of a non-self-adjoint operator, say, $T_0$,  in a complex, separable Hilbert space $\cH$, naturally 
enters the scene. Assuming the non-self-adjoint operator $T_0$ to be m-accretive (cf.\ \eqref{2.7A})  
so its fractional powers can be defined in a standard manner, such a sesquilinear form $\gq_{T_0}^{}$, associated with $T_0$, should be of the type 
\begin{equation}
\gq_{T_0}^{}(f,g) = \big((T_0^*)^{1/2} f, T_0^{1/2} g\big)_{\cH} \, \text{ for } \,  f, g \in \dom(\gq_{T_0}^{}), 
\lb{1.3} 
\end{equation}
such that one can expect to have 
\begin{equation}
\gq_{T_0}^{}(f,g) = (f, T_0 g)_{\cH} \, \text{ for } \,   f \in \dom(\gq_{T_0}^{}) \,
\text{ and } \, g \in \dom(T_0).      \lb{1.4} 
\end{equation}
However, this instantly begs the question of what to consider for $\dom(\gq_{T_0}^{})$? Given \eqref{1.3}, 
it would be natural to hope that the common square root domain in \eqref{1.1} can serve for 
$\dom(\gq_{T_0}^{})$, that is, one would hope for
\begin{equation}
\dom(\gq_{T_0}^{}) = \dom\big(T_0^{1/2}\big) = \dom\big((T_0^*)^{1/2}\big),     \lb{1.5}
\end{equation}
and hence it is natural to hope for the same in connection with its analog for the perturbed situation 
with $T_0$ replaced by $T_0 + V$, see \eqref{1.2}.

As it turned out, for an m-accretive operator $T_0$, Kato \cite{Ka61} indeed proved in 1961 that 
\begin{equation}
\dom\big(T_0^{\alpha}\big) = \dom\big((T_0^*)^{\alpha}\big), \quad \alpha \in (0,1/2),    \lb{1.6} 
\end{equation}
and he also showed that equality in \eqref{1.6} is violated in general if $\alpha > 1/2$, but the 
important case $\alpha = 1/2$, and hence precisely what is needed in connection with sesquilinear 
forms, was left open at that point. Within a year, a counter example (the operator associated with 
$d/dx$ on $W^{1,2}_0((0,\infty))$ in $L^2((0,\infty))$) to relation \eqref{1.1} was provided by 
Lions \cite{Li62} in 1962 (see also \cite{Mc70} for additional material on such counter examples). 
Moreover, a counter example invoking a sectorial operator (cf.\ the 
paragraph following \eqref{2.7A}) was provided by McIntosh \cite{Mc72} in 1972. In particular, 
as a consequence of additional results by Kato \cite{Ka62}, this implies that also the form domain 
and the square root domain differed from each other in McIntosh's counter example. 

Digressing a bit, we note that while these negative results ruled out the validity of \eqref{1.2} 
for m-sectorial operators in 
general, they were not the end of the story. Quite to the contrary, and due to a great extent as 
a consequence of continuing efforts by McIntosh over the following decades, the emphasis shifted 
to concrete PDE situations for which an affirmative answer to \eqref{1.1} was feasible, leading to the 
notion of the Kato square root problem for elliptic differential operators. In this 
context we refer to \cite{CMM82}, \cite{Fu67}, \cite{KM85}, \cite{Mc82}, \cite{Mc85}, 
\cite{Mc90}, \cite{Mi91}, where second-order elliptic partial differential operators of the type 
\begin{equation}
-\sum_{1\leq j,k\leq n} \partial_j a_{j,k} (x) \partial_k + \sum_{1 \leq j \leq n} b_j (x) \partial_j  
+ \sum_{1 \leq j \leq n} \partial_j (c_j (x) \, \cdot \,) + d(x),      \lb{1.7} 
\end{equation}
are discussed in $L^2(\Omega)$, $\Omega \subseteq \bbR^n$ open, under mild 
regularity assumptions on $\partial \Omega$ and the coefficients $a_{j,k}$, $b_j$, $c_j$, and $d$ 
(or certain symmetry hypotheses on these coefficients). The problem continued its great attraction 
throughout the 1990s as is evidenced by the efforts in \cite{Al91}, \cite{AO12}, 
\cite{AMN97}, \cite{AT91}, \cite{AT92}, \cite{AT96},  
\cite{ER97}, \cite{Gu92}, \cite{Jo91}, \cite{Ya84}, \cite{Ya87}, culminating in the treatise 
by Auscher and Tchamitchian \cite{AT98} (see also \cite{AT01}). The final breakthrough and 
the complete solution of Kato's square root problem for elliptic operators of the type 
$- {\rm div}(a\nabla \, \cdot \,)$ on $\bbR^n$ with $L^\infty$-coefficients 
$a_{j,k}$, that is, without any smoothness hypotheses on $a_{j,k}$, occurred in 2001 in papers 
by Auscher, Hofmann, Lacey, Lewis, McIntosh, and Tchamitchian \cite{AHLLMT01}, 
\cite{AHLMT02}, \cite{AHLT01}, \cite{AHMT01}, \cite{HLM02} (see also \cite{Au02}, \cite{Ho01}, 
\cite{Ho01a}). For subsequent developments, including higher-order operators, operators on 
Lipschitz domains $\Omega \subset \bbR^n$, $L^p$-estimates, 
and mixed boundary conditions, leading up to recent work in this area, we also refer to 
\cite{Au04}, \cite{Au07}, \cite{AAM10}, \cite{AAM10a}, \cite{ABHDR12}, \cite{AT03}, \cite{AKM06}, \cite{BtEM12}, \cite{CUR12}, \cite{HM03}, \cite{HMP08}, \cite{Mo12}, \cite[Ch.\ 8]{Ou05}, 
\cite[Sect.\ 2.8, Ch.\ 16]{Ya10}.  

Returning to the principal topic of this paper, we will now assume $T_0$ to be an m-accretive 
operator from the outset and describe a class of 
non-self-adjoint perturbations $W$ of $T_0$ such that $T_0 + W$ is closed and densely defined 
(thereby giving precise meaning to the ``sum'' of $T_0$ and $W$) and such that the stability of 
square root domains \eqref{1.0a} holds. To define the sum $T$ of $T_0$ and $W$, we use 
a factorization approach $W = B^* A$ 
(permitting a two-Hilbert space approach, where $A$, $B$ act between possibly different Hilbert 
spaces $\cH$ and $\cK$) and follow a device due to Kato \cite{Ka66}, which introduces the resolvent 
of $T$ rather than $T$ itself in the form  
\begin{align}\lb{1.8}
\begin{split} 
&(T-zI_{\cH})^{-1} = (T_0-zI_{\cH})^{-1}    \\
& \quad -\overline{(T_0-zI_{\cH})^{-1}B^*}
\big[I_{\cK} - \ol{A (T -z I_{\cH})^{-1} B^*}\big]^{-1}A(T_0-zI_{\cH})^{-1}.     
\end{split}
\end{align}
In particular, under natural hypotheses on $A$ and $B$, $T$ defined in terms of \eqref{1.8}, 
represents a closed and densely defined extension of the operator (or form) sum $T_0 + B^* A$ 
(cf.\ Section \ref{s2} and Appendix \ref{sA} for details). Given $T$, and under the assumption 
that $T$ is m-accretive, we then employ the standard representation of square roots of 
m-accretive operators, (cf., e.g., \cite[Sect.\ V.3.11]{Ka80})
\begin{equation}\lb{1.9}
(T+EI_{\cH})^{-1/2}=\frac{1}{\pi}\int_0^{\infty}d\lambda \, \lambda^{-1/2}(T+(E+\lambda)I_{\cH})^{-1}, 
\quad E>0, 
\end{equation}
and combine it with the resolvent equation \eqref{1.8} for $T$ in terms of the resolvent of $T_0$. 
Using the analog of \eqref{1.9} with $T$ replaced by $T_0$, and under appropriate integrability 
conditions on 
\begin{equation} 
\lambda^{-1} \big\|A(T_0+(\lambda+E)I_{\cH})^{-1/2}\big\|_{\cB(\cH,\cK)} 
\big\|\overline{(T_0+(\lambda+E)I_{\cH})^{-1/2}B^{\ast}}\big\|_{\cB(\cK,\cH)},     \lb{1.10} 
\end{equation} 
with respect to $\lambda$ near $+\infty$, we reduce $(T+EI_{\cH})^{-1/2}$ to a 
multiplicative perturbation of $(T_0+EI_{\cH})^{-1/2}$ (cf.\ \eqref{2.32}), which enables us to prove 
\eqref{1.0a} (and subsequently, \eqref{1.2} given \eqref{1.1}). 

In this context we emphasize that additive perturbation problems have of course been studied in the 
literature. We refer, for instance, to \cite{AT98}, \cite{Mc82}, \cite{Mc85}, \cite{Mc90}, and \cite{Ya87}, 
where applications to second-order uniformly elliptic differential operators, typically, with bounded 
coefficients, have been treated. Moreover, the case of three-dimensional Schr\"odinger type 
operators with (strongly singular) Rollnik potentials was studied in \cite{Ya87}. (In addition, we note that multiplicative perturbations were also discussed in 
\cite{AMN97a}.) Our principal objective in this paper was to develop an abstract perturbation approach 
to this circle of ideas that is sufficiently general to comprise applications to PDE operators of the type   
\eqref{1.0} under very general conditions on the nonsmooth coefficients $a$ and $V$. 

Next, we briefly turn to a description of the content of each section: Section \ref{s2} contains our 
abstract approach to the stability of square root domains \eqref{1.2}. We describe Kato's additive 
perturbation approach defining the resolvent of $T$ in terms of $T_0$ and $W = B^* A$, and prove 
our principal stability results, Theorem \ref{t2.4}, Corollary \ref{c2.5a}, and a variation of it, 
Theorem \ref{t2.12}, which 
permits a straightforward application to PDE situations in our subsequent Section \ref{s3}. In it we 
consider various $L^2$-realizations of $- {\rm div}(a\nabla \, \cdot \,) $ and 
its additive perturbation, an operator of multiplication by the function $V$, in the Hilbert space 
$L^2(\Omega; d^n x)$, $n \in \bbN$. We explicitly treat Dirichlet, Neumann, and mixed boundary 
conditions on $\partial \Omega$, under varying assumptions on $\Omega$ (from just an open 
subset of $\bbR^n$ in the Dirichlet case to $\Omega \subseteq \bbR^n$ a strongly Lipschitz 
domain in the case of Neumann or mixed boundary conditions), see Theorem \ref{t3.12}. 
Particular emphasis is also given to the special case $\Omega = \bbR^n$ (cf.\ Theorem \ref{t3.9}). 
Our conditions on $V$ are of the type $V \in L^p(\Omega) + L^\infty(\Omega)$ with $p > n/2$.  The 
case of critical (i.e., optimal) $L^p$-conditions is the prime objective in Section \ref{s4} which 
is devoted to a proof of the fact that the square root domain associated with 
\begin{equation}\label{1.11}
- {\rm div}(a\nabla \, \cdot \,)  + \big(\B_1 \cdot \nabla \cdot \big) + \dv \big(\B_2 \cdot \big) + V,
\end{equation}
in $L^2(\bbR^n)$ under the conditions 
\begin{equation} 
a \in L^\infty(\bbR^n)^{n \times n}, \quad 
\B_1,\B_2 \in L^n(\bbR^n)^n + L^\infty(\bbR^n)^n, \quad V\in L^{n/2}(\bbR^n) + L^\infty(\bbR^n),    
\lb{1.12} 
\end{equation} 
is given by the standard Sobolev space $W^{1,2}(\bbR^n)$ for $n \geq 3$ (the case $n=2$ is 
more subtle, see Remark \ref{r4.6}\,$(ii)$). We also note that in our applications we focus primarily 
on dimensions $n \geq 2$. The case $n=1$ with $\Omega \subseteq \bbR$ an open interval will be 
considered separately in great detail in \cite{GHN13}. The principal result of Appendix \ref{sA} 
establishes, in short, equality between $T$ defined according to Kato's resolvent method \eqref{1.8} 
on one hand, and the form sum of $T_0$ and $B^* A$ on the other.  
 
Finally, we briefly summarize some of the notation used in this paper: 
Let $\cH$ be a separable complex Hilbert space, $(\cdot,\cdot)_{\cH}$ the scalar product in $\cH$
(linear in the second argument), and $I_{\cH}$ the identity operator in $\cH$. 
Next, if $T$ is a linear operator mapping (a subspace of) a Hilbert space into another, then 
$\dom(T)$ and $\ker(T)$ denote the domain and kernel (i.e., null space) of $T$. 
The closure of a closable operator $S$ is denoted by $\ol S$. 
The spectrum
and resolvent set 
of a closed linear operator in a Hilbert space will be abbreviated by $\sigma(\cdot)$ 
and $\rho(\cdot)$, respectively. 
The form sum of two (appropriate) operators $T_0$ and $W$ is denoted by $T_0 +_{\gq} W$. 

The Banach spaces of bounded and compact linear operators on a separable complex Hilbert 
space $\cH$ are denoted by $\cB(\cH)$ and $\cB_\infty(\cH)$, respectively; the corresponding 
$\ell^p$-based trace ideals will be denoted by $\cB_p (\cH)$, $p>0$. The trace of trace class 
operators in $\cH$ is denoted by ${\tr}_{\cH}(\cdot)$, similarly, ${\det}_{\cH}(\cdot)$ abbreviates 
the Fredholm determinant.  The analogous notation 
$\cB(\cX_1,\cX_2)$, $\cB_\infty (\cX_1,\cX_2)$, etc., will be used for bounded and compact 
operators between two Banach spaces $\cX_1$ and $\cX_2$. 
Moreover, $\cX_1\hookrightarrow \cX_2$ denotes the continuous embedding
of the Banach space $\cX_1$ into the Banach space $\cX_2$. 
 
In addition, $I_n$ denotes the $n\times n$ identity matrix in $\bbC^n$, and  
we use the convention that $\lesssim$ abbreviates $\leq C$ for an appropriate constant 
$C>0$. Moreover, the symbol $\approx$ between two norms indicates equivalent norms. 
Finally, we abbreviate $L^p(\Omega; d^n x) := L^p(\Omega)$ and 
$L^p(\Omega, \bbC^n; d^n x) := L^p(\Omega)^n$.

\section{Some Abstract Results}  \lb{s2}

In this section we present our abstract results on the stability of square root 
domains: Given an m-accretive operator $T_0$, we determine conditions on 
an additive perturbation $W$ such that 
the square root domain of $T_0 + W$ coincides with that of $T_0$.

We start by with an appropriate definition of the sum of $T_0$ and $W$ 
following Kato \cite{Ka66} (see also the slight extension in \cite{GLMZ05}): 

\begin{hypothesis}\lb{h2.1}
$(i)$  Suppose that $T_0$ is a densely defined, closed, linear operator in $\cH$ with nonempty resolvent set,
\begin{equation}\lb{2.1}
\rho(T_0)\neq \emptyset,
\end{equation}
$A:\dom(A)\rightarrow \cK$, $\dom(A)\subseteq \cH$, is a densely defined, closed, linear operator from $\cH$ to $\cK$, and $B:\dom(B)\rightarrow \cK$, $\dom(B)\subseteq \cH$, is a densely defined, closed, linear operator from $\cH$ to $\cK$ such that
\begin{equation}\lb{2.2}
\dom(A)\supseteq \dom(T_0), \quad \dom(B)\supseteq \dom(T_0^{\ast}).
\end{equation}
$(ii)$ For some $($and hence for all\,$)$ $z\in \rho(T_0)$, the operator $-A(T_0-zI_{\cH})^{-1}B^{\ast}$, defined on $\dom(B^{\ast})$, has a bounded extension in $\cK$, denoted by $K(z)$,
\begin{equation}\lb{2.3}
K(z)=-\overline{A(T_0-zI_{\cH})^{-1}B^*} \in \cB(\cK).
\end{equation}
$(iii)$  $1\in \rho(K(z_0))$ for some $z_0\in \rho(T_0)$. 
\end{hypothesis}

Suppose Hypothesis \ref{h2.1} holds and define
\begin{equation}\lb{2.4}
\begin{split}
R(z)=(T_0-zI_{\cH})^{-1}-\overline{(T_0-zI_{\cH})^{-1}B^*}[I_{\cK}-K(z)]^{-1}A(T_0-zI_{\cH})^{-1},&\\
z\in \{\zeta\in \rho(T_0)\,|\,1\in \rho(K(\zeta))\}.&
\end{split}
\end{equation}
Under the assumptions of Hypothesis \ref{h2.1}, $R(z)$ given by \eqref{2.4} defines a densely defined, closed, linear operator $T$ in $\cH$ (cf. \cite{GLMZ05}, \cite{Ka66}) by
\begin{equation}\lb{2.5}
R(z)=(T-zI_{\cH})^{-1}, \quad z\in \{\zeta\in \rho(T_0) \, | \, 1\in \rho(K(\zeta))\}.
\end{equation}
Combining \eqref{2.4} and \eqref{2.5} yields Kato's resolvent equation (in the slightly more general form of \cite{GLMZ05}, in which $T_0$ is no longer assumed to be self-adjoint), 
\begin{align}\lb{2.6}
& (T-zI_{\cH})^{-1}=(T_0-zI_{\cH})^{-1}-\overline{(T_0-zI_{\cH})^{-1}B^*}
[I_{\cK}-K(z)]^{-1}A(T_0-zI_{\cH})^{-1},    \no \\
& \hspace*{6.6cm} z\in \{\zeta\in \rho(T_0)\,|\, 1\in \rho(K(\zeta))\}. 
\end{align}
The operator $T$ defined by \eqref{2.5} is an extension of $(T_0+B^*A)|_{\dom(T_0)\cap \dom(B^{\ast}A)}$,
\begin{equation}\lb{2.7}
T\supseteq (T_0+B^{\ast}A)|_{\dom(T_0)\cap \dom(B^{\ast}A)}. 
\end{equation} 

We add that the operator sum $(T_0+B^*A)|_{\dom(T_0)\cap \dom(B^*A)}$ can be problematic 
since it is possible that $\dom(T_0)\cap \dom(B^*A)=\{0\}$ (for a pertinent example, see, e.g., 
\cite[Sect.\ I.6]{Si71}; see also \cite{SV85}).  In light of \eqref{2.7}, $T$ defined by \eqref{2.5} 
should be viewed as a generalized sum of $T_0$ and $B^*A$, that is, the perturbation $W$ of 
$T_0$ has been factored according to $W = B^* A$. Under the additional assumption that 
$\dom\big(T_0^{1/2}\big) = \dom\big((T_0^*)^{1/2}\big)$, and the additional hypotheses \eqref{2.7B}
on $A$ and $B$, we will show in Theorem \ref{tA.3} that $T$ also extends the form sum of $T_0$ 
and $B^* A$. 

We recall that a linear operator $D$ in $\cH$ is called {\it accretive} if the numerical range of $D$ 
(i.e., the set $\{(f,Df)_{\cH}\in\bbC \,|\, f \in \dom(D), \, \|f\|_{\cH}=1\}$) is a subset of the closed 
right complex half-plane. 
$D$ is called {\it m-accretive} if $D$ is a closed and maximal accretive operator (i.e., $D$ has no proper accretive extension). One recalls that an equivalent definition of an m-accretive operator $D$ in $\cH$ reads
\begin{equation}
(D + \zeta I_{\cH})^{-1} \in \cB(\cH), \quad
\big\|(D + \zeta I_{\cH})^{-1} \big\|_{\cB(\cH)} \leq [\Re(\zeta)]^{-1}, \quad \Re(\zeta) > 0.   \lb{2.7A}
\end{equation}
One also recalls that any m-accretive operator is necessarily densely defined. Following 
\cite[p.~279]{Ka80}, 
$D$ is called {\it quasi-m-accretive} if for some $\alpha \in \bbR$, $D + \alpha I_{\cH}$ is m-accretive.  
Moreover,  $D$ is called an {\it m-sectorial} operator with a {\it vertex} $\gamma \in \bbR$ and a corresponding {\it semi-angle} $\theta\in [0,\pi/2)$, in short, {\it m-sectorial}, if $D$ is 
quasi-m-accretive, closed (and hence densely defined) operator, and the numerical range of $D$ is contained 
in the sector $\cS_{\gamma,\theta}$.

We let $S_{\gamma,\theta}$, $\gamma \in \bbR$, $\theta\in [0,\pi/2)$, denote the sector with vertex located at $\gamma$ and semi-angle $\theta$,
\begin{equation}\lb{2.7aa}
\cS_{\gamma,\theta}:=\{\zeta \in \bbC\, |\, |\text{arg}(\zeta-\gamma)|\leq \theta\}, 
\quad \theta \in [0,\pi/2).
\end{equation}

It is not {\it a priori} clear that $T$, defined according to \eqref{2.5}, is m-accretive if $T_0$ is. For this purpose we now state a sufficient condition of m-accretivity of $T$: 

\begin{lemma}\lb{l2.2}
In addition to Hypothesis \ref{h2.1} assume that 
\begin{equation}\lb{2.7B}
\dom(A)\supseteq \dom\big(T_0^{1/2}\big), \quad \dom(B)\supseteq \dom\big((T_0^*)^{1/2}\big),
\end{equation}
and suppose that $T_0$ is m-sectorial.  If there exist $\gamma>0$, $\theta\in [0,\pi/2)$, and $0<\delta<1$ such that
\begin{equation}\lb{2.7a}
\big\|A(T_0-z)^{-1/2}\big\|_{\cB(\cH,\cK)} \, 
\big\|\overline{(T_0-z)^{-1/2}B^*} \big\|_{\cB(\cK,\cH)}\leq 1-\delta, 
\quad z\in \bbC\backslash S_{\gamma,\theta},
\end{equation}
then $T$ defined by \eqref{2.5} is m-sectorial.
\end{lemma}
\begin{proof}
It suffices to show that $\sigma(T)$ is contained in a sector and that $T$ is quasi-m-accretive.  One 
observes that
\begin{align}
\sigma(T)\subseteq\sigma(T_0) \cup (\bbC \backslash  \{\zeta\in \bbC \, |\, 1\in \rho(K(\zeta))\}).\lb{2.7b}
\end{align}
On the other hand, one has
\begin{align}
\bbC\backslash S_{\gamma,\theta}&\subseteq \{\zeta\in \bbC \, |\, \|K(\zeta)\|_{\cB(\cK)}<1\} 
\subseteq \{\zeta\in \bbC \, |\, 1\in \rho(K(\zeta))\}.\lb{2.7c}
\end{align}
Thus, 
\begin{equation}\lb{2.7d}
\bbC\backslash S_{\gamma,\theta} \subset \{\zeta\in \bbC \, |\, 1\in \rho(K(\zeta))\},
\end{equation}
and taking complements in $\bbC$, 
\begin{equation}\lb{2.7e}
\bbC\backslash \{\zeta\in \bbC \, |\, 1\in \rho(K(\zeta))\} \subset S_{\gamma,\theta}.
\end{equation}
Therefore, \eqref{2.7b} and \eqref{2.7e} yield 
\begin{equation}\lb{2.7f}
\sigma(T)\subseteq \sigma(T_0)\cup S_{\gamma,\theta}.
\end{equation}
Since $T_0$ is assumed to be m-sectorial, $\sigma(T_0)$ is contained in a sector.  As a result, $\sigma(T_0)\cup S_{\gamma,\theta}$ is contained in some sector, say $S_{\gamma',\theta'}$.  By \eqref{2.7f}, $S_{\gamma',\theta'}$ also contains $\sigma(T)$.

Since $T_0$ is quasi-m-accretive there exists a constant $t>0$ such that $T_0+tI_{\cH}$ is m-accretive and for which
\begin{equation}\lb{2.7g}
S_{\gamma',\theta'}+t\subset \{\zeta\in \bbC \, | \, \Re(\zeta)\geq 0\}.
\end{equation}
Without loss of generality, we may assume that $t> \gamma$.  In this case, $\rho(T_0+tI_{\cH})\cap\rho(T+tI_{\cH})$ covers the left half-plane, $\{\zeta\in \bbC\, |\, \Re(\zeta)<0\}$, and one concludes that 
\begin{align}
&\big\| ((T+(t+z)I_{\cH})^{-1} \big\|_{\cB(\cH)}\no\\
&\quad \leq \big\|(T_0+(t+z)I_{\cH})^{-1} \big\|_{\cB(\cH)} 
+ \big\|\overline{(T_0+(t+z)I_{\cH})^{-1}B^*} \big\|_{\cB(\cK,\cH)}\no\\
&\qquad \times\big\|(I_{\cK}-K(-(t+z)))^{-1} \big\|_{\cB(\cK)} \, 
\big\|A(T_0+(t+z)I_{\cH})^{-1} \big\|_{\cB(\cH,\cK)}\no\\
&\quad \leq  \big\|(T_0+(t+z)I_{\cH})^{-1} \big\|_{\cB(\cH)} 
+ \big\|(T_0+(t+z)I_{\cH})^{-1/2} \big\|_{\cB(\cH)}^2\no\\
&\qquad \times \big\|\overline{(T_0+(t+z)I_{\cH})^{-1/2}B^*} \big\|_{\cB(\cK,\cH)} \, 
\big\|(I_{\cK}-K(-(t+z)))^{-1} \big\|_{\cB(\cK)}\no\\
&\qquad \times \big\|A(T_0+(t+z)I_{\cH})^{-1/2} \big\|_{\cB(\cH,\cK)}\no\\
&\quad \leq C[\Re(z)]^{-1}, \quad z\in \bbC,\; \Re(z)>0,\lb{2.7h}
\end{align}
for some constant $C= C(T_0,\delta)>0$. To obtain \eqref{2.7h}, we have used three estimates; the first two follow from m-accretivity of $T_0$:
\begin{align}\lb{2.7i}
& \big\|(T_0+(t+z)I_{\cH})^{-1} \big\|_{\cB(\cH)} \leq \frac{C_1}{\Re(z)},   \quad 
\big\|(T_0+(t+z)I_{\cH})^{-1/2} \big\|_{\cB(\cH)} \leq \frac{C_2}{[\Re(z)]^{1/2}},    \no \\
& \hspace*{8cm}  z\in \bbC, \; \Re(z)>0.
\end{align}
The third is
\begin{align}
\big\|[I_{\cK}-K(-(t+z))]^{-1}\big\|_{\cB(\cK)}\leq \delta^{-1},\quad z\in \bbC, \; \Re(z)>0,\lb{2.7j}
\end{align}
which follows from 
\begin{align}
\|K(-(t+z))\|_{\cB(\cK)} & \leq \big\|A(T_0+(t+z)I_{\cH})^{-1/2} \big\|_{\cB(\cH,\cK)}     \no \\
& \quad \times \big\|\overline{(T_0+(t+z)I_{\cH})^{-1/2}B^*} \big\|_{\cB(\cK,\cH)}\no\\
& \leq 1-\delta,\quad z\in \bbC, \; \Re(z)>0,\lb{2.7k}
\end{align}
taking the Neumann series for $[I_{\cK}-K(-(t+z))]^{-1}$ into account.  One notes that \eqref{2.7k} holds based on \eqref{2.7a} since $\Re(z)>0$ guarantees $-(t+z)\in \bbC\backslash S_{\gamma,\theta}$.  In light of \eqref{2.7h}, $T+tI_{\cH}$ is m-accretive and hence $T$ is quasi-m-accretive. 
\end{proof}

\begin{hypothesis}\lb{h2.3}
$(i)$ Suppose $T_0$ is m-accretive. \\ 
$(ii)$ Suppose that $A:\dom(A)\rightarrow \cK$, $\dom(A)\subseteq \cH$, is a closed, linear operator 
from $\cH$ to $\cK$, and $B:\dom(B)\rightarrow \cK$, $\dom(B)\subseteq \cH$, is a closed, linear 
operator from $\cH$ to $\cK$ such that 
\begin{equation}\lb{2.9}
\dom(A)\supseteq \dom\big(T_0^{1/2}\big), \quad \dom(B)\supseteq \dom\big((T_0^*)^{1/2}\big).
\end{equation}
$(iii)$  Suppose that there exist constants $R>0$ and $E_0>0$ such that
\begin{align}\lb{2.10}
\begin{split} 
\int_R^{\infty}d\lambda\, \lambda^{-1} & \big\|A(T_0+(\lambda+E)I_{\cH})^{-1/2}\big\|_{\cB(\cH,\cK)}   \\
& \times \big\|\overline{(T_0+(\lambda+E)I_{\cH})^{-1/2}B^{\ast}}\big\|_{\cB(\cK,\cH)}<\infty,   
\quad  E \geq E_0,
\end{split} 
\end{align}
with
\begin{align}\lb{2.11}
\begin{split} 
\lim_{E\rightarrow \infty}\int_R^{\infty}d\lambda\, \lambda^{-1} & 
\big\|A(T_0+(\lambda+E)I_{\cH})^{-1/2}\big\|_{\cB(\cH,\cK)}     \\
& \times \big\|\overline{(T_0+(\lambda+E)I_{\cH})^{-1/2}B^{\ast}}\big\|_{\cB(\cK,\cH)}=0,
\end{split} 
\end{align}
and 
\begin{align}
\lim_{E\rightarrow \infty}\big\|A(T_0+EI_{\cH})^{-1/2}\big\|_{\cB(\cH,\cK)}\big\|\overline{(T_0+EI_{\cH})^{-1/2}B^{\ast}}\big\|_{\cB(\cK,\cH)}=0.\lb{2.25ccc}
\end{align}

\end{hypothesis}

Our principal stability result of square root domains then reads as follows.

\begin{theorem}\lb{t2.4}
Assume Hypotheses \ref{h2.1} and \ref{h2.3}.  If $T$ defined as in \eqref{2.5} is m-accretive, then
\begin{equation}\lb{2.12a}
\dom\big(T^{1/2}\big) = \dom\big(T_0^{1/2}\big).
\end{equation}
\end{theorem}
\begin{proof}
Define
\begin{align}
J(E)&=\int_0^{\infty} d\lambda \, \lambda^{-1/2}(T_0+EI_{\cH})^{1/2}
\overline{(T_0+(E+\lambda)I_{\cH})^{-1}B^*} [I_{\cK}-K(-(E+\lambda))]^{-1}  \no\\
&\hspace*{1.6cm}  \times A(T_0+(E+\lambda)I_{\cH})^{-1}, \quad E \geq E_0.    \lb{2.18} 
\end{align}
The integral in \eqref{2.18} is to be viewed as a norm convergent Bochner integral.  To confirm that the integral is indeed norm convergent, one estimates the norm of the integrand of \eqref{2.18} as follows, 
\begin{align}
&\big\|\lambda^{-1/2}(T_0+EI_{\cH})^{1/2}\overline{(T_0+(E+\lambda)I_{\cH})^{-1}B^*}\no\\
&\qquad \times[I_{\cK}-K(-(E+\lambda))]^{-1}A(T_0+(E+\lambda)I_{\cH})^{-1}\big\|_{\cB(\cH)}\no\\
&\quad =\lambda^{-1/2}\big\|(T_0+EI_{\cH})^{1/2}(T_0+(E+\lambda)I_{\cH})^{-1/2}\overline{(T_0+(E+\lambda)I_{\cH})^{-1/2}B^*}\no\\
&\hspace*{.75cm} \times[I_{\cK}-K(-(E+\lambda))]^{-1}A(T_0+(E+\lambda)I_{\cH})^{-1/2}(T_0+(E+\lambda)I_{\cH})^{-1/2}\big\|_{\cB(\cH)}\no\\
& \quad \leq \lambda^{-1/2}\big\|(T_0+EI_{\cH})^{1/2}(T_0+(E+\lambda)I_{\cH})^{-1/2}
\big\|_{\cB(\cH)}     \no \\ 
& \qquad \times \big\|[I_{\cK}-K(-(E+\lambda))]^{-1}\big\|_{\cB(\cK)} 
\big\|A(T_0+(E+\lambda)I_{\cH})^{-1/2}\big\|_{\cB(\cH,\cK)}     \no\\
&\qquad \times \big\|\overline{(T_0+(E+\lambda)I_{\cH})^{-1/2}B^*}\big\|_{\cB(\cK,\cH)} 
\big\|(T_0+(E+\lambda)I_{\cH})^{-1/2}\big\|_{\cB(\cH)},       \lb{2.19} \\ 
& \hspace*{8.4cm} \lambda>0, \; E \geq E_0.     \no 
\end{align}
By the $H^{\infty}$-calculus for m-accretive operators (cf., e.g., \cite[Sect.\ 7.1.3]{Ha06}, 
\cite[Sect.\ 11]{KW04}), there exists $M_1>0$ for which
\begin{equation}\lb{2.20}
\big\|(T_0+EI_{\cH})^{1/2}(T_0+(E+\lambda)I_{\cH})^{-1/2}\big\|_{\cB(\cH)}\leq M_1, 
\quad  \lambda>0,\; E \geq E_0,
\end{equation}
and by \eqref{2.25ccc}, there exists an $E_1>E_0$ for which
\begin{align}
&\| K(-E)\|_{\cB(\cK)}\no\\
&\quad \leq \big\|A(T_0+EI_{\cH})^{-1/2}\big\|_{\cB(\cH,\cK)}\big\|\overline{(T_0+EI_{\cH})^{-1/2}B^{\ast}}\big\|_{\cB(\cK,\cH)}\leq 1/2,\quad E\geq E_1.
\end{align}
As a result,
\begin{equation}\lb{2.21}
\big\|[I_{\cK}-K(-(E+\lambda))]^{-1}\big\|_{\cB(\cK)}\leq 2, \quad  \lambda>0, \; E \geq E_1.
\end{equation}
Using \eqref{2.20} and \eqref{2.21} the estimate in \eqref{2.19} can be continued as follows
\begin{align}
&\big\|\lambda^{-1/2}(T_0+EI_{\cH})^{1/2}\overline{(T_0+(E+\lambda)I_{\cH})^{-1}B^*}\no\\
&\qquad \times [I_{\cK}-K(-(E+\lambda))]^{-1}A(T_0+(E+\lambda)I_{\cH})^{-1}\big\|_{\cB(\cH)}\no\\
& \quad \leq 2M_1\lambda^{-1/2} \big\|A(T_0+(E+\lambda)I_{\cH})^{-1/2}\big\|_{\cB(\cH,\cK)}\no\\
&\qquad \times \big\|\overline{(T_0+(E+\lambda)I_{\cH})^{-1/2}B^*}\big\|_{\cB(\cK,\cH)} \big\|(T_0+(E+\lambda)I_{\cH})^{-1/2}\big\|_{\cB(\cH)},\lb{2.22}\\
&\hspace*{8.4cm}\lambda>0, \; E \geq E_1.    \no
\end{align}
Again by the $H^{\infty}$-calculus for m-accretive operators one obtains the estimate
\begin{equation}\lb{2.23}
\begin{split}
\big\|(T_0+(E+\lambda)I_{\cH})^{-1/2}\big\|_{\cB(\cH)}\leq
\begin{cases}
E^{-1/2}, & 0<\lambda<R, \\
\lambda^{-1/2}, & R\leq\lambda<\infty, \end{cases} 
 \quad E>0. &
 \end{split}
\end{equation}
Together with \eqref{2.25ccc} and \eqref{2.23}, the estimate \eqref{2.22} yields 
\begin{align}
&\big\|\lambda^{-1/2}(T_0+EI_{\cH})^{1/2}\overline{(T_0+(E+\lambda)I_{\cH})^{-1}B^*}\no\\
&\qquad \times [I_{\cK}-K(-(E+\lambda))]^{-1}A(T_0+(E+\lambda)I_{\cH})^{-1}\big\|_{\cB(\cH)}\no\\
& \quad \leq 2M_1E^{-1/2}\lambda^{-1/2}\chi_{(0,R)}(\lambda) \no\\
&\qquad+2M_1\lambda^{-1}\chi_{(R,\infty)}(\lambda)\big\|A(T_0+(E+\lambda)I_{\cH})^{-1/2}\big\|_{\cB(\cH,\cK)}
\no \\
&\qquad \quad\times \big\|\overline{(T_0+(E+\lambda)I_{\cH})^{-1/2}B^*}\big\|_{\cB(\cK,\cH)},   \quad \lambda>0, \; E \geq E_1.    \lb{2.24} 
\end{align}
By hypothesis (cf.\ \eqref{2.10}), the majorizing function in \eqref{2.24} lies in 
$L^1((0,\infty); d\lambda)$.  Thus, the integral in \eqref{2.18} is norm convergent, and by \eqref{2.24},
\begin{align}
\|J(E)\|_{\cB(\cH)} & \leq 4R^{1/2}M_1E^{-1/2} \no\\
&\quad+ 2M_1\int_R^{\infty}d\lambda \, 
\lambda^{-1}\big\|A(T_0+(\lambda+E)I_{\cH})^{-1/2}\big\|_{\cB(\cH,\cK)}    \no \\
&\qquad \times \big\|\overline{(T_0+(\lambda+E)I_{\cH})^{-1/2}B^{\ast}}\big\|_{\cB(\cK,\cH)}, 
\quad E \geq E_1.    \lb{2.25} 
\end{align}
In particular, \eqref{2.11} and \eqref{2.25} show that 
\begin{equation}\lb{2.26}
\lim_{E\rightarrow \infty}\|J(E)\|_{\cB(\cH)}=0,
\end{equation}
which implies the existence of an $E_2>E_1$ for which 
\begin{equation}\lb{2.27}
\|J(E)\|_{\cB(\cH)}<1, \quad E \geq E_2.
\end{equation}

Returning to square roots of m-accretive operators, one recalls the representation (cf., e.g., \cite[Sect.\ V.3.11]{Ka80})
\begin{equation}\lb{2.28}
(T+EI_{\cH})^{-1/2}=\frac{1}{\pi}\int_0^{\infty}d\lambda \, \lambda^{-1/2}(T+(E+\lambda)I_{\cH})^{-1}, \quad E>0.
\end{equation}
Applying the resolvent equation \eqref{2.6} to the integrand in \eqref{2.28} yields
\begin{align}
&(T+EI_{\cH})^{-1/2} 
= \frac{1}{\pi}\int_0^{\infty}d\lambda \, \lambda^{-1/2}\big[(T_0+(E+\lambda)I_{\cH})^{-1}-\overline{(T_0+(E+\lambda)I_{\cH})^{-1}B^*}     \no \\
&\hspace*{3.9cm}\times [I_{\cK}-K(-(E+\lambda))]^{-1}A(T_0+(E+\lambda)I_{\cH})^{-1} \big],   \lb{2.29} \\
&\hspace*{10.1cm} E \geq E_2.\no
\end{align}
Multiplication of $\overline{(T_0+(E+\lambda)I_{\cH})^{-1}B^*}$ in \eqref{2.29} from the left by the identity 
\begin{equation}\lb{2.30}
I_{\cH}
= \ol{(T_0+EI_{\cH})^{-1/2}(T_0+EI_{\cH})^{1/2}}, \quad E \geq E_2,
\end{equation}
and use of
\begin{equation}\lb{2.31}
(T_0+EI_{\cH})^{-1/2}=\frac{1}{\pi}\int_0^{\infty}d\lambda \, \lambda^{-1/2}(T_0+(E+\lambda)I_{\cH})^{-1}, \quad E>0,
\end{equation}
implies
\begin{equation}\lb{2.32}
(T+EI_{\cH})^{-1/2}=(T_0+EI_{\cH})^{-1/2}\big[I_{\cH}-\pi^{-1}J(E)\big], 
\quad E \geq E_2.
\end{equation}
Since $I_{\cH}-\pi^{-1}J(E)$ is boundedly invertible for $E\geq E_2$ (cf. \eqref{2.27}), \eqref{2.32} proves that 
\begin{equation}\lb{2.33}
\ran\big((T+EI_{\cH})^{-1/2}\big)=\ran\big((T_0+EI_{\cH})^{-1/2}\big), \quad E \geq E_2,
\end{equation}
and consequently, 
\begin{equation}\lb{2.34}
\dom\big((T+EI_{\cH})^{1/2}\big)=\dom\big((T_0+EI_{\cH})^{1/2}\big), \quad E \geq E_2.
\end{equation}
Since
\begin{equation}\lb{2.35}
\begin{split}
\dom\big((T+EI_{\cH})^{1/2}\big)&=\dom\big(T^{1/2}\big),  \\
\dom\big((T_0+EI_{\cH})^{1/2}\big)&=\dom\big(T_0^{1/2}\big),
\end{split}
\quad E>0,
\end{equation}
relation \eqref{2.34} implies \eqref{2.12a}.
\end{proof}

\begin{remark}
The assumption in \eqref{2.25ccc} can be weakened.  In fact, it would suffice to assume that 
for some $0<\delta<1$, 
\begin{equation}\lb{2.25cccc}
\|K(-E)\|_{\cB(\cK)}\leq 1-\delta,\quad E\geq E_0. 
\end{equation}
However, in the applications that will be considered subsequently, one can obtain decay estimates for the individual factors in \eqref{2.25ccc} that immediately yield convergence to zero as $E\rightarrow \infty$.  Moreover, the assumption \eqref{2.25ccc} enters explicitly into the proof to Theorem \ref{t2.12} below.
\end{remark}

Next, we strengthen Hypothesis \ref{h2.3} as follows:

\begin{hypothesis}\lb{h2.5}
In addition to the assumptions in Hypothesis \ref{h2.3} suppose that 
\begin{equation}\lb{2.8}
\dom\big(T_0^{1/2}\big)=\dom\big((T_0^*)^{1/2}\big).
\end{equation}
\end{hypothesis}

\begin{corollary}\lb{c2.5a}
Assume Hypotheses \ref{h2.1} and \ref{h2.5}.  If $T$ defined as in \eqref{2.5} is m-accretive, then
\begin{equation}\lb{2.12}
\dom\big(T^{1/2}\big) = \dom\big(T_0^{1/2}\big) = \dom\big((T_0^*)^{1/2}\big)= \dom\big((T^*)^{1/2}\big).
\end{equation}\end{corollary}
\begin{proof} 
By symmetry with respect to $T_0$ and $T_0^*$ in our current hypotheses, the proof of 
Theorem \ref{t2.4} also implies 
\begin{equation}  \lb{2.37}
\dom\big((T^*)^{1/2}\big)=\dom\big((T_0^*)^{1/2}\big).
\end{equation}
The proof of \eqref{2.37} mirrors that of \eqref{2.12a}, replacing $T_0$ by $T_0^*$ and 
$T$ by $T^*$ (interchanging the roles of $A$ and $B$) throughout \eqref{2.18}--\eqref{2.35}.  Here, one employs the adjoint resolvent equation \eqref{A.1} and the fact that norms are preserved under taking adjoints:
\begin{align}
&\big\|A(T_0+EI_{\cH})^{-1/2}\big\|_{\cB(\cH,\cK)} = \big\|\ol{(T_0^*+EI_{\cH})^{-1/2}A^*}\big\|_{\cB(\cK,\cH)}, \no\\
&\big\|\overline{(T_0+EI_{\cH})^{-1/2}B^{\ast}}\big\|_{\cB(\cK,\cH)} = \big\|B(T_0^*+EI_{\cH})^{-1/2}\big\|_{\cB(\cH,\cK)},\lb{2.83zz}\\
& \big\|\ol{A(T_0+EI_{\cH})^{-1}B^*} \big\|_{\cB(\cK)} = \big\|\ol{B(T_0^*+EI_{\cH})^{-1}A^*} \big\|_{\cB(\cK)},\quad E>0.\no
\end{align}
We omit further details at this point.  The result \eqref{2.12} then follows from \eqref{2.8}, \eqref{2.12a}, and \eqref{2.37}. 
\end{proof}

\begin{corollary}\lb{c2.5}
Assume Hypotheses \ref{h2.1}, Hypothesis \ref{h2.3}\,$(ii)$,\,$(iii)$, equation \eqref{2.8}, and that $T$ defined by \eqref{2.5} is m-accretive. If there exist constants $E_1>1$, $E_2>1$, $C_1>0$, $C_2>0$, $\e>0$, such that
\begin{align}
\big\| A(T_0+EI_{\cH})^{-1/2}\big\|_{\cB(\cH,\cK)}&\leq C_1(\ln (E))^{-1/2+\e}, \quad E \geq E_1,\lb{2.50a}\\
\big\| \overline{(T_0+EI_{\cH})^{-1/2}B^*} \big\|_{\cB(\cK,\cH)}&\leq C_2(\ln (E))^{-1/2+\e}, 
\quad E \geq E_2,       \lb{2.51a}
\end{align}
then \eqref{2.12} holds.
\end{corollary}
\begin{proof}
In view of Corollary \ref{c2.5a}, it suffices to verify \eqref{2.10}, \eqref{2.11}, and \eqref{2.25ccc} for some $R>0$ and some $E_0>0$.  \eqref{2.50a} and \eqref{2.51a} yield 
\begin{align}
&\lambda^{-1}\big\| A(T_0+(\lambda+E)I_{\cH})^{-1/2}\big\|_{\cB(\cH,\cK)}\big\| 
\overline{(T_0+(\lambda+E)I_{\cH})^{-1/2}B^*} \big\|_{\cB(\cK,\cH)}\no\\
&\quad \leq \lambda^{-1}[\ln(\lambda+E)]^{-(1+2\e)}\lb{2.52a}\\
&\quad \leq \lambda^{-1}[\ln(\lambda)]^{-(1+2\e)}, \quad E \geq \max \{E_1,E_2\}.       \lb{2.53a}
\end{align}
Since \eqref{2.53a} belongs to $L^1((2,\infty);d\lambda)$, relation \eqref{2.10} holds with $R=2$ and 
$E\geq \max\{E_1,E_2\}$.  Moreover, by \eqref{2.52a},
\begin{align}
&\int_2^{\infty} d\lambda \ \lambda^{-1}\big\| A(T_0+(\lambda+E)I_{\cH})^{-1/2}\big\|_{\cB(\cH,\cK)}\big\| \overline{(T_0+(\lambda+E)I_{\cH})^{-1/2}B^*} \big\|_{\cB(\cK,\cH)}\no\\
&\quad \leq \int_2^{\infty} d\lambda \  \lambda^{-1}[\ln(\lambda+E)]^{-(1+2\e)}, 
\quad E \geq \max \{E_1.E_2\}.       \lb{2.54a}
\end{align}
Dominated convergence (applied to the r.h.s. of \eqref{2.54a}) then implies \eqref{2.11} (with $R=2$).  Finally, \eqref{2.50a}, \eqref{2.51a} and the simple estimate
\begin{equation}
\|K(-E)\|_{\cB(\cK)}\leq \big\|A(T_0+EI_{\cH})^{-1/2} \big\|_{\cB(\cH,\cK)} \big\|\ol{(T_0+EI_{\cH})^{-1/2}B^*} \big\|_{\cB(\cK,\cH)},\quad E>0,
\end{equation}  
imply the existence of a constant $E_3>1$ for which
\begin{equation}
\|K(-E)\|_{\cB(\cK)}\leq 1/2,\quad E\geq E_3.
\end{equation}
Hence, the assumptions of Hypothesis \ref{h2.3} hold with $R=2, E_0=\max\{E_1,E_2,E_3 \}$.
\end{proof}

The same strategy of proof also permits one to discuss fractional powers different from $1/2$  
(we also note the weaker assumption \eqref{2.61z}).

\begin{theorem}\lb{t2.8z}
Assume Hypothesis \ref{h2.1} and let $\alpha \in (0,1)$.  Suppose that $T_0$ is m-accretive and  
\begin{equation}\lb{2.57z}
\dom\big(T_0^{\alpha} \big)=\dom\big((T_0^*)^{\alpha} \big).
\end{equation}
In addition, assume that $A:\dom(A)\rightarrow \cK$, $\dom(A)\subseteq \cH$, is a closed, linear 
operator from $\cH$ to $\cK$, and $B:\dom(B)\rightarrow \cK$, $\dom(B)\subseteq \cH$, is a closed, 
linear operator from $\cH$ to $\cK$ such that
\begin{equation}\lb{2.58z}
\dom(A)\supseteq \dom\big(T_0^{1-\alpha}\big),\quad \dom(B)\supseteq \dom\big((T_0^*)^{1-\alpha}\big).
\end{equation}
If there exist constants $R>0$, $E_0>0$ for which
\begin{align}
&\int_R^{\infty} d\lambda\, \lambda^{-2\alpha}\big\|\ol{(T_0+(E+\lambda)I_{\cH})^{\alpha-1}B^*} \big\|_{\cB(\cH)}\no\\
& \hspace*{1.9cm} \times\big\|A(T_0+(E+\lambda)I_{\cH})^{\alpha-1} \big\|_{\cB(\cH)}<\infty,\quad E\geq E_0,\lb{2.59z}
\end{align}
with
\begin{align}
&\lim_{E\rightarrow \infty}\int_R^{\infty} d\lambda\, \lambda^{-2\alpha}\big\|\ol{(T_0+(E+\lambda)I_{\cH})^{\alpha-1}B^*} \big\|_{\cB(\cH)}\no\\
& \hspace*{2.7cm} \times\big\|A(T_0+(E+\lambda)I_{\cH})^{\alpha-1} \big\|_{\cB(\cH)}=0,\lb{2.60z}
\end{align}
and
\begin{equation}\lb{2.61z}
\|K(-E)\|_{\cB(\cK)}\leq 1-\delta,\quad E\geq E_0,
\end{equation}
for some $0<\delta<1$, then 
\begin{equation}\lb{2.62z}
\dom\big(T^{\alpha} \big) = \dom\big(T_0^{\alpha} \big) =\dom\big((T_0^*)^{\alpha} \big)
=\dom\big((T^*)^{\alpha} \big).
\end{equation}
\end{theorem}
\begin{proof}
We begin by showing the first equality in \eqref{2.62z}.  To this end, define
\begin{align}
J_{\alpha}(E)&=\int_0^{\infty}d\lambda\, \lambda^{-\alpha} (T_0+EI_{\cH})^{\alpha}\ol{(T_0+(E+\lambda)I_{\cH})^{-1}B^*}\no\\
& \hspace*{1.6cm} 
\times[I_{\cK}-K(-(E+\lambda))]^{-1}A(T_0+(E+\lambda)I_{\cH})^{-1},\quad E>E_0.\lb{2.66z}
\end{align}
The integral in \eqref{2.66z} is to be viewed as a norm-convergent Bochner integral.  To actually show the integral is norm-convergent, one estimates the norm of the integrand in \eqref{2.66z} as follows:
\begin{align}
&\big\|\lambda^{-\alpha} (T_0+EI_{\cH})^{\alpha}\ol{(T_0+(E+\lambda)I_{\cH})^{-1}B^*}[I_{\cK}-K(-(E+\lambda))]^{-1}\no\\
&\quad \times A(T_0+(E+\lambda)I_{\cH})^{-1}\big\|_{\cB(\cH)}\no\\
&\quad =\big\|\lambda^{-\alpha}(T_0+EI_{\cH})^{\alpha} (T_0+(E+\lambda)I_{\cH})^{-\alpha}\ol{(T_0+(E+\lambda)I_{\cH})^{\alpha-1}B^*}\no\\
&\qquad \times[I_{\cK}-K(-(E+\lambda))]^{-1}A(T_0+(E+\lambda)I_{\cH})^{\alpha-1}(T_0+(E+\lambda)I_{\cH})^{-\alpha} \big\|_{\cB(\cH)}\no\\
&\quad \leq \lambda^{-\alpha} \big\|(T_0+EI_{\cH})^{\alpha} (T_0+(E+\lambda)I_{\cH})^{-\alpha} \big\|_{\cB(\cH)}\big\|[I_{\cK}-K(-(E+\lambda))]^{-1} \big\|_{\cB(\cK)}\no\\
&\qquad \times \big\|\ol{(T_0+(E+\lambda)I_{\cH})^{\alpha-1}B^*} \big\|_{\cB(\cK,\cH)}\big\|A(T_0+(E+\lambda)I_{\cH})^{\alpha-1} \big\|_{\cB(\cH,\cK)}\no\\
&\qquad \times \big\|(T_0+(E+\lambda)I_{\cH})^{-\alpha} \big\|_{\cB(\cH)},\quad \lambda>0,\, E>0.\lb{2.67z}
\end{align}
The $H^{\infty}$-calculus for m-accretive operators (cf., e.g., \cite[Sect.~7.1.3]{Ha06}, \cite[Sect.~11]{KW04}), implies the existence of a constant $M_1>0$ for which
\begin{align}
\big\|(T_0+EI_{\cH})^{\alpha} (T_0+(E+\lambda)I_{\cH})^{-\alpha} \big\|_{\cB(\cH)}\leq M_1,\quad \lambda>0,\, E\geq E_0.\lb{2.68z}
\end{align}
Moreover,
\begin{align}
\big\|[I_{\cK}-K(-(E+\lambda))]^{-1} \big\|_{\cB(\cK)}\leq \delta^{-1},\quad \lambda>0,\, E\geq E_0.\lb{2.69z}
\end{align}
Using \eqref{2.68z} and \eqref{2.69z}, the estimate in \eqref{2.67z} may be continued, and the norm of the integrand in \eqref{2.66z} is bounded by
\begin{align}
&\delta^{-1}M_1\big\|\ol{(T_0+(E+\lambda)I_{\cH})^{\alpha-1}B^*} \big\|_{\cB(\cK,\cH)}\big\|A(T_0+(E+\lambda)I_{\cH})^{\alpha-1} \big\|_{\cB(\cH,\cK)}\no\\
&\quad \times\big\|(T_0+(E+\lambda)I_{\cH})^{-\alpha} \big\|_{\cB(\cH)},\quad \lambda>0,\, E\geq E_0.\lb{2.70z}
\end{align}
Again, the $H^{\infty}$-calculus for m-accretive operators implies
\begin{align}
\begin{split}
\big\|(T_0+(E+\lambda)I_{\cH})^{-\alpha}\big\|_{\cB(\cH)}\leq
\begin{cases}
E^{-\alpha}, & 0<\lambda<R, \\
\lambda^{-\alpha}, & R\leq\lambda<\infty, \end{cases} 
 \quad E>0. &
 \end{split}\lb{2.71z}
\end{align}
Inserting the estimate in \eqref{2.71z} into \eqref{2.70z} implies
\begin{align}
&\big\|\lambda^{-\alpha} (T_0+EI_{\cH})^{\alpha}\ol{(T_0+(E+\lambda)I_{\cH})^{-1}B^*}[I_{\cK}-K(-(E+\lambda))]^{-1}\no\\
&\qquad \times A(T_0+(E+\lambda)I_{\cH})^{-1}\big\|_{\cB(\cH)}\no\\
&\quad \leq M_1\delta^{-1}E^{-\alpha}\lambda^{-\alpha}\chi_{(0,R)}(\lambda)\no\\
&\qquad +M_1\delta^{-1}\lambda^{-2\alpha}\chi_{(R,\infty)}(\lambda)\big\|\ol{(T_0+(E+\lambda)I_{\cH})^{\alpha-1}B^*} \big\|_{\cB(\cK,\cH)}\no\\
&\qquad \times\big\|A(T_0+(E+\lambda)I_{\cH})^{\alpha-1} \big\|_{\cB(\cH,\cK)},\quad \lambda>0,\, E\geq E_0.\lb{2.72z}
\end{align}
By hypothesis, the majorizing term in \eqref{2.72z} lies in $L^1((0,\infty);d\lambda)$.  Thus, the integral in \eqref{2.66z} is norm convergent and 
\begin{align}
\|J_{\alpha}(E)\|_{\cB(\cH)}&\leq \frac{M_1}{1-\alpha}\delta^{-1}R^{1-\alpha}E^{-\alpha}\no\\
&\quad  + M_1\delta^{-1}\int_R^{\infty}d\lambda\, \lambda^{-2\alpha}\big\|A(T_0+EI_{\cH})^{\alpha-1}\big\|_{\cB(\cH,\cK)}       \lb{2.73z} \\
&\hspace*{3cm} \times\big\|\overline{(T_0+EI_{\cH})^{-\alpha}B^{\ast}}\big\|_{\cB(\cK,\cH)},\quad E\geq E_0.     \no 
\end{align}
In particular, \eqref{2.60z} and \eqref{2.73z} show that
\begin{equation}\lb{2.74z}
\lim_{E\rightarrow 0}\big\|J_{\alpha}(E)\|_{\cB(\cH)}=0,
\end{equation}
hence there exists a constant $E_1>0$ for which
\begin{equation}\lb{2.75z}
\big\|J_{\alpha}(E)\|_{\cB(\cH)}<1,\quad E\geq E_1.
\end{equation}
Returning to fractional powers of m-accretive operators, we recall the representation for the $\alpha$-th power of the resolvent of an m-accretive operator (cf., e.g., \cite[Remark V.3.50]{Ka80}),
\begin{equation}\lb{2.76z}
(T_0+EI_{\cH})^{-\alpha}=\frac{\sin(\pi \alpha)}{\pi}\int_0^{\infty} d\lambda\, \lambda^{-\alpha}(T_0+(E+\lambda)I_{\cH})^{-1},
\quad E>0,
\end{equation}
Insing Kato's resolvent identity (cf.\ \eqref{2.6}) into \eqref{2.76z} yields,
\begin{align}
(T+EI_{\cH})^{-\alpha} &= \frac{\sin(\pi \alpha)}{\pi}\int_0^{\infty}d\lambda\, \lambda^{-\alpha} \big[(T_0+(E+\lambda)I_{\cH})^{-1}\nonumber\\
& \hspace*{2.8cm}
-\overline{(T_0+(E+\lambda)I_{\cH})^{-1}B^*}[I_{\cK}-K(-(E+\lambda))]^{-1}\no\\
& \hspace*{2.8cm} \quad \times A(T_0+(E+\lambda)I_{\cH})^{-1} \big],\no\\
&=\frac{\sin(\pi \alpha)}{\pi}\int_0^{\infty} d\lambda\, \lambda^{-\alpha} (T_0+(E+\lambda)I_{\cH})^{-1}\no\\
&\quad  - \frac{\sin(\pi \alpha)}{\pi}\int_0^{\infty}d\lambda\, \lambda^{-\alpha} \ol{(T_0+(E+\lambda)I_{\cH})^{-1}B^*}\no\\
&\hspace*{2.7cm} \times[I_{\cK}-K(-(E+\lambda))]^{-1}A(T_0+(E+\lambda)I_{\cH})^{-1}\no\\
&= (T_0+EI_{\cH})^{-\alpha}\bigg[I_{\cH} -\frac{\sin(\pi \alpha)}{\pi}J_{\alpha}(E)\bigg],\quad  E \geq E_0.\lb{2.77z}
\end{align}
Since the factor $[\cdots]$ containing $J_{\alpha}(E)$ is boundedly invertible for $E\geq E_1$, \eqref{2.77z} implies
\begin{align}
\ran\big((T+EI_{\cH})^{-\alpha}\big) = \ran\big((T_0+EI_{\cH})^{-\alpha} \big),\quad E\geq \max\{E_0,E_1\}.\lb{2.78z}
\end{align}
Consequently,
\begin{equation}\lb{2.79z}
\dom\big((T+EI_{\cH})^{\alpha} \big) = \dom\big((T_0+EI_{\cH})^{\alpha} \big), \quad E\geq \max\{E_0,E_1\},
\end{equation}
which implies
\begin{equation}\lb{2.80z}
\dom\big(T^{\alpha} \big) = \dom\big(T_0^{\alpha} \big),
\end{equation}
by applying 
\begin{equation}\lb{2.81z}
\begin{split}
\dom\big((T+EI_{\cH})^{1/2}\big)&=\dom\big(T^{1/2}\big),  \\
\dom\big((T_0+EI_{\cH})^{1/2}\big)&=\dom\big(T_0^{1/2}\big),
\end{split}
\quad E>0.
\end{equation}
By symmetry with respect to $T_0$ and $T_0^*$ in our current hypotheses, the proof of 
\eqref{2.80z} also implies 
\begin{equation}\lb{2.82z}
\dom\big((T^*)^{\alpha}\big)=\dom\big((T_0^*)^{\alpha}\big).
\end{equation}
The proof of \eqref{2.82z} mirrors that of \eqref{2.80z}, replacing $T_0$ by $T_0^*$ and 
$T$ by $T^*$ (interchanging the roles of $A$ and $B$) throughout \eqref{2.66z}--\eqref{2.81z}.  Here, one employs the adjoint resolvent equation \eqref{A.1} and the fact that norms are preserved under taking adjoints:
\begin{align}
&\big\|A(T_0+EI_{\cH})^{\alpha-1}\big\|_{\cB(\cH,\cK)} = \big\|\ol{(T_0^*+EI_{\cH})^{\alpha-1}A^*}\big\|_{\cB(\cK,\cH)}, \no\\
&\big\|\overline{(T_0+EI_{\cH})^{1-\alpha}B^{\ast}}\big\|_{\cB(\cK,\cH)} = \big\|B(T_0^*+EI_{\cH})^{1-\alpha}\big\|_{\cB(\cH,\cK)},\lb{2.83z}\\
& \big\|\ol{A(T_0+EI_{\cH})^{-1}B^*} \big\|_{\cB(\cK)} = \big\|\ol{B(T_0^*+EI_{\cH})^{-1}A^*} \big\|_{\cB(\cK)},\quad E>0.\no
\end{align}
We omit further details at this point.
\end{proof}

\begin{remark}\lb{r2.7}
The assumption in Theorems \ref{t2.4} and \ref{t2.8z}, and Corollaries \ref{c2.5a} 
and \ref{c2.5}, that $T_0$ and $T$ 
are m-accretive may be replaced by the assumption that $T_0$ and $T$ are {\it m-sectorial}. 
If $T_0$ and $T$ are m-sectorial, $T_0$ and $T$ are {\it quasi-m-accretive}:  there exists a 
$t>0$ such that $T_0+tI_{\cH}$ and $T+tI_{\cH}$ are m-accretive.  Since the assumptions 
\eqref{2.10} and \eqref{2.11} on $T_0$ remain valid with $T_0$ replaced by $T_0+tI_{\cH}$, 
the proof of Theorem \ref{t2.4} can be carried out with $T_0$ (resp., $T$) replaced by 
$T_0+tI_{\cH}$ (resp., $T+tI_{\cH}$) to show that
\begin{equation}\lb{2.50}
\dom\big((T+tI_{\cH})^{1/2}\big) = \dom\big((T_0+tI_{\cH})^{1/2}\big) = 
\dom\big(((T+tI_{\cH})^*)^{1/2}\big).
\end{equation}
\end{remark}

\begin{lemma}\lb{l2.8}
Let $S$ and $T$ be m-accretive operators in $\cH$ and suppose that 
\begin{equation}\lb{2.52}
\dom\big(T^{1/2}\big)\subseteq \dom\big(S^{1/2}\big). 
\end{equation}
Then there exists a constant $C>0$ such that
\begin{equation}\lb{2.53}
\sup_{E\geq 1}\big\|(S+EI_{\cH})^{1/2}(T+EI_{\cH})^{-1/2}\big\|_{\cB(\cH)}\leq C.
\end{equation}
\end{lemma}
\begin{proof}
Employing the $H^{\infty}$-calculus for m-accretive operators (cf.\ \cite[Sect.\ 7.1.3]{Ha06}, 
\cite[Sect.\ 11]{KW04}) once again, 
\begin{equation}\lb{2.70bbb}
\big\|(S^*)^{1/2}(S^*+EI_{\cH})^{-1/2}\big\|_{\cB(\cH)} \leq \sup_{\substack{\lambda\in \bbC\\ \Re(\lambda)>0}}\frac{|\lambda|^{1/2}}{|\lambda + E|^{1/2}}, \quad E>0.
\end{equation}
In light of the trivial estimate,
\begin{align}
|\lambda + E| &= [(\Re(\lambda)+E)^2 + \Im(\lambda)^2]^{1/2} \no\\
&\geq [(\Re(\lambda)+1)^2 + \Im(\lambda)^2]^{1/2} \lb{2.71bbb}\\
&= |\lambda + 1|,\quad \lambda\in \bbC,\, \Re(\lambda)\geq 0,\, E\geq 1,\no
\end{align}
\eqref{2.70bbb} implies
\begin{equation}\lb{2.72bbb}
\big\|(S^*)^{1/2}(S+EI_{\cH})^{-1/2}\big\|_{\cB(\cH)} \leq C_1,\quad E\geq 1,
\end{equation}
where
\begin{equation}\lb{2.73bbb}
C_1:=\sup_{\substack{\lambda\in \bbC\\ \Re(\lambda)>0}}\frac{|\lambda|^{1/2}}{|\lambda + 1|^{1/2}}<\infty.
\end{equation}
Applying the $H^{\infty}$-calculus once more, 
\begin{align}
\big\|E^{1/2}(A+EI_{\cH})^{-1/2} \big\|_{\cB(\cH)}\leq \sup_{\substack{\lambda \in \bbC \\ \Re(\lambda)>0}}\frac{|E|^{1/2}}{|\lambda + E|^{1/2}}\quad E>0,\, A\in \{S^*,T\}.\lb{2.74bbb}
\end{align}
Then
\begin{equation}\lb{2.75bbb}
\frac{|E|^{1/2}}{|\lambda + E|^{1/2}}\leq \frac{|E|^{1/2}}{|\Re(\lambda)+E|^{1/2}}\leq \frac{|E|^{1/2}}{|E|^{1/2}}\leq 1,\quad \lambda \in \bbC,\, \Re(\lambda)>0,\, E>0,
\end{equation}
implies
\begin{equation}\lb{2.76bbb}
\big\|E^{1/2}(A+EI_{\cH})^{-1/2} \big\|_{\cB(\cH)}\leq 1,\quad E> 0,\, A\in \{S^*,T\}.
\end{equation}
Now consider the operator $S^{1/2}(T+I_{\cH})^{-1/2}$ which is defined on all of $\cH$ 
by the well-known fact that
\begin{equation}\lb{2.77bbb}
\dom\big((T+\varepsilon I_{\cH})^{1/2} \big)=\dom\big(T^{1/2}\big),\quad \varepsilon>0.
\end{equation}
In addition, $S^{1/2}(T+I_{\cH})^{-1/2}$ is closed, so it is bounded by the closed graph theorem:
\begin{equation}\lb{2.78bbb}
\big\|S^{1/2}(T+I_{\cH})^{-1/2} \big\|_{\cB(\cH)}\leq M,
\end{equation}
for a suitable constant $M>0$.  One then obtains the following estimate:
\begin{align}
&\big|\big(f,(S+EI_{\cH})^{1/2}(T+EI_{\cH})^{-1/2}g\big)_{\cH}\big|\no\\
&\quad=\big|\big(f,(S+EI_{\cH})^{-1/2}(S+EI_{\cH})(T+EI_{\cH})^{-1/2}g\big)_{\cH}\big|\no\\
&\quad=\big|\big((S^*+EI_{\cH})^{-1/2}f,(S+EI_{\cH})(T+EI_{\cH})^{-1/2}g\big)_{\cH}\big|\no\\
&\quad\leq \big|\big((S^{\ast}+EI_{\cH})^{-1/2}f,S(T+EI_{\cH})^{-1/2}g\big)_{\cH}\big|\no\\
&\qquad +\big|\big((S^*+EI_{\cH})^{-1/2}f,E(T+EI_{\cH})^{-1/2}g\big)_{\cH}\big|\no\\
&\quad=\big|\big((S^*)^{1/2}(S^*+EI_{\cH})^{-1/2}f,S^{1/2}(T+EI_{\cH})^{-1/2}g\big)_{\cH}\big|\no\\
&\qquad +\big|\big(E^{1/2}(S^*+EI_{\cH})^{-1/2}f,E^{1/2}(T+EI_{\cH})^{-1/2}g\big)_{\cH}\big|\no\\
&\quad\leq \big\|(S^*)^{1/2}(S^*+EI_{\cH})^{-1/2}\big\|_{\cB(\cH)}\big\|S^{1/2}(T+EI_{\cH})^{-1/2}\big\|_{\cB(\cH)}\no\\
&\qquad +\big\|E^{1/2}(S^*+EI_{\cH})^{-1/2}\big\|_{\cB(\cH)}\big\|E^{1/2}(T+EI_{\cH})^{-1/2}\big\|_{\cB(\cH)}, \no \\ 
& \quad\leq  C_1 \big\|S^{1/2}(T+I_{\cH})^{-1/2} \big\|_{\cB(\cH)} \big\|(T+I_{\cH})^{1/2}(T+EI_{\cH})^{-1/2} \big\|_{\cB(\cH)} + 1\no\\
&\quad \leq C_1M+1,\quad f,g\in \cH,\, \|f\|_{\cH}=\|g\|_{\cH}=1,\, E\geq 1,\lb{2.79bbb}
\end{align}
where we used the $H^{\infty}$-calculus once more to estimate
\begin{equation}\lb{2.80bbb}
\big\|(T+I_{\cH})^{1/2}(T+EI_{\cH})^{-1/2} \big\|_{\cB(\cH)}\leq \sup_{\substack{\lambda \in \bbC \\ \Re(\lambda)>0}}\frac{|\lambda +1|^{1/2}}{|\lambda + E|^{1/2}}\leq 1,\quad E\geq 1.
\end{equation}
Thus, \eqref{2.79bbb} implies \eqref{2.53} with $C=C_1M+1$.
\end{proof}

\begin{corollary}\lb{c2.9}
Let $S$ be a densely defined, closed, linear operator in $\cH$. Suppose that 
$T_j$, $j=1,2$, are m-accretive operators in $\cH$, with
\begin{equation}\lb{a4.26.1}
\dom(T_1^{1/2})=\dom(T_2^{1/2})\subseteq \dom(S).
\end{equation}
Then there exist constants $0<c\leq C<\infty$ such that
\begin{align}
\begin{split} 
& c\big\|S(T_1+EI_{\cH})^{-1/2}\big\|_{\cB(\cH)} 
\leq \big\|S(T_2+EI_{\cH})^{-1/2}\big\|_{\cB(\cH)}    \\
& \quad \leq C\big\|S(T_1+EI_{\cH})^{-1/2}\big\|_{\cB(\cH)}, \quad E\geq 1.     \lb{a4.26.2}
\end{split} 
\end{align}
\end{corollary}
\begin{proof}
By symmetry (in $T_1$ and $T_2$), it suffices to show the existence of a $C>0$ for which the second inequality in \eqref{a4.26.2} holds.  Since $\dom(T_1^{1/2})=\dom(T_2^{1/2})$, Lemma \ref{l2.8} provides for the existence of a constant $C$ such that
\begin{equation}\lb{a4.26.3}
\big\|(T_1+EI_{\cH})^{1/2}(T_2+EI_{\cH})^{-1/2} \big\|_{\cB(\cH)}\leq C, \quad E\geq 1.
\end{equation}
Hence, one estimates
\begin{align}
\big\|S(T_2+EI_{\cH})^{-1/2} \big\|_{\cB(\cH)} 
& \leq \big\|S(T_1+EI_{\cH})^{-1/2} \big\|_{\cB(\cH)}    \no \\ 
& \quad \times \big\|(T_1+EI_{\cH})^{1/2}(T_2+EI_{\cH})^{-1/2} \big\|_{\cB(\cH)}\no\\
& \leq C\big\|S(T_1+EI_{\cH})^{-1/2} \big\|_{\cB(\cH)}, \quad E\geq 1.\lb{a4.26.4}
\end{align}
\end{proof}

\begin{lemma}\lb{l2.10}
Let $S$ and $T$ denote non-negative self-adjoint operators in $\cH$ with $\dom\big(T^{1/2}\big)\subseteq \dom\big(S^{1/2}\big)$.  Then items $(i)$ and $(ii)$ below are equivalent: \\
$(i)$ There exist constants $M_1>0$, $E_0 \geq 1$, and $\beta_1>0$ such that
\begin{equation}\lb{4.24.1}
\big\|S^{1/2}(T+EI_{\cH})^{-1/2} \big\|_{\cB(\cH)}\leq M_1 E^{-\beta_1}, \quad E \geq E_0.
\end{equation}
$(ii)$  There exist constants $M_2>0$, $\varepsilon_0>0$, and $\beta_2>0$ such that
\begin{align}
\big\|S^{1/2}f\big\|_{\cH}\leq \varepsilon \big\|T^{1/2}f\big\|_{\cH}+M_2\varepsilon^{-\beta_2}\|f\|_{\cH},\quad 0<\varepsilon<\varepsilon_0,\; f\in \dom\big(T^{1/2}\big).\lb{4.24.2}
\end{align}
\end{lemma}
\begin{proof}
Assume that item $(i)$ holds.  Applying \eqref{4.24.1}, one estimates
\begin{align}
&\big\|S^{1/2}f\big\|_{\cH}^2 
=\big\|S^{1/2}(T+EI_{\cH})^{-1/2}(T+EI_{\cH})^{1/2}f \big\|_{\cH}^2\no\\
&\quad \leq \big\|S^{1/2}(T+EI_{\cH})^{-1/2} \big\|_{\cB(\cH)}^2\big\|(T+EI_{\cH})^{1/2}f\big\|_{\cH}^2\no\\
&\quad \leq M_1^2E^{-2\beta_1}\Big[\big\|T^{1/2}f \big\|_{\cH}^2+E\|f\|_{\cH}^2 \Big]\no\\
&\quad = M_1^2E^{-2\beta_1}\big\|T^{1/2}f\big\|_{\cH}^2+M_1^2E^{1-2\beta_1}\|f\|_{\cH}^2, \quad 
E \geq E_0, \; f\in \dom\big(T^{1/2}\big).    \lb{3.10} 
\end{align}
Taking square roots throughout \eqref{3.10} and using $(a+b)^{1/2}\leq a^{1/2}+b^{1/2}$, $a,b\geq 0$, 
yields
\begin{equation}\lb{3.11}
\big\|S^{1/2}f\big\|_{\cH}\leq M_1E^{-\beta_1}\big\|T^{1/2}f\big\|_{\cH}+M_1E^{\frac{1-2\beta_1}{2}}\|f\|_{\cH},\quad
E \geq E_0, \; f\in \dom\big(T^{1/2}\big).
\end{equation}
Since $E$ may be taken arbitrarily large, \eqref{3.11} shows that $S$ is infinitesimally form bounded with respect to $T$.  Indeed, setting $\e=M_1E^{-\beta_1}$, \eqref{3.11} may be expressed as 
\begin{equation}\lb{3.12}
\big\|S^{1/2}f\big\|_{\cH}\leq \e\big\|T^{1/2}f\big\|_{\cH}+M_1^{\frac{1}{2\beta_1}}\e^{\frac{2\beta_1-1}{2\beta_1}}\|f\|_{\cH},\quad
0<\e<M_1E_0^{-\beta_1}, \; f\in \dom\big(T^{1/2}\big).
\end{equation}
If $0<\beta_1<1/2$, then \eqref{3.12} implies \eqref{4.24.2}, choosing
\begin{equation}\lb{4.24.3}
\e_0=M_1E_0^{-\beta_1}, \quad M_2=M_1^{\frac{1}{2\beta_1}}, \quad 
\beta_2=\frac{1-2\beta_1}{2\beta_1}.
\end{equation}
If $\beta_1\geq 1/2$, then one further estimates \eqref{3.12} by
\begin{align}
\big\|S^{1/2}f\big\|_{\cH}\leq \e\big\|T^{1/2}f\big\|_{\cH}+M_1^{\frac{1}{2\beta_1}}\e^{-1}\|f\|_{\cH},\quad 0<\e<\min\{1,M_1E_0^{-\beta_1}\},&\lb{4.24.4}\\
\; f\in \dom\big(T^{1/2}\big),&\no
\end{align}
so that \eqref{4.24.2} and hence item $(ii)$ holds with the choice for $M_2$ given in \eqref{4.24.3} and
\begin{equation}\lb{4.24.5}
\varepsilon_0=\min\{1,M_1E_0^{-\beta_1}\}, \quad \beta_2=1.
\end{equation} 

Next, assume that item $(ii)$ holds.  Choosing $f=(T+EI_{\cH})^{-1/2}g$, $g\in \cH$, $E>0$, in \eqref{4.24.2}, one estimates using the functional calculus,
\begin{align}
&\big\|S^{1/2}(T+EI_{\cH})^{-1/2}g \big\|_{\cH}  \no \\ 
&\quad \leq \Big[\e\big\|T^{1/2}(T+EI_{\cH})^{1/2}\big\|_{\cB(\cH)}+M_2\e^{-\beta_2}\big\|(T+EI_{\cH})^{-1/2}\big\|_{\cB(\cH)}\Big] \|g\|_{\cH},\lb{4.5}\\
&\quad \leq \big[\e+M_2\e^{-\beta_2}E^{-1/2}\big]\|g\|_{\cH},\quad g\in \cH, \; 0<\e<\e_0, \; E>0,\no
\end{align}
which shows that 
\begin{equation}\lb{4.24.6}
\big\|S^{1/2}(T+EI_{\cH})^{-1/2}\big\|_{\cB(\cH)}\leq\e+M_2\e^{-\beta_2}E^{-1/2}, \quad 0<\e<\e_0, \; E>0.
\end{equation}
Choosing $\e=E^{-\alpha}$ for a fixed $\alpha$ satisfying 
\begin{equation}\lb{4.6}
0<\alpha(1+\beta_2)<1/2
\end{equation}
 and $E>\e_0^{-1/\alpha}$ in \eqref{4.24.6} yields
\begin{equation}\lb{4.7}
\big\|S^{1/2}(T+E)^{-1/2} \big\|_{\cB(\cH)}\leq E^{-\alpha}+M_2E^{\alpha\beta_2-1/2}, \quad E>\e_0^{-1/\alpha}.
\end{equation}
Condition \eqref{4.6} guarantees 
\begin{equation}\lb{4.8}
E^{\alpha\beta_2-1/2}\leq E^{-\alpha}, \quad E\geq 1.
\end{equation}
Using \eqref{4.8} to continue the estimate in \eqref{4.7}, one obtains
\begin{equation}\lb{4.24.7}
\big\|S^{1/2}(T+E)^{-1/2} \big\|_{\cB(\cH)}\leq (1+M_2)E^{-\alpha}, \quad 
E > \max\big\{1,\e_0^{-1/\alpha}\big\}.
\end{equation}
Thus, \eqref{4.24.1} follows from \eqref{4.24.7}, choosing
\begin{equation}\lb{4.24.8}
M_1=(1+M_2),\quad E_0 = \max \big\{1,\e_0^{-1/\alpha}\big\}, \quad \beta_1=\alpha.
\end{equation}
\end{proof}

Finally, we turn to another stability result which will be applied to concrete PDE situations in 
Section \ref{s4}.

\begin{hypothesis}\lb{h2.11}
$(i)$  Suppose that $T$, is an m-accretive operator in $\cH$ with 
\begin{equation}\lb{2.62}
\dom\big(T^{1/2}\big)=\dom\big((T^*)^{1/2}\big),
\end{equation}
and $A:\dom(A)\rightarrow \cK$, $\dom(A)\subseteq \cH$, is a densely defined, closed, linear operator from $\cH$ to $\cK$, and $B:\dom(B)\rightarrow \cK$, $\dom(B)\subseteq \cH$, is a densely defined, closed, linear operator from $\cH$ to $\cK$ such that
\begin{equation}\lb{2.63}
\dom(A)\supseteq \dom\big(T^{1/2}\big), \quad \dom(B)\supseteq \dom\big(T^{1/2}\big).
\end{equation}
$(ii)$ For some $($and hence for all\,$)$ $z\in \rho(T)$, the operator $-A(T-zI_{\cH})^{-1}B^{\ast}$, defined on $\dom(B^*)$, has a bounded extension in $\cK$, denoted by $K(z)$,
\begin{equation}\lb{2.64}
K(z)=-\overline{A(T-zI_{\cH})^{-1}B^*}\in \cB(\cK).
\end{equation}
$(iii)$  $1\in \rho(K(z_0))$ for some $z_0\in \rho(T)$. \\ 
$(iv)$  Define the densely defined, closed, linear operator $T_1$ in $\cH$ by
\begin{align}\lb{2.65}
& (T_1-zI_{\cH})^{-1}=(T-zI_{\cH})^{-1}-\overline{(T-zI_{\cH})^{-1}B^*}[I_{\cK}-K(z)]^{-1}A(T-zI_{\cH})^{-1},    \no \\
& \hspace*{6.2cm} z\in \{\zeta\in \rho(T) \, | \, 1\in \rho(K(\zeta))\}. 
\end{align}
$(v)$  Suppose $T_1$ defined by \eqref{2.65} is m-accretive. 
\end{hypothesis}

\begin{theorem}\lb{t2.12}
Assume that $T_0$ and $T$ satisfy Hypothesis \ref{h2.5} and and that $T$ and $T_1$ satisfy Hypothesis \ref{h2.11} for the same $A$ and $B$ as in Hypothesis \ref{h2.5}.  If 
\begin{equation}\lb{2.66}
\dom\big(T_0^{1/2}\big)=\dom\big(T^{1/2}\big),
\end{equation}
then
\begin{align}\lb{2.67}
\begin{split} 
& \dom\big(T_1^{1/2}\big) = \dom\big(T_0^{1/2}\big) = \dom\big(T^{1/2}\big)    \\
& \quad = \dom\big((T^*)^{1/2}\big) = \dom\big((T_0^*)^{1/2}\big) = \dom\big((T_1^*)^{1/2}\big).
\end{split} 
\end{align}
\end{theorem}
\begin{proof}
It suffices to verify \eqref{2.10}, \eqref{2.11}, and \eqref{2.25ccc} with $T_0$ replaced by $T$, that is,
\begin{equation}\lb{2.68}
\begin{split}
&\int_R^{\infty}d\lambda\, \lambda^{-1}\big\|A(T+(\lambda+E)I_{\cH})^{-1/2}\big\|_{\cB(\cH,\cK)}\\
&\hspace*{1.75cm} \times \big\|\overline{(T+(\lambda+E)I_{\cH})^{-1/2}B^{\ast}}\big\|_{\cB(\cK,\cH)}<\infty,   
\quad E \geq E_0,
\end{split}
\end{equation}
\begin{equation}\lb{2.69}
\begin{split}
&\lim_{E\rightarrow \infty}\int_R^{\infty}d\lambda\, \lambda^{-1}\big\|A(T+(\lambda+E)I_{\cH})^{-1/2}\big\|_{\cB(\cH,\cK)}\\
&\hspace*{2.6cm}\times \big\|\overline{(T+(\lambda+E)I_{\cH})^{-1/2}B^{\ast}}\big\|_{\cB(\cK,\cH)}=0,
\end{split}
\end{equation}
and
\begin{equation}
\lim_{E\rightarrow \infty}\big\|A(T+EI_{\cH})^{-1/2}\big\|_{\cB(\cH,\cK)}\big\|\overline{(T+EI_{\cH})^{-1/2}B^{\ast}}\big\|_{\cB(\cK,\cH)} =0,     \lb{2.69ccc}
\end{equation}
as \eqref{2.67} then follows from Corollary \ref{c2.5a} with $T$ playing the role of $T_0$ in 
Corollary \ref{c2.5a}. We note that Hypothesis \ref{h2.11} (in particular, \eqref{2.62} and \eqref{2.63}) 
guarantees that Corollary \ref{c2.5a} applies to $T$ once \eqref{2.68}, \eqref{2.69}, and \eqref{2.69ccc} are verified.

To verify \eqref{2.68}, one observes 
 \begin{align}
& \big\|A(T+(\lambda+E)I_{\cH})^{-1/2}\big\|_{\cB(\cH,\cK)} \big\|\overline{(T+(\lambda+E)I_{\cH})^{-1/2}B^{\ast}}\big\|_{\cB(\cK,\cH)}\no\\
&\quad = \big\|A(T_0+(\lambda+E)I_{\cH})^{-1/2}(T_0+(\lambda+E)I_{\cH})^{1/2}(T+(\lambda+E)I_{\cH})^{-1/2}\big\|_{\cB(\cH,\cK)}\no\\
& \qquad \times \big\|\{(T+(\lambda+E)I_{\cH})^{-1/2}(T_0+(\lambda+E)I_{\cH})^{1/2} \times    \no \\
& \qquad \quad \times 
(T_0+(\lambda+E)I_{\cH})^{-1/2}B^{\ast}\}^{\rm cl}\big\|_{\cB(\cK,\cH)}\no\\ 
& \quad = \big\|A(T_0+(\lambda+E)I_{\cH})^{-1/2}(T_0+(\lambda+E)I_{\cH})^{1/2}(T+(\lambda+E)I_{\cH})^{-1/2}\big\|_{\cB(\cH,\cK)}\no\\
& \qquad \times \big\| \big[B(T_0^*+(\lambda+E)I_{\cH})^{-1/2}(T_0^*+(\lambda+E)I_{\cH})^{1/2} \times   \no \\ 
& \qquad \quad \times 
(T^*+(\lambda+E)I_{\cH})^{-1/2}\big]^*\big\|_{\cB(\cK,\cH)}\no\\
& \quad \leq \big\|A(T_0+(\lambda+E)I_{\cH})^{-1/2}\big\|_{\cB(\cH,\cK)}\big\|B(T_0^*+(\lambda+E)I_{\cH})^{-1/2}\big\|_{\cB(\cH,\cK)}\no\\
&\qquad \times\big\|(T_0+(\lambda+E)I_{\cH})^{1/2}(T+(\lambda+E)I_{\cH})^{-1/2}\big\|_{\cB(\cH)} 
\lb{2.70} \\
&\qquad \times \big\|(T_0^*+(\lambda+E)I_{\cH})^{1/2}(T^*+(\lambda+E)I_{\cH})^{-1/2}\big\|_{\cB(\cH)},  \quad  \lambda>0,\; E>0,     \no 
 \end{align}
where for typographical reasons only, the symbol $\{\dots\}^{\rm cl}$ abbreviates the operator closure (which is usually 
denoted by just a bar in this paper).  
 In view of \eqref{2.8}, \eqref{2.62}, and \eqref{2.66}, Lemma \ref{l2.8} guarantees the existence of a constant $M>0$ for which 
\begin{equation}\lb{2.71}
\begin{split}
\big\|(T_0+EI_{\cH})^{1/2}(T+EI_{\cH})^{-1/2}\big\|_{\cB(\cH)}&\leq M^{1/2}, \\
\big\|(T_0^*+EI_{\cH})^{1/2}(T^*+EI_{\cH})^{-1/2}]\big\|_{\cB(\cH)}&\leq M^{1/2}, 
\end{split}
\quad E\geq 1.
\end{equation}
Using
\begin{equation}\lb{2.72}
\big\|B(T_0^*+EI_{\cH})^{-1/2}\big\|_{\cB(\cH,\cK)}=\big\|\overline{(T_0+EI_{\cH})^{-1/2}B^*}\big\|_{\cB(\cK,\cH)},  \quad  E>0, 
\end{equation}
and \eqref{2.71}, one further estimates \eqref{2.70} by
\begin{align}
& \big\|A(T+(\lambda+E)I_{\cH})^{-1/2}\big\|_{\cB(\cH,\cK)} \big\|\overline{(T+(\lambda+E)I_{\cH})^{-1/2}B^{\ast}}\big\|_{\cB(\cK,\cH)}\no\\
&\quad \leq M \big\|A(T_0+(\lambda+E)I_{\cH})^{-1/2}\big\|_{\cB(\cH,\cK)}\big\|\overline{(T_0+EI_{\cH})^{-1/2}B^*}\big\|_{\cB(\cK,\cH)},\lb{2.73}\\
&\hspace*{8.43cm} \lambda>0, \, E>1.\no
\end{align}
By hypothesis, $T_0$ satisfies \eqref{2.10}, \eqref{2.11}, and \eqref{2.25ccc}.  With the estimate \eqref{2.73} in hand, \eqref{2.68}, \eqref{2.69}, and \eqref{2.69ccc} follow from \eqref{2.10}, \eqref{2.11}, and \eqref{2.25ccc}, respectively. 
\end{proof}

\begin{remark} \lb{r2.13}
We note that the stability of square root domains has important consequences for trace formulas 
involving resolvent differences and symmetrized Fredholm (perturbation) determinants as 
discussed in detail in \cite{GZ12}: 
Let $A$ and $A_0$ be densely defined, closed, linear operators in $\cH$. 
Suppose there exists $t_0 \in \bbR$ such that $A + t_0 I_{\cH}$ and
$A_0 + t_0 I_{\cH}$ are of positive-type and $(A  + t_0 I_{\cH}) \in \Sect(\omega_0)$,
$(A_0  + t_0 I_{\cH}) \in \Sect(\omega_0)$ for some $\omega_0 \in [0,\pi)$. 
In addition, assume that for some $t_1 \geq t_0$,
\begin{align}
& \dom\big((A_0 + t_1 I_{\cH})^{1/2}\big) \subseteq
\dom\big((A + t_1 I_{\cH})^{1/2}\big),     \lb{7.18} \\
& \dom\big((A_0^* + t_1 I_{\cH})^{1/2}\big) \subseteq
\dom\big((A^* + t_1 I_{\cH})^{1/2}\big),     \lb{7.19} \\
& \ol{(A + t_1 I_{\cH})^{1/2} \big[(A + t_1 I_{\cH})^{-1} - (A_0 + t_1 I_{\cH})^{-1}\big]
	(A + t_1 I_{\cH})^{1/2}} \in \cB_1(\cH).    \lb{7.20}
\end{align}
Then 
\begin{align}
\begin{split}
& - \f{d}{dz} \ln\Big({\det}_{\cH}\Big(\ol{(A - z I_{\cH})^{1/2}(A_0 - z I_{\cH})^{-1}
	(A - z I_{\cH})^{1/2}}\Big)\Big)   \\
&\quad = {\tr}_{\cH}\big((A - z I_{\cH})^{-1} - (A_0 - z I_{\cH})^{-1}\big),     \lb{7.21}
\end{split}
\end{align}
for all $z \in \bbC \backslash \big(\ol{- t_0 + S_{\omega_0}}\big)$ such that 
$\ol{(A - z I_{\cH})^{1/2}(A_0 - z I_{\cH})^{-1}(A - z I_{\cH})^{1/2}}$ is boundedly invertible. 
(Analogous trace formulas can be derived for modified Fredholm determinants replacing 
$\cB_1(\cH)$ by $\cB_k(\cH)$, $k \in \bbN$, $k \geq 2$, including additional terms under 
the trace.) 

Here $A$ is said to be of {\it positive-type} if
\begin{align}
(-\infty, 0] \subset \rho(A) \, \text{ and } \, 
M_A = \sup_{t \geq 0} \big\|(1 + t) (A + t I_{\cH})^{-1}\big\|_{\cB(\cH)} < \infty,   \lb{7.7} 
\end{align}
and $A$ is called {\it sectorial of angle $\omega \in [0,\pi)$}, denoted by
$A \in \Sect(\omega)$, if
\begin{align}
\sigma(A) \subseteq \ol{S_{\omega}}  \, 
\text{ and for all $\omega' \in (\omega,\pi)$, }  \,
M(A,\omega') = \sup_{z\in \bbC\backslash \ol{S_{\omega'}}}
\big\| z (A - z I_{\cH})^{-1}\big\|_{\cB(\cH)} < \infty,       \lb{7.8} 
\end{align}
where $S_{\omega} \subset \bbC$, $\omega \in [0,\pi)$, denotes the open sector
\begin{equation}
S_{\omega} = \begin{cases} \{z\in\bbC \,|\, z \neq 0, \, |\arg(z)|<\omega\},
& \omega \in (0,\pi), \\
(0,\infty), & \omega = 0,
\end{cases}
\end{equation}
with vertex at $z=0$ along the positive real axis and opening angle $2 \omega$. Assumptions 
such \eqref{7.18}, \eqref{7.19} make it plain that control over square root domains is at the core 
for the validity of the trace formula \eqref{7.21}.  

In the special case where in addition $A$ and $A_0$ are self-adjoint in $\cH$, trace relations of 
the type \eqref{7.21} are intimately related to the notion of the spectral shift function associated 
with the pair $(A,A_0)$, and hence underscore its direct relevance to spectral and scattering 
theory for this pair (cf.\ \cite{GZ12} for details). 
\end{remark}

\begin{remark} \lb{r2.14}
We conclude this section by recalling that if $\gq^{}_T(\cdot,\cdot)$ denotes the sectorial form uniquely 
associated with m-sectorial operator $T$ (cf.\ \cite[Theorem\ VI.2.7]{Ka80}), then, as shown by 
Kato \cite{Ka62}, the fact 
$\dom\big(T^{1/2}\big) = \dom\big((T^*)^{1/2}\big)$ also implies 
\begin{equation}
\dom(\gq^{}_T) = \dom\big(T^{1/2}\big) = \dom\big((T^*)^{1/2}\big) = \dom(\gq^{}_{T^*}). 
\end{equation}
In fact, if two out of $\dom(\gq^{}_T)$, $\dom\big(T^{1/2}\big)$, $\dom\big((T^*)^{1/2}\big)$ are equal, all three are equal, see \cite{Ka62}, \cite{Li62}, and \cite[Theorem\ 8.2]{Ou05}. 
\end{remark}

\section{Applications to Divergence Form Elliptic \\ 
Partial Differential Operators}  \lb{s3}

In this section we apply the abstract formalism of Section \ref{s2} to uniformly elliptic 
second-order partial differential operators of the form $- {\rm div}(a\nabla \, \cdot \,) + V$ 
on certain open sets $\Omega \subseteq \bbR^n$, $n \in \bbN$, with appropriate 
boundary conditions on $\partial\Omega$.  

For simplicity of notation, we agree in the following to abbreviate $L^p(\Omega; d^n x)$  simply by 
$L^p(\Omega)$,  
$p \in [1,\infty) \cup \{\infty\}$, $n \in \bbN$. 

We start with some general considerations and eventually strengthen our hypotheses on $\Omega$. 

\begin{hypothesis}\lb{h3.1} 
Let $n\in\bbN$ and assume that $\Omega \subseteq \bbR^n$ is nonempty and open\footnote{If $n=1$, we assume 
that $\Omega \subseteq \bbR$ equals a nonempty open interval.}. \\
$(i)$ Suppose $a: \Omega \rightarrow \bbC^{n\times n}$ is a Lebesgue measurable, matrix-valued function which
is bounded and uniformly elliptic, that is, there exist constants 
$0 < a_1 \leq a_2 < \infty$ such that 
\begin{align}\lb{3.2}
\begin{split} 
a_1 \|\xi\|_{\bbC^n}^2 \leq \Re\big[(\xi,a(x)\xi)_{\bbC^n}\big] \, \text{ and } \, 
|(\zeta, a(x)\xi)_{\bbC^n}|\leq a_2 \|\xi\|_{\bbC^n} \|\zeta\|_{\bbC^n},&  \\  
\text{$\xi,\zeta \in \bbC^n$, a.e.\ $x\in \Omega$.}&
\end{split} 
\end{align}
$(ii)$ Denote by $\cW(\Omega)$ a closed subspace of $W^{1,2}(\Omega)$ satisfying 
\begin{equation}\lb{3.3}
C_0^{\infty}(\Omega)\subset \cW(\Omega).
\end{equation}
$(iii)$  Let $L_{a,\cW(\Omega)}$ denote the m-accretive $($in fact, m-sectorial\,$)$ operator uniquely associated with the densely defined, accretive, and closed sesquilinear form 
\begin{equation}\lb{3.4}
\gq^{}_{a,\cW(\Omega)}(f,g) = 
\int_{\Omega} d^nx \, \big((\nabla f)(x), a(x)(\nabla g)(x)\big)_{\bbC^n}, 
\quad f, g \in \dom(\gq^{}_{a,\cW(\Omega)})=\cW(\Omega).
\end{equation}
$(iv)$ Suppose that $V:\Omega\rightarrow \bbC$ is $($Lebesgue$)$ measurable and factored according to 
\begin{equation}\lb{3.13}
V(x) = u(x) v(x), \quad v(x)=|V(x)|^{1/2}, \quad u(x) = e^{i\arg(V(x))} v(x) \, 
\text{ for a.e. $x\in \Omega$,}
\end{equation}
such that 
\begin{equation}\lb{3.6}
\cW(\Omega)\subseteq \dom(v).
\end{equation}
\end{hypothesis}

The quadratic form $\gq^{}_V$ in $L^2(\Omega)$ uniquely associated with $V$ is defined by
\begin{equation}
\gq^{}_V(f,g) = \big(v f, e^{i \arg(V)} v g\big)_{L^2(\Omega)}, \quad 
f, g \in \dom(\gq^{}_V) = \dom(v).     \lb{3.7a} 
\end{equation}
Under appropriate assumptions on $V$ (see Hypotheses \ref{h3.2} and \ref{h3.3} below), the 
form sum of $\gq^{}_{a,\cW(\Omega)}$ and $\gq^{}_V$ will define a sectorial form on $\cW(\Omega)$ 
and the uniquely associated operator with $\gq^{}_{a,\cW(\Omega)} + \gq^{}_V$ will be denoted by 
$L_{a, \cW(\Omega)} +_{\gq} V$ (see also the paragraph following \eqref{B.46}).

With respect to the potential coefficient $V$ we introduce the following hypotheses:

\begin{hypothesis}\lb{h3.2} 
Let $n\in\bbN$, assume that $\Omega \subseteq \bbR^n$ is nonempty and 
open\footnote{\label{note1}If $n=1$, 
we assume that $\Omega \subseteq \bbR$ equals a nonempty open interval.}, and let 
$V\in L^p(\Omega) + L^\infty(\Omega)$ for some $p>n/2$, $p\geq 1$. 
\end{hypothesis}

In addition, we also discuss the critical $L^p$-index $p = n/2$ for $V$ for $n \geq 3$:

\begin{hypothesis}\lb{h3.3} 
Let $n \in \bbN \backslash \{2\}$, assume that $\Omega \subseteq \bbR^n$ is nonempty and 
open\footnotemark[\value{footnote}], and let $V\in L^{n/2}(\Omega) + L^\infty(\Omega)$ 
for $n \geq 3$ and $V\in L^1(\Omega) + L^\infty(\Omega)$ for $n = 1$. 
\end{hypothesis}

Here $V\in L^q(\Omega) + L^\infty(\Omega)$ means that $V$ permits a decomposition 
$V = V_q + V_{\infty}$ with $V_q \in L^q(\Omega)$ for some $q \geq 1$ and 
$V_{\infty} \in L^\infty(\Omega)$. 

\begin{remark}\lb{r3.4}
$(i)$ One notes that by assumption \eqref{3.2}, the real part of $\gq^{}_{a,\cW(\Omega)}$ is 
non-negative,
\begin{equation}\lb{4.26.1}
\Re[\gq^{}_{a,\cW(\Omega)}(f,f)]\geq a_1\|\nabla f\|^2_{L^2(\Omega)}, \quad f\in \cW(\Omega).
\end{equation}
$(ii)$ To see that the (sesquilinear) form $\gq^{}_{a,\cW(\Omega)}$ defined by \eqref{3.4} 
is closed, one recalls that by definition, $\gq^{}_{a,\cW(\Omega)}$ is closed if and only if 
$\dom(\gq^{}_{a,\cW(\Omega)})$ is closed with respect to the norm 
$\|\cdot\|_{\gq^{}_{a,\cW(\Omega)}}=(\Re[\gq^{}_{a,\cW(\Omega)}(\cdot,\cdot)] 
+ \|\cdot\|_{L^2(\Omega)}^2)^{1/2}$. Under condition \eqref{3.2}, $\|\cdot\|_{W^{1,2}(\Omega)}$ 
and $\|\cdot\|_{\gq^{}_{a,\cW(\Omega)}}$ are 
equivalent norms on $\cW(\Omega)$.  Since $\cW(\Omega)$ is assumed to be closed under 
the $\|\cdot\|_{W^{1,2}(\Omega)}$ norm, $\cW(\Omega)$ is closed under the 
$\|\cdot\|_{\gq^{}_{a,\cW(\Omega)}}$ norm and it follows that 
$\gq^{}_{a,\cW(\Omega)}$ is a closed form. 
For additional details in connection with $L_{a,\cW(\Omega)}$ associated to \eqref{3.4} (especially, in the case where $a$ is a symmetric $n \times n$ matrix), see, for 
instance, \cite{AtE97}, \cite{GGKR02}, \cite{GKR01}, \cite{Gr89}, \cite{GR89}, \cite{HDHR08}, 
\cite{HDMRS09}, \cite{HDR09}, \cite{HKR09}, \cite{KR97}, \cite{KR99}, \cite{KR00}, \cite{Ou04}, 
and \cite[Sect.\ 4.1]{Ou05}. 
\end{remark}

\begin{remark}\lb{r3.5}
Next, we briefly discuss three special cases corresponding to specific choices for 
$\cW(\Omega)$ and use this to introduce some convenient notation:   \\
($i$)  Choosing $\cW(\Omega) = W_0^{1,2}(\Omega)$ formally corresponds to imposing a {\it Dirichlet boundary condition} on $\partial \Omega$.  For simplicity of notation, we denote the corresponding operator $L_{a,W_0^{1,2}(\Omega)}$ simply by $L_{a,\Omega,D}$, that is, 
\begin{equation}\lb{4.25.1}
L_{a,\Omega,D}:=L_{a,W_0^{1,2}(\Omega)}.
\end{equation}
($ii$)  Choosing $\cW(\Omega) = W^{1,2}(\Omega)$ formally corresponds to imposing a {\it Neumann boundary condition} on $\partial \Omega$.  The corresponding operator $L_{a,W^{1,2}(\Omega)}$ will be denoted by $L_{a,\Omega,N}$, that is,  
\begin{equation}\lb{4.25.2}
L_{a,\Omega,N}:=L_{a,W^{1,2}(\Omega)}.
\end{equation}
($iii$) With $\Pi$ denoting a fixed, relatively open subset of $\partial \Omega$, consider 
\begin{align}
\begin{split} 
W^{1,2}_{\Pi}(\Omega)&= \overline{\{f|_{\Omega} \; | \; f \in C_0^{\infty}(\bbR^n); \, 
\supp \, (f) \cap (\partial \Omega \backslash \Pi) = \emptyset\}}^{W^{1,2}(\Omega)}     \lb{4.25.2a} \\
&= \ol{\big\{f|_{\Omega} \, \big|\, f \in C_0^{\infty}(\bbR^n \backslash 
(\partial \Omega \backslash \Pi))\big\}}^{W^{1,2}(\Omega)},   
\end{split} 
\end{align}
where $\overline{\cM}^{W^{1,2}(\Omega)}$ denotes the closure of 
$\cM\subseteq W^{1,2}(\Omega)$ in $W^{1,2}(\Omega)$. (One observes in connection with \eqref{4.25.2a}, that since $\partial \Omega \backslash \Pi$ is relatively closed in 
$\partial \Omega$, it is also closed in $\bbR^n$.) The corresponding operator 
$L_{a,W^{1,2}_{\Pi}(\Omega)}$ will be denoted by $L_{a,\Omega,\Pi}$, that is,  
\begin{equation}
L_{a,\Omega,\Pi}:=L_{a,W_{\Pi}^{1,2}(\Omega)},       \lb{4.25.3}
\end{equation} 
The choice of $\cW(\Omega) = W^{1,2}_{\Pi}(\Omega)$ formally corresponds 
to {\it mixed boundary conditions} with a {\it Dirichlet boundary condition} on 
$\partial \Omega \backslash \overline{\Pi}$ and a {\it Neumann-type boundary condition} on $\Pi$ 
(cf., e.g., \cite{AtE97}, \cite[Definition 2.3 and Remark 2.4]{HKR09}, \cite[Definition 2.2]{KR00}, 
\cite{Ou04}, and \cite[\S4.1]{Ou05} as well as the references and discussions therein.). 
The choice $\Pi = \emptyset$ is permitted and corresponds to the Dirichlet case $(i)$, 
\begin{equation}
W^{1,2}_{\emptyset}(\Omega) = W_0^{1,2}(\Omega). 
\end{equation}
The choice $\Pi=\partial \Omega$ is permitted as well, but we caution that in general,  
$W^{1,2}_{\partial \Omega}(\Omega) \subsetneq W^{1,2}(\Omega)$. However, under additional 
assumptions, for instance, if $\Omega$ possesses the {\it segment 
property} (cf., e.g., \cite[Theorem 3.22]{AF03}), equivalently, if $\partial \Omega$ is {\it of class $C$}  
(cf.\ \cite[Theorems\ V.4.4 and V.4.7]{EE89}), it is known that 
\begin{equation} 
W^{1,2}_{\partial \Omega}(\Omega) = W^{1,2}(\Omega).
\end{equation} 
\end{remark}

Moreover, in the particular case where $\Omega=\bbR^n$, one infers that 
$\cW(\bbR^n)=W^{1,2}(\bbR^n)$ by \eqref{3.3} and the fact that $C_0^{\infty}(\bbR^n)$ is dense 
in $W^{1,2}(\bbR^n)$ (cf., e.g., \cite[Theorem 3.22]{AF03}). We then drop the subscript 
$W^{1,2}(\bbR^n)$ and simply write
\begin{equation}\lb{4.25.4}
L_a:=L_{a,W^{1,2}(\bbR^n)}.
\end{equation}
Finally, in the special case where the matrix-valued function $a$ coincides with 
the $n\times n$ identity matrix $I_n$ a.e.\ on $\Omega$, we shall write
\begin{equation}\lb{5.4.1}
\begin{split}
& -\Delta := L_{I_n}, \quad -\Delta_{\cW(\Omega)} := L_{I_n,\cW(\Omega)},      \\
& -\Delta_{\Omega,N}:=L_{I_n,\Omega,N},\quad\quad -\Delta_{\Omega,D}:=L_{I_n,\Omega,D},\quad\quad -\Delta_{\Omega,\Pi}:=L_{I_n,\Omega,\Pi},
\end{split}
\end{equation}
where $\Delta:=\sum_{j=1}^n \partial_{x_j}^2$ denotes the usual Laplacian on 
$\mathbb{R}^n$. In particular, $-\Delta_{\cW(\Omega)}$ is self-adjoint, and 
consequently, its square root domain coincides with its form domain, that is,
\begin{equation}\lb{4.26.5}
\dom\big((-\Delta_{\cW(\Omega)})^{1/2} \big)=\cW(\Omega).
\end{equation}

Our next result uses the relations \eqref{4.26.2} as input and then proves that if $V$ is infinitesimally 
form bounded with respect to $-\Delta_{\cW(\Omega)}$, then $V$ is also infinitesimally form bounded 
with respect to $L_{a,\cW(\Omega)}$ and that the square root domains \eqref{4.26.2} remain stable 
under perturbations $V$ satisfying Hypotheses \ref{h3.1}\,$(iv)$. 

\begin{theorem}\lb{t3.6}
Assume Hypotheses \ref{h3.1} and suppose that $L_{a,\cW(\Omega)}$ satisfies
\begin{equation}\lb{4.26.2}
\dom\big(L_{a,\cW(\Omega)}^{1/2}\big)=\dom\big((L_{a,\cW(\Omega)}^*)^{1/2}\big)=\cW(\Omega). 
\end{equation} 
If there exist constants $M>0$, $\beta > 0$, and $\e_0>0$ such that for all $0<\e<\e_0$, 
\begin{equation}\lb{3.9aa}
\big\| |V|^{1/2}f \big\|_{L^2(\Omega)}^2\leq \e\big\|(-\Delta_{\cW(\Omega)})^{1/2}f\big\|_{L^2(\Omega)}^2+M\e^{-\beta}\|f\|_{L^2(\Omega)}^2,  \quad  f\in \cW(\Omega),
\end{equation}
then the following items $(i)$ and $(ii)$ hold: \\
$(i)$ $V$ is infinitesimally form bounded with respect to $L_{a,\cW(\Omega)}$ and 
\begin{align}\lb{3.9a}
\begin{split} 
\big\| |V|^{1/2}f \big\|_{L^2(\Omega)}^2\leq \e\Re[\gq^{}_{a,\cW(\Omega)}(f,f)] 
+ M a_1^{-\beta} \e^{-\beta}\|f\|_{L^2(\Omega)}^2,&   \\ 
f\in \cW(\Omega),\; 0<\e< a_1^{-1} \e_0.&
\end{split} 
\end{align}
$(ii)$  The form sum $L_{a, \cW(\Omega)} +_{\gq} V$ is an m-sectorial operator satisfying  
\begin{align}\lb{3.9}
\begin{split} 
& \dom\big((L_{a, \cW(\Omega)} +_{\gq} V)^{1/2} \big) = \dom\big(((L_{a, \cW(\Omega)} +_{\gq} V)^*)^{1/2} \big)  \\
& \quad =\dom\big( L_{a,\cW(\Omega)}^{1/2}\big) = \dom\big((L_{a, \cW(\Omega)}^*)^{1/2} \big) 
= \cW(\Omega).
\end{split} 
\end{align}
\end{theorem}
\begin{proof}
Let $\mathfrak{h}_V$ denote the sesquilinear form associated to the operator of multiplication 
by $V$, that is, 
\begin{equation}\lb{4.26.3}
\mathfrak{h}_V(f,g) = \big(vf,e^{i\arg(V(\cdot))} vg\big)_{L^2(\Omega)}, 
\quad f,g\in \dom(v)=\dom(\mathfrak{h}_V).
\end{equation}
Since $\cW(\Omega) \subseteq \dom(v)$ by hypothesis, in order to prove that $V$ is infinitesimally 
form bounded with respect to $L_{a,\cW(\Omega)}$, it suffices to prove the existence of an 
$\widetilde \e_0>0$ and an $\eta:(0,\widetilde \e_0)\rightarrow (0,\infty)$ for which 
\begin{equation}\lb{4.26.4}
\|vf\|_{L^2(\Omega)}^2\leq \e \Re[\gq^{}_{a,\cW(\Omega)}(f,f)]+\eta(\e)\|f\|_{L^2(\Omega)}^2, 
\quad f\in \cW(\Omega),\; 0<\e<\widetilde\e_0.
\end{equation}
Sufficiency of the condition in \eqref{4.26.4} is due to the obvious estimate
\begin{equation}
|\mathfrak{h}_V(f,f)|\leq \|vf\|_{L^2(\Omega)}^2,\quad f\in \dom(v),
\end{equation}
and the fact $\mathfrak{h}_V$ is infinitesimally form bounded with respect to 
$\gq^{}_{a,\cW(\Omega)}$ if and only if the former is infinitesimally form bounded with 
respect to $\Re[\gq^{}_{a,\cW(\Omega)}]$ (cf., e.g., \cite[p.\ 319]{Ka80}).

Freely using \eqref{3.9aa}, one estimates
\begin{align}
\|vf\|_{L^2(\Omega)}^2&\leq \epsilon\big\|(-\Delta_{\cW(\Omega)})^{1/2}f\big\|_{L^2(\Omega)}^2 
+M\epsilon^{-\beta}\|f\|_{L^2(\Omega)}^2   \lb{4.26.8a}\\
&= \epsilon\|\nabla f\|_{L^2(\Omega)^n}^2+M\epsilon^{-\beta}\|f\|_{L^2(\Omega)}^2    \lb{4.26.8b}\\
&\leq a_1^{-1}\epsilon\Re [\gq^{}_{a,\cW(\Omega)}(f,f)]+M\epsilon^{-\beta}\|f\|_{L^2(\Omega)}^2,   
\quad f\in \cW(\Omega), \; 0<\epsilon<\e_0,      \lb{4.26.8c}  
\end{align}
where the 2nd representation theorem (cf., e.g., \cite[Theorem VI.2.23]{Ka80}) was applied to 
the (non-negative, self-adjoint) reference operator $-\Delta_{\cW(\Omega)}$ and its associated  
form $\gq^{}_{I_n}$ in  \eqref{4.26.8b}.  To arrive at \eqref{4.26.8c} the 
first inequality in \eqref{3.2} was used.  Choosing $\epsilon= a_1 \e$, $0<\e< a_1^{-1} \e_0$ in 
\eqref{4.26.8a}--\eqref{4.26.8c} yields the form bound \eqref{3.9a}.

On the basis of \eqref{3.9a}, $L_{a,\cW(\Omega)} +_{\gq} V$ defines an m-sectorial operator by  
standard perturbation theoretic results (see, e.g., \cite[Theorem\ VI.1.33]{Ka80}). 
Next, one notes that by Theorem \ref{tA.3}, the operator $T$ defined as in \eqref{2.4} and \eqref{2.5} 
in terms of the identifications 
\begin{equation}\lb{3.15}
T_0=L_{a,\cW(\Omega)}, \quad A=v, \quad  B^*=u,
\end{equation}
coincides with the form sum $L_{a,\cW(\Omega)} +_{\gq} V$. In particular, according to the 
factorization \eqref{3.13}, $L_{a,\cW(\Omega)}$ and $L_{a,\cW(\Omega)} +_{\gq}V$ satisfy the 
resolvent equation
\begin{align}
&(L_{a,\cW(\Omega)} +_{\gq} V-zI_{L^2(\Omega)})^{-1}=(L_{a,\cW(\Omega)}-zI_{L^2(\Omega)})^{-1}  \no\\
& \quad -\overline{(L_{a,\cW(\Omega)}-zI_{L^2(\Omega)})^{-1}u}\big[I_{L^2(\Omega)} 
-\overline{v(L_{a,\cW(\Omega)}-zI_{L^2(\Omega)})^{-1}u} 
\big]^{-1}\lb{3.14}\\
&\qquad \times v(L_{a,\cW(\Omega)}-zI_{L^2(\Omega)})^{-1},  \quad 
z\in \rho(L_{a,\cW(\Omega)})\cap\rho(L_{a,\cW(\Omega)} +_{\gq} V). \no
\end{align}
Given \eqref{3.15}, one now proceeds to verifying items $(i)$, $(iii)$ of Hypothesis \ref{h2.1}: 
For this purpose one notes that by Lemma \ref{l2.10}, \eqref{3.9aa} implies the existence
of constants $M'>0$, $E_1\geq 1$, and $\beta>0$ such that
\begin{equation}\lb{3.7}
\big\| v(-\Delta_{\cW(\Omega)}+EI_{L^2(\Omega)})^{-1/2}\big\|_{\cB(L^2(\Omega))}\leq M'E^{-\beta}, \quad E>E_1. 
\end{equation}
Subsequently, to verify Hypothesis \ref{h2.1}\,$(iii)$, one uses \eqref{3.7} and estimates, 
\begin{align}
&\big\|\overline{(L_{a,\cW(\Omega)}+EI_{L^2(\Omega)})^{-1/2}u}\big\|_{\cB(L^2(\Omega))}\no\\
&\quad =\big\|{\ol u}(L_{a,\cW(\Omega)}^*+EI_{L^2(\Omega)})^{-1/2}\big\|_{\cB(L^2(\Omega))}\no\\
&\quad =\big\|v(L_{a,\cW(\Omega)}^*+EI_{L^2(\Omega)})^{-1/2}\big\|_{\cB(L^2(\Omega))}\no\\
&\quad\leq  \big\|v(-\Delta_{\cW(\Omega)}+EI_{L^2(\Omega)})^{-1/2}\big\|_{\cB(L^2(\Omega))}\no\\
&\qquad \times\big\|(-\Delta_{\cW(\Omega)}+EI_{L^2(\Omega)})^{1/2}(L_{a,\cW(\Omega)}^*+EI_{L^2(\Omega)})^{-1/2}\big\|_{\cB(L^2(\Omega))}\no\\
&\quad \leq C M' E^{-\beta},\quad E>E_1,\lb{3.16}
\end{align}
where $C>0$ is an appropriate constant (guaranteed to exist by \eqref{4.26.5}, \eqref{3.14}, and Lemma \ref{l2.8}) for which
\begin{equation}
\big\|(-\Delta_{\cW(\Omega)}+EI_{L^2(\Omega)})^{1/2}(L_{a,\cW(\Omega)}^*
+ EI_{L^2(\Omega)})^{-1/2}\big\|_{\cB(L^2(\Omega))}\leq C,\quad E\geq 1.
\end{equation}
We note that \eqref{4.26.2} is used freely throughout \eqref{3.16}.  Consequently, \eqref{3.7} 
and \eqref{3.16} can be used to estimate
\begin{align}
&\big\|\overline{v(L_{a,\cW(\Omega)}+EI_{L^2(\Omega)})^{-1}u}\big\|_{\cB(L^2(\Omega))}\no\\
&\quad =\big\| v(L_{a,\cW(\Omega)}+EI_{L^2(\Omega)})^{-1/2}\overline{(L_{a,\cW(\Omega)}+EI_{L^2(\Omega)})^{-1/2}u}\big\|_{\cB(L^2(\Omega))}\no\\
&\quad \leq C(M')^2E^{-2\beta},  \quad E>E_1.\lb{3.17}
\end{align}
Thus, $1\in \rho\big(\overline{v(L_{a,\cW(\Omega)}+EI_{L^2(\Omega)})^{-1}u} \big)$ for all $E>E_1$ sufficiently large, implying the validity of Hypothesis \ref{h2.1}\,$(iii)$.

Thus, by Corollary \ref{c2.5a} and 
Remark \ref{r2.7}, the proof of \eqref{3.9} reduces to verifying Hypothesis \ref{h2.5}, that is,  
Hypothesis \ref{h2.3} and \eqref{2.8} 
with the identifications \eqref{3.15}.  Hypothesis \ref{h2.3}\,$(i)$ and \eqref{2.8} are satisfied by 
assumption (cf.\ \eqref{4.26.2}).  Regarding parts $(ii)$ and $(iii)$ of Hypotheses \ref{h2.3}, \eqref{2.9} is clear in light of 
\eqref{4.26.2} and \eqref{3.6}. In order to verify \eqref{2.10}, one notes that by \eqref{3.7} 
and \eqref{3.16},
\begin{align}
&\lambda^{-1}\big\|v(L_{a,\cW(\Omega)}+(E+\lambda)I_{L^2(\Omega)})^{-1/2}
\big\|_{\cB(L^2(\Omega))}\no\\
&\qquad \times \big\|\overline{(L_{a,\cW(\Omega)} 
+ (E+\lambda)I_{L^2(\Omega)})^{-1/2}u} \big\|_{\cB(L^2(\Omega))}\no\\
&\quad \leq \lambda^{-1}C(M')^2(E+\lambda)^{-2\beta}\no\\
&\quad \leq  C(M')^2\lambda^{-1-2\beta}, 
\quad  E>E_1, \; \lambda>0,      \lb{3.20}
\end{align}
implying \eqref{2.10} with $R=1$, $E_0=E_1$, and the identifications \eqref{3.15}.  An application 
of Lebesgue's dominated convergence theorem, using \eqref{3.20}, then yields \eqref{2.11}.  The statement in \eqref{2.25ccc} follows immediately from \eqref{3.17}.
Finally, the first two equalities in \eqref{3.9} follow from Corollary \ref{c2.5a} and Remark \ref{r2.7}, 
and the last equality in \eqref{3.9} follows from assumption \eqref{4.26.2}.
\end{proof}

Next we investigate the special case $\Omega =\bbR^n$, $n\in \bbN$, in detail and show how to derive 
the basic inputs \eqref{4.26.2} and \eqref{3.9aa} for Theorem \ref{t3.6} in this particular situation. We 
start with \eqref{3.9aa}. 

\begin{lemma} \lb{l3.7}
Suppose $V$ satisfies Hypothesis \ref{h3.2} with $\Omega =\bbR^n$, $n\in \bbN$, and let 
$-\Delta$, $\dom(-\Delta)=W^{2,2}(\bbR^n)$, denote the self-adjoint realization of $($minus$)$ the 
Laplacian in $L^2(\bbR^n)$. Then $V$ is infinitesimally form bounded with respect to $-\Delta$ and there exists a 
constant $M_{n,p,V}>0$ $($depending only upon $n$, $p$, and $V$$)$ such that for all $\e > 0$ sufficiently small, 
\begin{align}
\begin{split} 
\big\||V|^{1/2} f\big\|_{L^2(\bbR^n)}^2\leq \e \big\|(-\Delta)^{1/2} f\big\|_{L^2(\bbR^n)}^2
+ M_{n,p,V}\e^{-n/(2p-n)}\|f\|_{L^2(\bbR^n)}^2,&\lb{5.2.1}\\
 f\in W^{1,2}(\bbR^n).& 
 \end{split} 
\end{align}
\end{lemma}
\begin{proof}
Without loss of generality we put the $L^\infty$-part $V_{\infty}$ of $V$ equal to zero. 
By Hypothesis \ref{h3.2} one infers 
\begin{equation}\lb{3.32}
v(-\Delta + E I_{L^2(\bbR^n)})^{-1/2}\in \cB_{2p}\big(L^2(\bbR^n)\big), \quad E > 0.
\end{equation}
Since $v\in L^{2p}(\bbR^n)$, \cite[Theorem XI.20]{RS79} (cf. also \cite[Theorem 4.1]{Si05}) yields
\begin{align}
&\big\|v(-\Delta + E I_{L^2(\bbR^n)})^{-1/2} \big\|_{\cB_{2p}(L^2(\bbR^n))}\no\\
&\quad \leq (2\pi)^{-n/2p}\|V\|_{L^{p}(\bbR^n)}^{1/2} 
\big\|(|\cdot|^2 + E)^{-1/2}\big\|_{L^{2p}(\bbR^n)}, \quad E > 0.\lb{3.33}
\end{align}
Employing  
\begin{equation}\lb{3.34}
\big\|(|\cdot|^2 + E)^{-1/2}\big\|_{L^{2p}(\bbR^n)}=C_n E^{(n-2p)/(4p)}, \quad E > 0,
\end{equation}
where $C_n>0$ only depends on $n\in\bbN$, \eqref{3.33} and \eqref{3.34} 
imply 
\begin{equation}\lb{3.35}
\big\|v(-\Delta + E I_{L^2(\bbR^n)})^{-1/2} \big\|_{\cB_{2p}(L^2(\bbR^n))} 
\leq C_{n,p,V} E^{(n-2p)/(4p)}, \quad E > 0,
\end{equation}
where  
\begin{equation}\lb{3.36}
C_{n,p,V}:=(2\pi)^{-n/2p}\|V\|_{L^p(\bbR^n)}^{1/2}C_n.
\end{equation}
The form bound in \eqref{5.2.1} now follows by repeating the estimate in \eqref{3.10} 
with $S=V$ and $T=-\Delta$.
\end{proof}

\begin{remark} \lb{r3.8} 
The infinitesimal form bound result of Lemma \ref{l3.7} may be found in \cite[Proposition\ 6.4]{Fa72}, 
but the exact form of the bound (in particular, the coefficient $M_{n,p,V}\e^{-n/(2p-n)}$ in 
\eqref{5.2.1}) is not explicitly stated there.  The proof given above is based directly on this circle 
of ideas.  An alternative proof for the case $n\geq 3$, in the more general setting of domains 
in $\bbR^n$,  may be found in \cite[Theorem\ 1.8.4] {Da89} (our proof of item $(iii)$ in 
Theorem \ref{t3.14} mirrors this alternative approach).
\end{remark}

Next we focus on \eqref{4.26.2} in the particular case $\Omega = \bbR^n$, $n \in \bbN$. In this 
context we recall that the ``Kato square root problem" for second-order divergence form elliptic 
operators in $\bbR^n$, has been solved in \cite{AHLMT02} (we note that the case $n=1$ was 
treated earlier in \cite{CMM82}, the case $n=2$ in \cite{HM02}).  To be precise, we assume 
Hypothesis \ref{h3.1}\,$(i)$--$(iii)$ with 
$\Omega=\bbR^n$ and hence have $\cW(\Omega)=W^{1,2}(\bbR^n)$.  Then $L_a$ defined by 
\eqref{4.25.4} (cf.\ also \eqref{3.4}) formally takes the form of a second-order divergence form 
operator,
\begin{equation}\label{eq3}
L_a:= - {\rm div}(a\nabla \, \cdot \,) = - \sum_{1\leq j,k\leq n} \partial_j a_{j,k} \partial_k.
\end{equation}
The main result of \cite{AHLMT02} is that Hypothesis \ref{h2.3}\,$(i)$ holds with $T_0:=L_a$,
$\mathcal{H} =L^2(\bbR^n)$, and that
\begin{equation}\label{eq4}
\dom\big(L_a^{1/2}\big) = \dom\big((L_a^*)^{1/2}\big) = W^{1,2}(\bbR^n), 
\end{equation}
and hence provides precisely \eqref{4.26.2} in the particular case $\Omega = \bbR^n$, $n \in \bbN$.
Quantitatively, one has
\begin{equation} 
\big\|L_a^{1/2} f \big\|_{L^2(\bbR^n)}\approx \|\nabla f\|_{L^2(\bbR^n)^n},
\end{equation} 
where the implicit constants depend only upon $a_1$, $a_2$, and $n$ (and we recall that the 
symbol $\approx$ indicates equivalent norms). 

Thus, Lemma \ref{l3.7} and \eqref{eq4} together with Theorem \ref{t3.6}, in the particular 
case $\Omega = \bbR^n$, $n \in \bbN$, imply Theorem \ref{t3.9} below. However, we prefer to  
provide a second and alternative proof, reducing it directly to Corollary \ref{c2.5a} as follows: 

\begin{theorem}\label{t3.9} 
Assume Hypotheses \ref{h3.1} and \ref{h3.2} with $\Omega=\bbR^n$ 
$($so that $\cW(\Omega)=W^{1,2}(\bbR^n)$$)$, $n \in \bbN$, 
and let $L_a$ denote the operator defined in \eqref{4.25.4}. Then 
 \begin{align}
 \begin{split}  
\dom\big((L_a +_{\gq} V)^{1/2}\big) = \dom\big(((L_a +_{\gq} V)^*)^{1/2}\big)   \\
=\dom\big(L_a^{1/2}\big) =\dom\big((L_a^*)^{1/2}\big) = W^{1,2}(\bbR^n).
\end{split} 
 \end{align} 
\end{theorem}
\begin{proof} Again, without loss of generality, we may 
assume in the following that $V_{\infty} = 0$ and hence $V \in L^p(\bbR^n)$ for some $p > n/2$, 
$p \geq 1$.
  
For the sake of notational convenience, we shall adapt the following convention
for the duration of the proof of Theorem \ref{t3.9}: The symbol $\e$ will denote a generic
positive number
(typically small, but not necessarily so), which may differ from one occurrence to the next, but 
which depends upon its value in the previous occurrence. We also frequently denote an identity 
or injection operator simply by the symbol $I$ when the underlying spaces are understood.  

Using this convention, we write $V \in L^{\frac{n}2+\e}(\bbR^n)$, and we factor $V=V_1V_2$, where
$V_j \in L^{n+\e}(\bbR^n)$, $j=1,2$.  Here,  we have assumed that $p$ is finite, 
as the case $V\in L^\infty(\bbR^n)$ is simpler and is hence omitted in the following. 
  We let $B^*$ and $A$ denote, respectively, the operators 
defined by multiplication by $V_1$ and $V_2$.  To prove the theorem, it suffices, in view of
\eqref{eq4}, to verify the hypotheses of Corollary \ref{c2.5a}, with $\cK=L^2(\bbR^n)=\cH$,
$T_0=L_a$, and $T=L_a +_{\gq} V$.  As we have noted above, Hypothesis \ref{h2.3}\,$(i)$ is the
principal theorem of \cite{AHLMT02}.  To verify the remaining hypotheses, we begin by recording 
some simple observations. We set 
\begin{equation} 
2^*:= 2n/(n-2) \text{ (which equals $\infty$ when $n=2$), and } \, 
2_*:= 2n/(n+2).  
\end{equation} 
We then have by H\"older's inequality,
\begin{equation}\label{eq5}
A: L^2(\bbR^n) \rightarrow L^{2_*+\e}(\bbR^n), \quad 
A: L^{2^*-\e}(\bbR^n) \rightarrow L^{2}(\bbR^n),
\end{equation}
and similarly for $A^*,B$ and $B^*$.  Here, when $n=2$, we interpret ``$2^*-\e$" to be an appropriate
finite exponent (to be precise, it is dual to the exponent $1+\e$).  Let 
\begin{equation} 
L^2_\alpha(\bbR^n) :=
\big\{f\in L^2(\bbR^n) \,\big|\, (-\Delta)^{\alpha/2} f\in L^2(\bbR^n)\big\}, \quad \alpha > 0,  
\end{equation} 
denote the usual fractional order Sobolev space.  For an appropriate choice of $\alpha\in (0,1)$, we have,
by \eqref{eq4}, the Hardy--Littlewood--Sobolev theorem of fractional integration (see, e.g., 
\cite[Theorem\ V.1]{St70}, 
and \eqref{eq5},  that
\begin{equation}\label{eq6}
\dom\big(L_a^{1/2}\big)=W^{1,2}(\bbR^n)\subset L^2_{\alpha} (\bbR^n)
\subset L^{2^*-\e} (\bbR^n) \cap L^2 (\bbR^n) \subset \dom(A).
\end{equation}
We obtain the same containment with $B$ in place of $A$, and with
$L_a^*$ in place of $L_a$.  Thus, \eqref{2.9} holds.  
Next, we verify \eqref{2.10} and \eqref{2.11}.  We recall that the heat semigroup
$e^{-tL_a}$ enjoys the ``hypercontractive" estimate (cf.\ \cite[Corollary\ 3.5]{Au07}), 
\begin{equation}\label{eq7}
\big\|e^{-tL_a}f \big\|_{L^q(\bbR^n)} \lesssim t^{-\frac{n}2\left(\frac1p-\frac1q\right)}\|f\|_{L^p(\bbR^n)},   
\quad 2_*\leq p\leq q \leq 2^*. 
\end{equation}
Here we used the convention that $\lesssim$ abbreviates $\leq C$ for an appropriate constant 
$C>0$. In this particular context, the underlying constant depends on 
$p$, $q$, $n$, $a_1$ and $a_2$,  see \cite{Au07}. (We remark that if $n\geq 3$, then one can 
improve this interval by an $\e$ at both endpoints, but we do not require that fact here.)
We observe that for $\tau>0$, we have by the functional calculus that
\begin{equation}\label{eq8}
(L_a+\tau I)^{-1/2}= \int_0^{\infty} \frac{ds}{s^{1/2}} \, e^{-s(L_a+\tau I)} 
= \int_0^{\infty} \frac{ds}{s^{1/2}} \, e^{-sL_a}e^{-s\tau}.
\end{equation}
Combining  \eqref{eq7} and \eqref{eq8}, we find that
\begin{align}\label{eq9}
\begin{split}
& \|(L_a+\tau I)^{-1/2}\|_{\cB(L^{2_*+\e}(\bbR^n), L^{2}(\bbR^n))} + 
\|(L_a+\tau I)^{-1/2}\|_{\cB(L^2(\bbR^n), L^{2^*-\e}(\bbR^n))}    \\
& \quad \lesssim \int_0^\infty \frac{ds}{s^{1-\e}} \, e^{-s\tau} 
\approx \tau^{-\e}.
\end{split}
\end{align}
In turn, combining the latter bound with  \eqref{eq5}, and its equivalent for $B^*$, 
and taking $\tau = \lambda +E$, we obtain \eqref{2.10} and \eqref{2.11}, with, say, $R=1$ and
$E_0=0$ as well as \eqref{2.25ccc}. This verifies Hypothesis \ref{h2.3}\,$(ii)$ and $(iii)$.

It remains to consider Hypothesis \ref{h2.1}.  Properties \eqref{2.1} and \eqref{2.2} are well-known.
Property \eqref{2.3} follows from  \eqref{eq5}, and its equivalent for $B^*$, and the fact that, in 
particular, the resolvent $(L_a - z I)^{-1}$ satisfies a hypercontractive estimate from 
$L^{2_*+\e I}(\bbR^n)$ to $L^{2^*-\e}(\bbR^n)$ (see, e.g., \cite{Au07}). Moreover, this 
hypercontractive bound holds 
with norm $\lesssim |z|^{-\e}$ for $z\in\rho(L)$ with $-\Re (z)$ large, so that, for such $z$, the 
operator norm of $K(z)$ is small; thus, property \eqref{2.4} holds as well.
\end{proof}

After focusing on the special case $\Omega = \bbR^n$, $n\in\bbN$ (which avoids the consideration 
of boundary conditions), we now turn to more general 
domains $\Omega \subset \bbR^n$ and to Dirichlet, Neumann, and mixed boundary conditions 
on $\partial \Omega$. To do so requires further assumptions on $\Omega$ 
and hence we will now assume that $\Omega$ is a {\it strongly Lipschitz domain}.  We recall that a 
strongly Lipschitz domain is a connected open set $\Omega \subseteq \bbR^n$ (bounded or unbounded), 
whose boundary consists of finitely many parts, each of which is covered by a coordinate patch on 
which the boundary coincides (in an appropriately rotated coordinate system), with the graph 
of a Lipschitz function, such that at most one of these parts is possibly unbounded. This includes, for 
instance, special Lipschitz domains (cf.\ \cite[Sect.\ VI.3.2]{St70}), bounded Lipschitz domains, and 
exteriors of the latter (see also \cite[Appendix\ B]{AR03}).

\begin{hypothesis} \lb{h3.10}
In addition to Hypothesis \ref{h3.1}, suppose $\Omega\subseteq \bbR^n$ is a strongly Lipschitz domain.
\end{hypothesis}

\begin{remark} \lb{r3.11} 
More generally, by virtue of the results in \cite{AKM06}, 
one can consider domains which arise as  bi-Lipschitz images of some smooth domain 
$\Omega' \subset \bbR^n$,  where $\Omega'$ may be $\bbR^n$, a half-space
in $\bbR^n$, or a smooth bounded sub-domain of $\bbR^n$ (or the complement,
in either $\bbR^n$ or a half space, of the closure of such). We refer the reader to \cite{AKM06} 
for details. In particular, it is possible to replace the strong Lipschitz assumption in 
Hypothesis \ref{h3.10} by the assumption in this remark consistently in all subsequent results where Hypothesis \ref{h3.10} is employed. 
\end{remark}

We note that the result of \cite{AHLMT02}, the solution of the Kato square 
root problem for divergence form operators in $\bbR^n$, has been extended to the setting 
of a strongly Lipschitz domain $\Omega$, where boundary conditions now play a non-trivial role.  
In the case that $L_a$ has Dirichlet or Neumann boundary condition, this was done in \cite{AT03}, 
using a transference procedure from the global result of \cite{AHLMT02}. The case of mixed boundary condition was treated in \cite{AKM06}. More precisely, the result of \cite{AT03}
is that for $\Omega$ satisfying Hypothesis \ref{h3.10}, Hypothesis \ref{h2.3}\,$(i)$ holds with $T_0
=L_{a,\Omega,D}$ (Dirichlet condition), or $T_0=L_{a,\Omega,N}$ (Neumann condition), 
$\cH=L^2(\Omega)$, and
\begin{align}\label{eq10}
\begin{split} 
\dom\big(L_{a,\Omega,D}^{1/2}\big) &= \dom\big((L_{a,\Omega,D}^*)^{1/2}\big) = W^{1,2}_0(\Omega), \\ 
\dom\big(L_{a,\Omega,N}^{1/2}\big) &= \dom\big((L_{a,\Omega,N}^*)^{1/2}\big) = W^{1,2}(\Omega),
\end{split} 
\end{align}
where as usual $W^{1,2}_0(\Omega)$ is the closure of $C^\infty_0(\Omega)$ in $W^{1,2}(\Omega)$.
The result of \cite{AKM06} is that 
\begin{equation}\label{eq11}
\dom\big(L_{a,\Omega,\Pi}^{1/2}\big)=\dom\big((L_{a,\Omega,\Pi}^*)^{1/2}\big) 
= W^{1,2}_{\Pi}(\Omega),
\end{equation}
where $L_{a,\Omega,\Pi}$ is defined in \eqref{4.25.3} and $\Pi \subseteq \partial \Omega$ is relatively 
open.

We will now use the strategy behind Theorem \ref{t3.6} to extend Theorem \ref{t3.9} to the setting 
of a strongly Lipschitz domain and hence prove stability of square root domains in connection with 
Dirichlet, Neumann, and mixed boundary conditions on $\partial \Omega$.

\begin{theorem}\lb{t3.12}
Assume Hypotheses \ref{h3.2} and \ref{h3.10}, let $\Pi \subseteq \partial \Omega$ be relatively open 
$($$\Pi\in \{\partial \Omega, \emptyset\}$ permitted, cf.\ Remark \ref{r3.15}\,$(ii)$$)$ and consider 
$\cW(\Omega)=W^{1,2}_{\Pi}(\Omega)$. Then the following items $(i)$ and $(ii)$ hold: \\
$(i)$ $V$ is infinitesimally form bounded with respect to $-\Delta_{\Omega,\Pi}$, and there exist constants $M_{n,p,V,\Omega}>0$ and $\e_0>0$ $($depending only on $p$, $n$ and $V$ and $\Omega$$)$ such 
that for all $0<\e<\e_0$, 
\begin{equation}\lb{4.28.12}
\begin{split}
\big\| |V|^{1/2}f \big\|_{L^2(\Omega)}^2\leq \e\big\|(-\Delta_{\Omega,\Pi})^{1/2}f \big\|_{L^2(\Omega)}^2
+ M_{n,p,V,\Omega}\e^{-n/(2p-n)}\|f\|_{L^2(\Omega)}^2,&\\
 f\in W^{1,2}_{\Pi}(\Omega).&
 \end{split}
\end{equation}
$(ii)$  $V$ is infinitesimally form bounded with respect to $L_{a,\Omega,\Pi}$ and 
\begin{align}\lb{4.29.1} 
& \big\| |V|^{1/2}f \big\|_{L^2(\Omega)}^2\leq \e \Re[\gq^{}_{a,W^{1,2}_{\Pi}(\Omega)}(f,f)] 
+ M_{n,p,V,\Omega} a_1^{-n/(2p-n)}\e^{-n/(2p-n)}\|f\|_{L^2(\Omega)}^2,   \no \\
& \hspace*{6.6cm} f\in W^{1,2}_{\Pi}(\Omega),\; 0<\e< a_1^{-1}\e_0. 
\end{align}
The form sum $L_{a,\Omega,\Pi} +_{\gq} V$ is an m-sectorial operator which satisfies
\begin{align}
\begin{split}
& \dom\big((L_{a,\Omega,\Pi} +_{\gq} V)^{1/2} \big) = \dom\big(((L_{a,\Omega,\Pi} +_{\gq} V)^*)^{1/2} \big)   \\
& \quad = \dom\big(L_{a,\Omega,\Pi}^{1/2}\big) = \dom\big((L_{a,\Omega,\Pi}^*)^{1/2}\big) 
= W^{1,2}_{\Pi}(\Omega).      \lb{4.28.13}
\end{split} 
\end{align}
\end{theorem}
\begin{proof}
Without loss of generality we put the $L^\infty$-part $V_{\infty}$ of $V$ equal to zero. 
Since $\Omega$ is assumed to be a strongly Lipschitz domain, the Stein extension theorem (cf., e.g., 
\cite[Theorem 5.24]{AF03} or \cite[Theorem 5 in \S VI.3.1]{St70}) guarantees the existence of a 
{\it total extension operator} (cf., e.g., \cite[Definition 5.17]{AF03}) $\cE$ satisfying
\begin{align}
\begin{split} 
&\cE:W^{m,2}(\Omega)\rightarrow W^{m,2}(\bbR^n),    \\
&\|\cE f\|_{W^{m,2}(\bbR^n)}^2\leq C_{m,\Omega}\|f\|_{W^{m,2}(\Omega)}^2, 
\quad f\in W^{m,2}(\Omega),\; m\in \{0,1\},    \lb{5.2.2}
\end{split} 
\end{align}
for some constants $C_{m,\Omega}>0$, $m=1,2$, and
\begin{equation}\lb{5.2.2a}
(\cE f)(x)=f(x) \, \text{ for a.e.\ }\, x\in \Omega,\; f\in L^2(\Omega).
\end{equation}

Let $V_{\text{ext}}$ and $v_{\text{ext}}$ denote the extensions of $V$ and $v$, respectively, to 
all of $\bbR^n$ defined by setting $V_{\text{ext}}$ and $v_{\text{ext}}$ identical to zero on 
$\bbR^n\backslash\Omega$.  Evidently, $V_{\text{ext}}\in L^p(\bbR^n)$, and consequently 
$V_{\text{ext}}$ satisfies the hypotheses of Lemma \ref{l3.7}.  As a result, there exists a 
constant $M_{n,p,V}>0$ for which
\begin{equation}\lb{5.2.3}
\begin{split}
\big\|v_{\text{ext}}f\big\|_{L^2(\bbR^n)}^2\leq \e \big\|(-\Delta)^{1/2} f\big\|_{L^2(\bbR^n)}^2+M_{n,p,V}\e^{-n/(2p-n)}\|f\|_{L^2(\bbR^n)}^2,&\\
f\in W^{1,2}(\bbR^n),\; \e>0.&
\end{split}
\end{equation}
Consequently, using \eqref{5.2.2} and \eqref{5.2.2a}, one estimates
\begin{align}
\|vf\|_{L^2(\Omega)}^2&=\|v_{\text{ext}}\cE f\|_{L^2(\bbR^n)}^2\no\\
&\leq \epsilon\big\|(-\Delta)^{1/2} \cE f\big\|_{L^2(\bbR^n)}^2+M_{n,p,V}
\epsilon^{-n/(2p-n)}\|\cE f\|_{L^2(\bbR^n)}^2\no\\
&=\epsilon\|\nabla \cE f\|_{L^2(\bbR^n)^n}^2+M_{n,p,V}\epsilon^{-n/(2p-n)}\|\cE f\|_{L^2(\bbR^n)}^2\no\\
&\leq \epsilon C_{1,\Omega}\|\nabla f\|_{L^2(\Omega)^n}^2+\big(\epsilon C_{1,\Omega} 
+ C_{0,\Omega}M_{n,p,V}\epsilon^{-n/(2p-n)}\big)\|f\|_{L^2(\Omega)}^2\no\\
&\leq \epsilon C_{1,\Omega}\|\nabla f\|_{L^2(\Omega)^n}^2+\big(C_{1,\Omega}+C_{0,\Omega}M_{n,p,V}\big)\epsilon^{-n/(2p-n)}\|f\|_{L^2(\Omega)}^2,\lb{5.2.4}\\
&\hspace*{6.05cm}f\in W^{1,2}_{\Pi}(\Omega), \; 0<\epsilon<1.\no
\end{align}
The form bound in \eqref{4.28.12} now follows by choosing $\e=\epsilon C_{1,\Omega}$ throughout \eqref{5.2.4} and noting that
\begin{equation}
\|\nabla f\|_{L^2(\Omega)^n}^2=\big\|(-\Delta_{\Omega,\Pi})^{1/2}f\big\|_{L^2(\Omega)}^2, 
\quad f\in W^{1,2}_{\Pi}(\Omega),
\end{equation}
by the 2nd representation theorem (cf., e.g., \cite[VI.2.23]{Ka80}), see \eqref{3.4} with $a(\cdot) = I_n$.

Finally, we note that the results of \cite{AT03} and \cite{AKM06} guarantee that (cf.\ \eqref{eq11})
\begin{equation}\lb{4.29.3aa}
\dom\big(L_{a,\Omega,\Pi}^{1/2} \big)=\dom\big((L_{a,\Omega,\Pi}^*)^{1/2} \big)=W^{1,2}_{\Pi}(\Omega).
\end{equation}
In light of \eqref{4.28.12} and \eqref{4.29.3aa}, item $(ii)$ now follows from Theorem \ref{t3.6}.
\end{proof}

\begin{remark} \lb{r3.13} 
Without providing details, we note that in order to prove Theorem \ref{t3.12} one can alternatively 
follow the proof of Theorem \ref{t3.9} essentially verbatim. The key observations are: first, that the containment \eqref{eq6} continues to hold since elements of $W^{1,2}(\Omega)$ may be extended 
to all of $\bbR^n$ as discussed in \eqref{5.2.2}, \eqref{5.2.2a}; second, that the hypercontractive 
estimate \eqref{eq7} continues to hold, with $L_{a,\Omega,D}$, $ L_{a,\Omega,N}$ or 
$L_{a,\Omega,\Pi}$ in place of the global operator $L_a$, as does the analogous hypercontractive 
estimate for the resolvent.
\end{remark}

The following result improves on the infinitesimally bounded aspects in Theorem \ref{t3.12} and 
proves the analog of input \eqref{3.9aa} in Theorem \ref{t3.6} assuming $V$ satisfies the critical 
$L^p$-index $p=n/2$, $n \geq 3$, in Hypothesis \ref{h3.3}; at the same time it extends the result 
of \cite[Theorem\ 1.8.3]{Da89} to the case of mixed boundary conditions on strongly Lipschitz domains: 

\begin{theorem}\lb{t3.14}
Assume Hypotheses \ref{h3.1} and \ref{h3.3}. Then the following items $(i)$--$(iii)$ hold, and in each case $\eta$ and $\widetilde \eta$ are non-negative functions defined on $(0,\infty)$ possibly depending on $\Omega$, $n$, and the choice of boundary condition: \\
$(i)$  $V$ is infinitesimally form bounded with respect to $-\Delta_{\Omega,D}$,
\begin{equation}\lb{5.3.3}
\big\| |V|^{1/2} f \big\|_{L^2(\Omega)}^2\leq \e \big\|(-\Delta_{\Omega,D})^{1/2}f \big\|_{L^2(\Omega)}^2 
+ \eta(\e)\|f\|_{L^2(\Omega)}^2,\quad f\in W^{1,2}_0(\Omega),\; \e>0.
\end{equation}
As a result, $V$ is infinitesimally form bounded with respect to $L_{a,\Omega,D}$,
\begin{equation}\lb{4.28.1aa}
\|vf\|_{L^2(\Omega)}^2\leq \e\Re[\gq^{}_{a,W^{1,2}_0(\Omega)}(f,f)] 
+ \widetilde \eta(\e)\|f\|_{L^2(\Omega)}^2,\quad f\in W^{1,2}_0(\Omega),\; \e>0.
\end{equation}
$(ii)$  If $n=1$ $($and hence $\Omega$ is an interval\,$)$, then $V$ is infinitesimally form bounded 
with respect to $-\Delta_{\Omega,N}$, 
\begin{equation}\lb{4.28.1}
\big\| |V|^{1/2} f \big\|_{L^2(\Omega)}^2\leq \e\big\|(-\Delta_{\Omega,N})^{1/2}f \big\|_{L^2(\Omega)}^2 
+ \eta(\e)\|f\|_{L^2(\Omega)}^2,\quad f\in W^{1,2}(\Omega),\; \e>0.
\end{equation}
As a result, $V$ is infinitesimally form bounded with respect to $L_{a,\Omega,N}$,
\begin{equation}\lb{4.28.1a}
\big\| |V|^{1/2} f \big\|_{L^2(\Omega)}^2\leq \e\Re[\gq^{}_{a,W^{1,2}(\Omega)}(f,f)] 
+ \widetilde \eta(\e)\|f\|_{L^2(\Omega)}^2,\quad f\in W^{1,2}(\Omega),\; \e>0.
\end{equation}
$(iii)$ Assume, in addition, Hypothesis \ref{h3.10}, $n\geq 3$, and let $\Pi \subseteq \partial \Omega$ 
be relatively open $($$\Pi\in \{\partial \Omega, \emptyset\}$ permitted, cf.\ Remark \ref{r3.15}\,$(ii)$$)$. 
Then $V$ is infinitesimally form bounded with respect to $-\Delta_{\Omega,\Pi}$,
\begin{equation}\lb{5.3.4}
\big\| |V|^{1/2} f \big\|_{L^2(\Omega)}^2\leq \e \big\|(-\Delta_{\Omega,\Pi})^{1/2}f \big\|_{L^2(\Omega)}^2 
+ \eta(\e)\|f\|_{L^2(\Omega)}^2,\quad f\in W^{1,2}_{\Pi}(\Omega),\; \e>0.
\end{equation}
As a result, $V$ is infinitesimally form bounded with respect to $L_{a,\Omega,\Pi}$,
\begin{equation}\lb{5.3.5}
\big\| |V|^{1/2} f \big\|_{L^2(\Omega)}^2\leq \e\Re[\gq^{}_{a,W^{1,2}_{\Pi}(\Omega)}(f,f)]
+ \widetilde \eta(\e)\|f\|_{L^2(\Omega)}^2,\quad f\in W^{1,2}_{\Pi}(\Omega),\; \e>0.  
\end{equation}
\end{theorem}
\begin{proof}
As before, we put the $L^\infty$-part $V_{\infty}$ of $V$ equal to zero. \\ 
$(i)$.~The infinitesimal form bound statement \eqref{5.3.3} may be found, for example, in \cite[Theorem 1.8.4]{Da89} in the case $n\geq 3$ under the stronger assumption that 
$\Omega$ is connected.  However, a careful look at the proof of \cite[Theorem 1.8.4]{Da89} and 
the results needed there, \cite[Theorems 1.7.1, 1.7.6, and 1.8.3]{Da89}, reveals that the connectedness assumption on $\Omega$ is not necessary.  Indeed, the principal elements of the proof of 
\cite[Theorem 1.8.4]{Da89} are the fact that $W^{1,2}_0(\Omega)$ may be embedded into 
$L^{2^*}(\Omega)$, $2^*=2n/(n-2)$ (cf., e.g., \cite[Theorem V.3.6]{EE89}), and that 
$C_0^{\infty}(\Omega)$ is a form core for $-\Delta_{\Omega,D}$; these facts do not require 
connectedness of $\Omega$.  For the case $n=1$, the form bound in \eqref{5.3.3} may be 
found, for example, in \cite[p.\ 345--346]{Ka80}.  The form bound \eqref{4.28.1aa} follows from 
the uniform ellipticity condition on $a$.  In fact, from \eqref{4.26.1} and \eqref{5.3.3}, one infers
\begin{equation}\lb{5.3.6}
\|vf\|_{L^2(\Omega)}^2\leq a_1^{-1}\epsilon\Re[\gq^{}_{a,W^{1,2}_0(\Omega)}(f,f)] 
+ \eta(\epsilon)\|f\|_{L^2(\Omega)}^2, \quad f\in W^{1,2}_0(\Omega),\; \epsilon>0.
\end{equation}
Choosing $\epsilon= a_1 \e$, $\e>0$, in \eqref{5.3.6} yields \eqref{4.28.1aa}.

\noindent 
$(ii)$.~For the form bound given in \eqref{4.28.1}, we refer once again to \cite[p.\ 345--346]{Ka80}.  Subsequently, \eqref{4.28.1} and \eqref{4.26.1} together imply
\begin{equation}\lb{5.3.7}
\|vf\|_{L^2(\Omega)}^2\leq a_1^{-1}\epsilon\Re[\gq^{}_{a,W^{1,2}(\Omega)}(f,f)]
+ \eta(\epsilon)\|f\|_{L^2(\Omega)}^2,  \quad f\in W^{1,2}(\Omega),\; \epsilon>0,
\end{equation}
and \eqref{4.28.1a} follows, choosing $\epsilon= a_1 \e$.

\noindent 
$(iii)$.~We closely follow the proof for the corresponding Dirichlet case given 
in \cite{Da89} and mentioned above in the proof of item $(i)$. By a special case of the Sobolev 
embedding theorem (cf., e.g., \cite[Theorem 4.12, Part I, Case C]{AF03}),
\begin{equation}\lb{4.28.5}
W^{1,2}(\Omega) \hookrightarrow L^{2^*}(\Omega), \quad 2^* = 2n/(n-2),
\end{equation}
where $\hookrightarrow$ abbreviates continuous (and dense) embedding, and hence there 
exists a constant $c>0$ such that
\begin{equation}\lb{4.28.6}
\|f\|_{L^{2^*}(\Omega)}^2\leq c\big(\|\nabla f \|_{L^2(\Omega)^n}^2+\|f\|_{L^2(\Omega)}^2 \big), 
\quad f \in W^{1,2}(\Omega).
\end{equation}
Using H\"older's inequality, \eqref{4.28.6} implies
\begin{equation}\lb{4.28.6a}
\begin{split}
(f,|W|f)_{L^2(\Omega)}\leq c\|W\|_{L^{n/2}(\Omega)}\big(\|\nabla f \|_{L^2(\Omega)^n}^2 
+ \|f\|_{L^2(\Omega)}^2\big),&\\
f \in W^{1,2}(\Omega), \; W\in L^{n/2}(\Omega).&
\end{split}
\end{equation}
Next, let $\e>0$ be given.  Since $V\in L^{n/2}(\Omega)$, there exist functions 
$V_{n/2,\e}\in L^{n/2}(\Omega)$ and $V_{\infty,\e}\in L^{\infty}(\Omega)$ with
\begin{equation}\lb{4.28.7}
\|V_{n/2,\e}\|_{L^{n/2}(\Omega)}\leq\e/c, \quad V(x)=V_{n/2,\e}(x)+V_{\infty,\e}(x) \, 
\text{ for a.e.\ $x\in \Omega$}.
\end{equation}
Applying \eqref{4.28.6a} with $W=V_{n/2,\e}$, one estimates
\begin{align}
\|v f\|_{L^2(\Omega)} &= (f,|V|f)_{L^2(\Omega)} 
\leq (f,[|V_{n/2,\e}|+\|V_{\infty,\e}\|_{L^{\infty}(\Omega)}]f)_{L^2(\Omega)}\no\\
&\leq \e\|\nabla f \|_{L^2(\Omega)^n}^2+\eta(\e)\|f\|_{L^2(\Omega)}^2, 
\quad f \in W^{1,2}(\Omega),    \lb{4.28.8}
\end{align}
with 
\begin{equation}\lb{4.28.9}
\eta(\e):=\e+\|V_{\infty,\e}\|_{L^{\infty}(\Omega)}.
\end{equation}
Noting the fact that 
\begin{equation}\lb{4.28.11}
\|\nabla f\|_{L^2(\Omega)^n}^2=\big\|(-\Delta_{\Omega,\Pi})^{1/2}f \big\|_{L^2(\Omega)}^2, 
\quad f \in W_{\Pi}^{1,2}(\Omega),
\end{equation}
by the 2nd representation theorem (cf., e.g., \cite[Theorem VI.2.23]{Ka80}), and using 
that $W^{1,2}_{\Pi}(\Omega)$ is a closed subspace of $W^{1,2}(\Omega)$, \eqref{4.28.8} then 
also yields 
\begin{equation}\lb{4.28.10}
\|vf\|_{L^2(\Omega)}^2\leq \e\|\nabla f\|_{L^2(\Omega)^n}^2 
+ \eta(\e)\|f\|_{L^2(\Omega)}^2,\quad f \in W_{\Pi}^{1,2}(\Omega).
\end{equation}
Since $\e>0$ was arbitrary, \eqref{5.3.4} follows.

To prove \eqref{5.3.5}, one notes that the uniform ellipticity condition on $a$ implies
\begin{equation}\lb{5.3.1}
\|\nabla f\|_{L^2(\Omega)^n}^2 \leq a_1^{-1} \Re[\gq^{}_{a,W^{1,2}_{\Pi}(\Omega)}(f,f)], 
\quad f\in W^{1,2}_{\Pi}(\Omega).
\end{equation}
Taking \eqref{5.3.1} together with \eqref{5.3.4} and \eqref{4.28.11}, one infers that
\begin{equation}\lb{5.3.2}
\|vf\|_{L^2(\Omega)}^2\leq a_1^{-1}\epsilon\Re[\gq^{}_{a,W^{1,2}_{\Pi}(\Omega)}(f,f)]
+\eta(\epsilon)\|f\|_{L^2(\Omega)}^2, \quad f\in W^{1,2}_{\Pi}(\Omega),\; \epsilon>0.
\end{equation}
The form bound in \eqref{5.3.5} follows by taking $\epsilon = a_1 \e$, $\e>0$, in \eqref{5.3.2}. 
\end{proof}

\begin{remark} \lb{r3.15} 
$(i)$ The main idea for the proof of Theorem \ref{t3.14}, is taken from the proof of the corresponding 
Dirichlet case in \cite[Theorem\ 1.8.3]{Da89}. \\
$(ii)$ Due to our strongly Lipschitz hypothesis on $\Omega$ in Theorems \ref{t3.12} and \ref{t3.14}\,$(iii)$, 
\begin{equation} 
W^{1,2}_{\partial \Omega}(\Omega)=W^{1,2}(\Omega),  
\end{equation} 
as noted in Remark \ref{r3.5}\,$(iii)$.  Consequently, taking $\Pi=\partial \Omega$ in Theorems \ref{t3.14} 
and \ref{t3.12} shows that \eqref{4.28.1} and \eqref{4.28.12} (resp., \eqref{4.29.1} and \eqref{4.28.13}) 
hold for the Neumann-type Laplacian $-\Delta_{\Omega,N}$ (resp., $L_{a,\Omega,N}$).  Alternatively, 
taking $\Pi=\emptyset$ yields  
\begin{equation} 
W^{1,2}_{\emptyset}(\Omega)=W_0^{1,2}(\Omega),  
\end{equation} 
and shows that \eqref{4.28.1} and \eqref{4.28.12} (resp., \eqref{4.29.1} and 
\eqref{4.28.13}) hold for the Dirichlet-type Laplacian $-\Delta_{\Omega,D}$ (resp., 
$L_{a,\Omega,D}$). \\
$(iii)$ Estimates of the type \eqref{5.3.3}-- \eqref{5.3.5} can also be used to include 
first-order terms and hence consider the operator 
$-\dv (a\nabla \, \cdot \,) + \big(\B_1\cdot \nabla \cdot \big) + \dv \big(\B_2 \cdot \big) +V$ 
on $\Omega$, assuming 
$\B_j \in L^{n + \varepsilon}(\Omega)^n + L^\infty(\Omega)^n$, $j=1,2$, for some $\varepsilon > 0$. 
Rather than repeating the analysis for this more general situation, we refer to our detailed treatment 
of the critical $L^p$-index for the case $\Omega = \bbR^n$ in Section \ref{s4} (see also item $(iv)$ 
below). \\
$(iv)$ In contrast to Theorem \ref{t3.12}, where the analogs of $\eta (\varepsilon)$ and 
$\wti \eta (\varepsilon)$ are explicitly provided in terms of $C \varepsilon^{-n/(2p-n)}$, the extension  
to the critical $L^p$-index $p=n/2$ in Theorem \ref{t3.14} no longer permits one to determine 
explicit estimates for $\eta (\varepsilon)$ and $\wti \eta (\varepsilon)$. Consequently, we can no 
longer use \eqref{3.9} in Theorem \ref{t3.6}\,$(ii)$ to conclude the stability of square root domains as 
in \eqref{4.28.13}. A systematic extension of this circle of ideas to the square root problem for 
the more general situation 
\begin{equation} 
-\dv (a\nabla \, \cdot \,) + \big(\B_1\cdot \nabla \cdot \big) + \dv \big(\B_2 \cdot \big) +V \,  
\text{ on } \, \bbR^n 
\end{equation} 
at the critical $L^p$-index will be undertaken in the following Section \ref{s4}. \\ 
$(v)$ Introducing the set 
\begin{equation}\lb{4.28.3}
\cD_{\Pi}(\Omega) = 
\{f|_{\Omega} \; | \; f \in C_0^{\infty}(\bbR^n); \, 
\supp \, (f) \cap (\partial \Omega \backslash \Pi) = \emptyset\} \subset W^{1,2}(\Omega),
\end{equation}
one recalls that $\|\cdot\|_{W^{1,2}(\Omega)}$ and $\|\cdot\|_{\gq^{}_{a,\cW(\Omega)}}$ are 
equivalent norms on $\cW(\Omega)$ (cf.\ Remark \ref{r3.4}\,$(ii)$), and hence on 
$W_{\Pi}^{1,2}(\Omega)$. Thus, one infers that 
\begin{equation}\lb{4.28.4}
\text{$\cD_{\Pi}(\Omega)$ is a form core for $-\Delta_{\Omega,\Pi}$ and 
$L_{a,\Omega,\Pi}$, and hence also for $L_{a,\Omega,\Pi} +_{\gq} V$,}
\end{equation} 
employing the infinitesimal form boundedness in \eqref{5.3.5} (see also 
\cite[Theorem\ IV.5.1]{EE89}, \cite[Theorems\ VI.1.21, VI.1.33]{Ka80}). 
\end{remark}

\section{The square root problem for $-\dv (a\nabla \, \cdot \,) + \big(\B_1\cdot \nabla \cdot \big) 
+ \dv \big(\B_2  \cdot \big) + V$ \\ 
at the critical index}\label{s4}

In this section we extend Theorem \ref{t3.9} to include first-order differential operator terms  
of the type $\big(\B_1\cdot \nabla \cdot \big) + \dv\big(\B_2 \cdot \big)$ in addition to the zero-order 
potential term $V$ and we impose the critical (i.e., optimal) $L^p$-conditions on $\B_j$, $j=1,2$, 
and $V$. 

For this purpose we introduce the following hypotheses:

\begin{hypothesis} \lb{h4.1}
In addition to Hypothesis \ref{h3.1} with $\Omega = \bbR^n$, $n\in\bbN$, $n \geq 2$, suppose that 
\begin{equation} 
\B_1,\B_2 \in L^n(\bbR^n)^n + L^\infty(\bbR^n)^n, \quad V\in L^{n/2}(\bbR^n) + L^\infty(\bbR^n).    
\lb{4.1} 
\end{equation} 
\end{hypothesis} 

Setting again $L_a:= -\dv (a\nabla \, \cdot \,)$, the operator formally given in \eqref{eq3} and 
precisely defined by \eqref{4.25.4}, we now consider the operator 
\begin{equation}\label{eq4.1}
M_a := L_a  +_{\gq} \big(\B_1 \cdot \nabla \cdot \big) +_{\gq} \dv \big(\B_2 \cdot \big) +_{\gq} V,
\end{equation}
in $L^2(\bbR^n)$.   We recall that we use $a_1, a_2$ to denote the ellipticity parameters of the coefficient matrix $a$
(see \eqref{3.2} above).

Decomposing $\B_j$, $j=1,2$, and $V$ as
\begin{align}
& \B_j = \B_{j,n} + \B_{j,\infty}, \quad 
\big\| \B_{j,n} \big\|_{L^n(\bbR^n)^n} <\infty, \quad \big\| \B_{j,\infty} \big\|_{L^{\infty}(\bbR^n)^n} < \infty, \; j=1,2,    \no \\
& V = V_{n/2} + V_{\infty}, \quad 
\|V_{n/2}\|_{L^{n/2}(\bbR^n)} <\infty, \quad 
\|V_{\infty}\|_{L^{\infty}(\bbR^n)} < \infty,     \lb{4.3ccc}
\end{align}
proves that under hypothesis 
\eqref{4.1}, $\big(\B_1 \cdot \nabla \cdot \big)$, $\dv \big(\B_2 \cdot \big)$, and $V$ are all 
infinitesimally form bounded with respect to $L_a$, rendering the form sum $M_a$ in \eqref{eq4.1} 
well-defined.  

For simplicity of notation, we will henceforth assume without loss of generality that the $L^\infty$-parts 
of $\B_j$, $j=1,2$, and $V$ in \eqref{4.3ccc} equal zero. (It will be clear from the proof of Theorem \ref{t4.5} 
that these $L^\infty$-parts can be handled in a straightforward manner.)

We start with the following  
result, in which we suppose that $\B_1$, $\B_2$, 
and $V$ are small in the corresponding $L^p$-norms, as in \eqref{eq4.smallness} below:

\begin{theorem} \lb{t4.2} 
Assume Hypothesis \ref{h4.1} with $n \in \bbN$, $n\geq 3$, and let $L_a$ and $M_a$ 
be defined as in \eqref{eq4.1} on $\bbR^n$, with  
$\B_j \in L^n(\mathbb{R}^n)$, $j=1,2$,  
$V \in L^{n/2}(\bbR^n)$, supposing in addition the smallness condition
\begin{equation}\label{eq4.smallness}
\big\|\B_{1}\big\|_{L^n(\bbR^n)^n} + \big\|\B_{2}\big\|_{L^n(\bbR^n)^n} 
+ \|V\|_{L^{n/2}(\bbR^n)} \leq \eps_0.
\end{equation}
If $0 <\eps_0 = \eps_0 (n,a_1,a_2)$ is sufficiently small, then  
\begin{equation}\lb{eq4.2}
\dom(M_{a}^{1/2}) = W^{1,2}(\bbR^n).
\end{equation}
Moreover,
\begin{equation}\label{eq4.3}
\|M_{a}^{1/2} f\|_{L^2(\bbR^n)} \leq C \|\nabla f\|_{L^2(\bbR^n)^n}, 
\quad f \in W^{1,2}(\bbR^n), 
\end{equation}
where $C = C\big(n,a_1,a_2, \big\|\B_{1}\big\|_{L^n(\bbR^n)^n},
\big\|\B_{2}\big\|_{L^n(\bbR^n)^n}, \|V\|_{L^{n/2}(\bbR^n)}\big)$.  

Finally, these results continue to hold when $n=2$, provided that $\B_{j}$, 
$j=1,2$, are divergence free, and that $V \equiv 0$.
\end{theorem} 

Later in this section we will remove the smallness condition \eqref{eq4.smallness}  
(cf.\ Theorem \ref{t4.5}). 

\begin{remark} \lb{r4.3}
$(i)$ One observes that if $\B$ is divergence free, then $\dv(\B f)= \B\cdot \nabla f$.  Thus, 
in the case $n=2$ it is equivalent to consider operators of the form $M_a=L_a+ \B\cdot\nabla$, with
$\big\|\B\big\|_{L^n(\bbR^n)^n} \leq \eps_0$.  \\
$(ii)$  Set $R:= M_a-L_a= (\B_1 \cdot \nabla \, \cdot \,) + (\dv \B_2 \, \cdot \,) + V$.
Then by H\"older's inequality, \eqref{eq4.smallness}, and Sobolev embedding,
we have for $n\geq 3$, with $\gq^{}_R(\cdot, \cdot)$ the sesquilinear form associated with $R$, 
and the decomposition $V = u v$, 
\begin{align}\label{eq4.formbound} 
\begin{split} 
|\gq^{}_R(f,f)| &=  \big|\big(\B_1^* f, \nabla f\big)_{L^2(\bbR^n)^n} -
 \big(\nabla  f, \B_2 f\big)_{L^2(\bbR^n)^n} + (u^* f, v f)_{L^2(\bbR^n)}\big|    \\
 & \lesssim \eps_0 \|\nabla f\|_{L^2(\bbR^n)^n}^2.
 \end{split} 
 \end{align}
The analogous bound continues to hold for $n=2$, but is slightly more delicate.
 In that case, by Remark \ref{r4.3}\,$(i)$, we have that $R f =\B\cdot\nabla f$,
 which belongs to the Hardy space $H^1(\mathbb{R}^2)$, 
 since $\B$ is divergence free.  Indeed, by the result of \cite{CLMS89}, 
 and \eqref{eq4.smallness},  we have
  \begin{equation}\label{eq2dclms}
  \big\|\B\cdot\nabla f\big\|_{H^1(\bbR^2)}\leq C\eps_0\|\nabla f\|_{L^2(\bbR^2)^2}.
  \end{equation}
  Moreover, the endpoint Sobolev embedding result
  in $\mathbb{R}^2$ yields that
  \begin{equation}\label{eq2dsobolev}
  \|f\|_{BMO(\bbR^2)} \leq C\|\nabla f\|_{L^2(\bbR^2)^2}.
  \end{equation}
  Thus, by Fefferman's duality theorem \cite{FS72}, one has 
  \begin{equation}\label{eq4.formbound2d}
|\gq^{}_R(f,f)| = \big|\big(\B^* f, \nabla f\big)_{L^2(\bbR^2)^2}\big| 
\lesssim \eps_0\|\nabla f\|_{L^2(\bbR^2)^2}^2.
\end{equation}\\
$(iii)$ We note
that, as a consequence of $(ii)$, for $\eps_0$ sufficiently small, depending on the ellipticity of
$a$ (see \eqref{3.2}),  the operator $M_a$ and its Hermitian
adjoint $M_a^*$ are one-to-one and m-accretive, by  
the theory of sectorial forms and m-sectorial operators (cf. 
\cite[Theorem \ V.3.2, 
Theorem\ VI.1.33, Theorem\ VI.2.1]{Ka80}).   
\end{remark}

The proof of Theorem \ref{t4.2} will be deduced as an immediate consequence
of Theorem \ref{t4.5a} below, and the following pair of lemmata,
the first of which concerns resolvent estimates for $M_a$. 
We set $2_*:=2n/(n+2)$, $2^*:=2n/(n-2)$ as before,
and more generally, introduce $p_*:=pn/(n+p)$, $p^*: = pn/(n-p)$, and 
$p^{**}: = pn/(n-2p)$.
(For simplicity, we abbreviate $I_{L^2(\bbR^n)}$ by $I$ in this section.)

\begin{lemma}\label{l4.4}   
For $0 < \eps_0 = \eps_0(n, a_1, a_2)$ sufficiently small in \eqref{eq4.smallness}, 
the resolvent operator $(I +t^2M_a)^{-1}$ 
enjoys the following properties. First, if $n\geq 3$, then 
\begin{equation}\label{eq4.resolvent}
\big\|(I +t^2M_a)^{-1}\big\|_{\cB(L^p(\bbR^n),L^2(\bbR^n))} \leq C_{p}\, t^{-n(1/p-1/2)}  \,,\quad 
2_*-\eps_1 < p\leq 2\,, 
\end{equation} 
for some $0 < \eps_1 = \eps_1(n, a_1, a_2)$. When $n=2$, the previous bound continues to 
hold with $2_* \leq p\leq 2$, where the endpoint case $p=2_*=1$ is interpreted as
\begin{equation}\label{eq4.resolvent2d}
\big\|(I +t^2M_a)^{-1}\big\|_{\cB(H^1(\bbR^2),L^2(\bbR^2))} \leq C\, t^{-1}  \,. 
\end{equation} 
In addition, for $n\geq 2$, 
\begin{equation}\label{eq4.resolventgrad}
\big\|t\nabla(I +t^2M_a)^{-1}\big\|_{\cB(L^{2}(\bbR^n),L^2(\bbR^n)^n)} 
\leq C \,, 
\end{equation} 
and for $n\geq 3$,
\begin{equation}\label{eq4.resolvent2}
\big\|t(I +t^2M_a)^{-1}\dv\big\|_{\cB(L^{q}(\bbR^n)^n,L^2(\bbR^n))} \leq C_{q}\, t^{-n(1/q-1/2)}  \,,\quad 
q=2-\eps_2 \,, 
\end{equation} 
for some $0 < \eps_2 = \eps_2(n, a_1, a_2)$. Moreover, for every pair of cubes 
$Q_1,Q_2\subset \bbR^n$, with side lengths $\ell(Q_1),\ell(Q_2)$ satisfying 
\begin{equation} 
0<t\leq 4 \ell(Q_j), \; j=1,2, \text{ and } \, \max(\ell(Q_1),\ell(Q_2)\leq 100 \dist(Q_1,Q_2), 
\end{equation} 
for  $n\geq 3$ we have the ``Gaffney-type" estimate
\begin{align} \label{eq4.gaff}
& \big\|(I +t^2M_a)^{-1}\big\|_{\cB(L^2(Q_1),L^2(Q_2))} 
+ \big\|t\nabla(I +t^2M_a)^{-1}\big\|_{\cB(L^2(Q_1),L^2(Q_2)^n)} \no \\
& \qquad + \big\|t(I +t^2M_a)^{-1}\dv\big\|_{\cB(L^2(Q_1)^n,L^2(Q_2))}    \no \\
& \quad \leq C e^{-c\dist(Q_1,Q_2)/t}, \quad k \in \bbN.
\end{align} 
For $n = 2$, 
\begin{align}\label{eq4.gaff2d}
& \big\|(I +t^2M_a)^{-1}\big\|_{\cB(L^2(Q_1),L^2(Q_2))} 
+ \big\|t\nabla(I +t^2M_a)^{-1}\big\|_{\cB(L^2(Q_1),L^2(Q_2)^2)}  \no \\
& \qquad + \big\|t(I +t^2M_a)^{-1}\dv\big\|_{\cB(L^2(Q_1)^2,L^2(Q_2))}    \no \\ 
& \quad \leq C \left(\frac{t}{\dist(Q_1.Q_2)}\right)^4, \quad k \in \bbN.
\end{align} 
 The various constants $c,\,C,\,C_p$ and $C_{q}$
in estimates \eqref{eq4.resolvent}--\eqref{eq4.gaff2d}  
depend upon $n$, $a_1$, and $a_2$.  
\end{lemma}

We remark that there is nothing special about the 4th order decay in the two-dimensional Gaffney 
estimate \eqref{eq4.gaff2d}: indeed, in the proof of this part of Lemma \ref{l4.4},
we could obtain polynomial decay of 
any specified order $N$, by choice of $\eps_0$ small enough depending on $N$, 
but as it turns out, 4th order decay will suffice for our purposes here in two dimensions.

\begin{lemma}\label{l4.5*}  Set $H_a := L_a +\B_1\cdot \nabla = -\dv a\nabla +\B_1\cdot\nabla$,
and suppose that $\B_1\in L^n(\bbR^n)^n$, with 
$\big\|\B_1\big\|_{L^n(\mathbb{R}^n)^n}\leq \eps_0$.
If $\eps_0$ is small enough, depending only on dimension $n$ and the ellipticity of $a$, then 
\begin{equation}\label{eq4.11*} 
\big\|\nabla H_a^{-1}\dv \vec{f} \big\|_{L^2(\mathbb{R}^n)^n} 
\leq C \big\|\vec{f}\big\|_{L^2(\mathbb{R}^n)^n}\,,
\end{equation}
where $C = C\big(n, a_1,a_2, \big\|\B_1\big\|_{L^n(\mathbb{R}^n)^n}\big)$.
\end{lemma}

Let us take Lemma \ref{l4.4} and Lemma \ref{l4.5*} for granted momentarily.  We shall defer their proofs to 
the end of this section. Before proceeding further, let us note for future reference the following 
immediate consequences of \eqref{eq4.gaff}. First, by  
\eqref{eq4.gaff} and Sobolev's inequality, we have for $n\geq 3$ that
\begin{equation}\label{eq4.gaffpq}
\big\|(I +t^2M_a)^{-1}\big\|_{\cB(L^{2_*}(Q_1),L^2(Q_2))} 
\lesssim t^{-1}e^{-c\dist(Q_1,Q_2)/t}, 
\end{equation} 
where $c$ and the implicit constant 
depend on $n$, $a_1$, $a_2$, $\big\|\B_1\big\|_{L^n(\bbR^n)^n}$, 
$\big\|\B_2\big\|_{L^n(\bbR^n)^n}$, and
$\|V\|_{L^{n/2}(\bbR^n)}$. Next, given a
cube $Q\subset \bbR^n$, set $S_k(Q):=2^{k+1}Q\backslash 
2^kQ$, and let $0<t\leq 4 \ell(Q)$.
We then have the following variant of the ``Gaffney" estimate, for $n\geq 3$:
\begin{align}\label{eq4.gaffsk}
& \big\|(I +t^2M_a)^{-1}\big\|_{\cB(L^2(S_k(Q)),L^2(Q))} 
+ \big\|t\nabla(I +t^2M_a)^{-1}\big\|_{\cB(L^2(S_k(Q)),L^2(Q)^n)} \no \\
& \qquad + \big\|t(I +t^2M_a)^{-1}\dv\big\|_{\cB(L^2(S_k(Q))^n,L^2(Q))}    \no \\
& \quad \leq C e^{-c2^k\ell(Q)/t}, \quad k \in \bbN,
\end{align} 
as one may deduce from \eqref{eq4.gaff} by covering $S_k(Q)$ by a bounded number of cubes $Q'$ of
side length $\ell(Q')\approx 2^k\ell(Q)$, with $\dist(Q,Q')\approx 2^k\ell(Q)$.  Similarly, from
\eqref{eq4.gaffpq} we obtain for $n\geq 3$ that
\begin{equation}\label{eq4.gaffpqsk}
\big\|(I +t^2M_a)^{-1}\big\|_{\cB(L^{2_*}(S_k(Q)),L^2(Q))} 
\lesssim\, t^{-1}e^{-c2^k\ell(Q))/t}  \,. 
\end{equation}
When $n=2$, we have
\begin{align}\label{eq4.gaffsk2d}
& \big\|(I +t^2M_a)^{-1}\big\|_{\cB(L^2(S_k(Q)),L^2(Q))} 
+ \big\|t\nabla(I +t^2M_a)^{-1}\big\|_{\cB(L^2(S_k(Q)),L^2(Q)^2)} \no \\
& \qquad + \big\|t(I +t^2M_a)^{-1}\dv\big\|_{\cB(L^2(S_k(Q))^2,L^2(Q))}   \no \\
& \quad \leq C \left(\frac{t}{2^k\ell(Q)}\right)^4, \quad k \in \bbN. 
\end{align} 

We now observe that, assuming Lemmas \ref{l4.4} and \ref{l4.5*}, and recalling 
Remark \ref{r4.3}\,$(iii)$,
we obtain the conclusion of Theorem \ref{t4.2} as an immediate consequence of
the following theorem, which is the deep fact underlying the results of this section. 

\begin{theorem} \lb{t4.5a} 
For $n\geq 3$,  let $L_a$ and $M_a$ 
be defined as in \eqref{eq4.1} on $\bbR^n$, with  
$\B_j \in L^n(\mathbb{R}^n)^n$, $j=1,2$,  and
$V \in L^{n/2}(\bbR^n)$. 
We suppose in addition that the operator $M_a$ and its adjoint $M_a^*$ are one-to-one 
and m-accretive,  that
the resolvent bounds of Lemma \ref{l4.4} hold, and that $H_a:= L_a +\B_1\nabla$ satisfies \eqref{eq4.11*}.
Then  
\begin{equation}\lb{eq4.2a}
\dom(M_{a}^{1/2}) = W^{1,2}(\bbR^n).
\end{equation}
Moreover,
\begin{equation}\label{eq4.3a}
\big\|M_{a}^{1/2} F\big\|_{L^2(\bbR^n)} \leq C \|\nabla F\|_{L^2(\bbR^n)^n}, 
\quad F \in W^{1,2}(\bbR^n), 
\end{equation}
where $C = C\big(n,a_1,a_2, \big\|\B_{1}\big\|_{L^n(\bbR^n)^n},
\big\|\B_{2}\big\|_{L^n(\bbR^n)^n}, \|V\|_{L^{n/2}(\bbR^n)}\big)$.  

Finally, these results continue to hold when $n=2$, provided that $\B_{j}$, 
$j=1,2$, are divergence free, and that $V \equiv 0$.
\end{theorem} 

Let us point out that smallness of $\B_j \in L^n(\mathbb{R}^n)^n$, $j=1,2$,  and
$V \in L^{n/2}(\bbR^n)$, is not used in any explicit way in the proof of Theorem \ref{t4.5a}.
On the other hand, we do not know how to establish m-accretivity of $M_a$, nor the conclusions of
Lemmas \ref{l4.4} and \ref{l4.5*}, in the absence of smallness. 

\begin{proof}[Proof of Theorem \ref{t4.5a}]  
The proof will
utilize the technology of the solution 
of the square root
problem for $L_a$ developed in \cite{AHLMT02}, \cite{HLM02}, and \cite{HM02}. 
Let us now proceed to the details.

We resolve $M_a^{1/2}$ as follows:
\begin{equation}\label{eq4.6}
M_a^{1/2} =c\int_0^\infty dt \, (I+t^2M_a)^{-3}\,t^3M^2_a t^{-1} \,.
\end{equation} 
It is enough to show that for $g,F\in C^\infty_0(\bbR^n)$, we have the estimate
\begin{equation}\label{eq4.15}
|(g, M_a^{1/2} F)_{L^2(\bbR^n)}| \leq C \|g\|_{L^2(\bbR^n)} \|\nabla F\|_{L^2(\bbR^n)^n}.
\end{equation} 
By \eqref{eq4.6} and Cauchy-Schwarz, it is enough to establish the square function
estimates
\begin{equation}\label{eq4.16a}
\iint_{\mathbb{R}^{n+1}_+}\frac{d^nxdt}{t} \left|t^2 M_a^*\left(I+t^2M_a^*\right)^{-2} g\right|^2 
\leq C \|g\|^2_{L^2(\bbR^n)}.
\end{equation} 
and
\begin{equation}\label{eq4.16}
\iint_{\mathbb{R}^{n+1}_+} \frac{d^n xdt}{t} \left|t \left(I+t^2M_a\right)^{-1} M_a F\right|^2 \leq 
C \|\nabla F\|^2_{L^2(\bbR^n)^n},
\end{equation} 
where $C = C\big(n,a_1,a_2, \big\|\B_{1}\big\|_{L^n(\bbR^n)^n},
\big\|\B_{2}\big\|_{L^n(\bbR^n)^n}, \|V\|_{L^{n/2}(\bbR^n)}\big)$, and where, as above, 
$a_1,a_2$ are the ellipticity parameters for the coefficient matrix $a$ (see \eqref{3.2}). 
Since $M_a$ is a 
one-to-one maximal accretive operator (by hypothesis),
it therefore has a bounded holomorphic functional calculus (see 
\cite[Theorem\ G, p.\ 102]{ADM96}), whence
estimate \eqref{eq4.16a} follows directly from the McIntosh--Yagi Theorem (see, e.g.,
\cite[Theorem F, pp.\ 98--99]{ADM96} or \cite{Mc86}.) We therefore turn to the proof of 
\eqref{eq4.16}, which is the heart of the matter.
Let us suppose for now that $n\geq 3$;  we shall discuss the modifications required to treat the 
two-dimensional case at the end of the proof.
 
Recalling that
$2_*:=2n/(n+2)$, $2^*:=2n/(n-2)$, we record several simple observations:
\begin{align}\label{eq4.10}
\|F\|_{L^{2^*}(\bbR^n)}& \leq C\|\nabla F\|_{L^2(\bbR^n)^n},   \\
\label{eq4.11}
\|VF\|_{L^{2_*}(\bbR^n)} & \leq \|V\|_{L^{n/2}(\bbR^n)} \|F\|_{L^{2^*}(\bbR^n)},\\
\label{eq4.12}
\big\|\vec{B}_1\cdot\vec{f}\big\|_{L^{2_*}(\bbR^n)} & \leq \big\|\B_{1}\big\|_{L^n(\bbR^n)^n} \big\|\vec{f}\big\|_{L^{2}(\bbR^n)^n},\\
\label{eq4.12a}
\big\|\vec{B}_2 F\big\|_{L^{2}(\bbR^n)^n} & \leq \big\|\B_{2}\big\|_{L^n(\bbR^n)^n} \|F\|_{L^{2^*}(\bbR^n)}.
\end{align}
The first of these is just Sobolev embedding, and the other three follow directly from 
H\"older's inequality, 

We recall that $L_a:=-\dv a\nabla$, and we set 
\begin{equation} 
H_a:= L_a +\B_1\nabla\,, 
\end{equation} 
so that 
\begin{align}
M_a:&= -\dv a\nabla + \B_1\nabla + \dv \B_2 + V   \no \\  
&= -\dv a\nabla + \B_1\nabla  +\dv a\big( a^{-1} \B_2 - \nabla L_a^{-1} V\big)    \no \\
& =H_a +\dv a\big( a^{-1} \B_2 -\nabla L_a^{-1} V\big)\,.    
\end{align}
Let $\nabla F\in L^2(\bbR^n)^n$.  By \eqref{eq4.11} and \eqref{eq4.12a}, 
$VF \in L^{2_*}(\bbR^n)$ and $\B_2 F\in L^2(\bbR^n)^n$.
Moreover,
by the Hodge decomposition for $L_a$, there is Sobolev-type embedding:
\begin{equation}\label{eq2.0} \|\nabla(L_a)^{-1} h\|_2 \lesssim  \|h\|_{2_*}\,,
\end{equation}
and therefore, $\nabla L_a^{-1} V F \in L^2(\bbR^n)^n$.  Hence, using again the Hodge decomposition for $L_a$, we can write
\begin{equation} 
\dv a\big(a^{-1}\B_2 F-\nabla L_a^{-1} V F\big) =- \dv a\nabla G = L_a G\,, 
\end{equation} 
so that
\begin{equation} 
M_a F = H_a F +L_a G\,,
\end{equation} 
where
\begin{equation} 
\|\nabla G\|_2 \lesssim \bigg(\big\|\B_2\big\|_n +\|V\|_{n/2}\big) \|\nabla F\|_2\,. 
\end{equation} 
Consequently, it is enough to prove the following pair of estimates:
\begin{equation}\label{eq4.29a}
\iint_{\mathbb{R}^{n+1}_+} \frac{d^n xdt}{t} \left|t \left(I+t^2M_a\right)^{-1} H_a F\right|^2 \leq 
C \|\nabla F\|^2_{L^2(\bbR^n)^n}\,,
\end{equation}
and
\begin{equation}\label{eq4.30a}
\iint_{\mathbb{R}^{n+1}_+} \frac{d^n xdt}{t}
\left|t \left(I+t^2M_a\right)^{-1} L_a G\right|^2 \leq 
C \|\nabla G\|^2_{L^2(\bbR^n)^n}\,.
\end{equation}

Before proceeding further, let us point out that the proof of  \eqref{eq4.30a}
may be reduced to that of \eqref{eq4.29a}, and therefore it will suffice to establish the latter.
Indeed, we may write
\begin{equation} 
L_a G = H_a G -\B_1\nabla G =:I-II\,.
\end{equation} 
Clearly, the contribution of term $I$ is of the same form as \eqref{eq4.29a}.
Moreover, 
\begin{equation} 
II = H_a H_a^{-1} \B_1 \nabla G =: H_a \widetilde{F}\,, 
\end{equation} 
where by \eqref{eq4.11*}, \eqref{eq4.12} and Sobolev embedding, $\nabla \widetilde{F}
= \nabla H_a^{-1} \B_1 \nabla G \in L^2(\bbR^n)^n$.  Thus, the contribution
of term $II$ is also of the same form as \eqref{eq4.29a}.

We therefore turn to  the proof of \eqref{eq4.29a}.
The first step 
is a standard reduction to a Carleson measure estimate, following the analogous 
step, with some modifications owing to the presence of lower order terms,
in the treatment of the square root problem for $L_a$ in \cite{AHLMT02}, \cite{AHLT01}, \cite{AT98},
\cite{HLM02}, \cite{HM02}. 

Before proceeding, we introduce two standard averaging operators that will be used in the sequel.
Let $P_t$ be a  ``nice" approximate identity, that is,
\begin{equation} 
P_t f:=\phi_t*f\,,
\end{equation}  
with $\phi_t(x):=t^{-n}\phi(x/t)$, for some non-negative, radial 
$\phi\in C^\infty_0(\mathbb{R}^n)$, supported in the unit ball in $\mathbb{R}^n$, with integral equal to 
one. Let $E_{t}$ denote the dyadic
averaging operator, that is, if $Q(x,t)$ is the minimal dyadic cube (with respect to the grid induced by $Q$) containing $x$,
with side length at least $t$, then 
\begin{equation} (E_{t}g)(x):=\frac1{|Q(x,t)|}\int_{Q(x,t)} d^n y \, g(y).
\end{equation}

Returning to the proof,  we let $\theta_t$ denote the operator defined by
\begin{equation}\label{eq4.thetadef}
\theta_t\vec{f}:= -t \left(I+t^2M_a\right)^{-1}(\dv(a) - \vec{B}_1)\cdot\vec{f}, 
\end{equation}
and observe that therefore
\begin{equation}\label{eq4.thetadef2}
t \left(I+t^2M_a\right)^{-1}H_a =\theta_t\nabla\,.
\end{equation}
Using a strategy that originates in \cite{CM86}, we may then write
\begin{align}\label{eq4.19a}
& t \big(I+t^2M_a\big)^{-1}H_aF(x) = \theta_t \nabla F(x)    \no \\
& \quad =\big(\theta_t -\theta_tI_n(x)
\cdot E_tP_t\big)\nabla F(x) + \theta_tI_n(x)
\cdot E_tP_t(\nabla F)(x)      \no \\
& \quad =: R_t (\nabla F)(x) + \theta_tI_n(x)
\cdot E_tP_t(\nabla F)(x)=:I+II, 
\end{align}
where $I_n$ is the $n\times n$ identity matrix, and where $E_t$ and $P_t$ have been defined above.  
The key term here is $II$.  By Carleson's embedding lemma, the desired bound for $II$
follows immediately from the Carleson measure estimate
\begin{equation}\label{eq4.carl}
\sup \frac1{|Q|} \iint_{R_Q} \frac{d^n xdt}{t} |(\theta_t I_n) (x)|^2  \leq C,
\end{equation}
where the supremum runs over all  cubes $Q\subset \mathbb{R}_n$, and where
$R_Q:= Q\times (0,\ell(Q))$ is the standard ``Carleson box" above $Q$.
To establish \eqref{eq4.carl} is our main task, and we postpone the proof momentarily.

Turning to the term $I$, we observe that by a standard ``Schur's lemma" argument, 
and standard Littlewood--Paley theory for the classical heat semigroup, to obtain 
the desired bound for $R_t$, that is, to show that
\begin{equation}\label{eq4.21}
\iint_{\mathbb{R}^{n+1}_+}\frac{d^nxdt}{t} \left|R_t\nabla F\right|^2 \leq 
C \|\nabla F\|^2_{L^2(\bbR^n)^n},
\end{equation}  it suffices to establish the quasi-orthogonality estimate
\begin{equation}\label{eq4.18}
\big\|R_t \nabla\, s^2 \Delta e^{2s^2\Delta} F\big\|_{L^2(\bbR^n)} 
\leq C \min\left(\left(\frac{s}{t}\right)^\beta,
\left(\frac{t}{s}\right)^\beta\right) \big\|s\,e^{s^2\Delta}\, \nabla\cdot \nabla F \big\|_{L^2(\bbR^n)},
\end{equation}
for some fixed $\beta>0$, where again, 
$\Delta = \sum_{j=1}^n \partial_{x_j}^2$ is the 
usual Laplacian on $\mathbb{R}^n$.  To this end, let us observe that by \eqref{eq4.gaffsk},
\eqref{eq4.gaffpqsk}, and \eqref{eq4.12}, we have that $\theta_t$ also satisfies 
estimate \eqref{eq4.gaffsk}
(with slightly different constants), that is, for $0<t\leq 4\ell(Q)$, we have
\begin{equation}\label{eq4.gafftheta}
\|\theta_t\|_{\cB(L^2(S_k(Q))^n,L^2(Q))}
 \leq C\, e^{-c2^k\ell(Q)/t} \,, \quad k \in \bbN.
\end{equation} 
Similarly, by \eqref{eq4.resolvent} and \eqref{eq4.12}, we have 
\begin{equation}\label{eq4.l2theta}
\sup_{t>0}\|\theta_t\|_{\cB(L^2(\mathbb{R}^n)^n, L^2(\bbR^n))}
 \leq C\,.
\end{equation}   
In particular then, by  \cite[Lemma\ 3.11]{AAAHK11}
we infer
\begin{equation}\label{eq4.25a}
\sup_{t>0}\|(\theta_t I_n) E_t\|_{\cB(L^2(\mathbb{R}^n)^n, L^2(\bbR^n))} \leq C.
\end{equation}  
Moreover, $(\theta_t I_n) E_t P_t$ trivially satisfies the same bound as $\theta_t$ in
\eqref{eq4.gafftheta}, since the kernels of $E_t$ and $P_t$ have compact support in a ball of radius $t$; 
thus $R_t$ also satisfies \eqref{eq4.gafftheta} and \eqref{eq4.l2theta}.
Note further that by construction, $R_t I_n = 0$.  Consequently, 
by \cite[Lemma\ 3.5(i)]{AAAHK11}, 
\begin{align}\label{eq4.25}
\begin{split} 
& \big\|R_t \nabla\, s^2 \Delta e^{2s^2\Delta} F \big\|_{L^2(\bbR^n)}  \leq 
C t \big\| \nabla^2\big(s^2 \Delta e^{2s^2\Delta} F\big)\big\|_{L^2(\bbR^n)}     \\
& \quad = C (t/s) \big\|\big(s^2\nabla^2
e^{s^2\Delta}\big) s\,e^{s^2\Delta} \Delta F \big\|_{L^2(\bbR^n)} 
\leq C (t/s) \big\| s\,e^{s^2\Delta} \Delta F \big\|_{L^2(\bbR^n)},
\end{split} 
\end{align}
which establishes \eqref{eq4.18} in the case $t\leq s$.    

Next, we verify \eqref{eq4.18} in the case $s<t$.  Recalling that $R_t$ is a difference, 
\begin{equation} 
R_t = \theta_t - (\theta_t I_n) E_tP_t, 
\end{equation} 
we now consider the contributions of the two terms separately.    We have
\begin{equation}\label{eq4.27}
\big\|(\theta I_n) E_tP_t\nabla\, s^2 \Delta e^{2s^2\Delta} F \big\|_{L^2(\bbR^n)} 
=\frac{s}{t} \big\|(\theta I_n) E_t\, t\nabla P_t \, e^{s^2\Delta}\,se^{s^2\Delta}\Delta F\big\|_{L^2(\bbR^n)},
\end{equation}
whence \eqref{eq4.18} follows for the present term, with $\beta=1$, by \eqref{eq4.25a} and the 
uniform $L^2$-boundedness of $t\nabla P_t$ and of $e^{s^2\Delta}$.   Consider now 
the contribution of $\theta_t$. To simplify notation, we set $Q_s:= s^2 \Delta e^{2s^2\Delta}$, 
and note that
\begin{align}\label{eq4.43a}
\theta_t \nabla Q_s F &= - t \left(I+t^2M_a\right)^{-1} \big(\dv (a \, \cdot \,) 
- \vec{B}_1 \cdot \big) \nabla Q_s F    \no \\
&= t \left(I+t^2M_a\right)^{-1} M_a \,Q_s F \,-\, t \left(I+t^2M_a\right)^{-1} 
\big[\dv(\B_2 \, \cdot \,) + V\big]\,Q_sF   \no \\
&=: G_{s,t} - H_{s,t}\,.
\end{align}
Then 
\begin{equation} 
G_{s,t} = \frac{s}{t} \big(I-(I+t^2M_a)^{-1}\big)\,s  e^{2s^2\Delta} \Delta F, 
\end{equation} 
which, by the case $p=2$ of \eqref{eq4.resolvent}, satisfies \eqref{eq4.18}  with  $\beta=1$.
To conclude our treatment of $R_t$, it remains to consider $H_{s,t}$,
which we split into
\begin{equation}\label{eq4.51}
H_{s,t} =  t\left(I+t^2M_a\right)^{-1}\dv \big(\B_2\,Q_sF\big) \,+\,  t\left(I+t^2M_a\right)^{-1}V\,Q_sF
\,=:\, H_{s,t}'+H_{s,t}''. 
\end{equation} 
By \eqref{eq4.resolvent},
with $p = 2_*-\eps$, for some $\eps>0$, and by H\"older's inequality,
we have
\begin{align}\label{eq4.52*}
\|H''_{s,t}\|_{L^2(\bbR^n)} & \lesssim t^{1-n(1/p-1/2)}\|V\,Q_s F\|_{L^p(\bbR^n)}
\lesssim  t^{1-n(1/p-1/2)}\|Q_s F\|_{L^{p^{**}}(\bbR^n)}    \no \\
& = t^{1-n(1/p-1/2)} s \big\|e^{s^2\Delta}\, s\Delta e^{s^2\Delta}F\big\|_{L^{p^{**}}(\bbR^n)}     \no \\
& \lesssim  t^{1-n(1/p-1/2)} s^{1-n(1/2-1/p^{**})} \big\|s\Delta e^{s^2\Delta}F\big\|_{L^2(\bbR^n)}     \no \\
& =\left(\frac{s}{t}\right)^{\frac{n}{p} - 1-\frac{n}{2}} \big\|s\Delta e^{s^2\Delta}F\big\|_{L^2(\bbR^n)},
\end{align}
where we recall that $p^{**}:= pn/(n-2p)$, and we have used the well-known
bound $\big\|e^{s^2\Delta} h\big\|_{L^r(\bbR^n)} \leq C_r\, s^{-n(1/2-1/r)}\|h\|_{L^2(\bbR^n)}$ 
for $r\geq 2$;  the last step is an elementary computation, using the relationship between $p$ 
and $p^{**}$.  We observe now that
\begin{equation} 
\beta: = \frac{n}{p} - 1-\frac{n}{2}= n\left(\frac1p -\frac{2+n}{2n}\right) 
=n\left(\frac1p-\frac1{2_*}\right)>0, 
\end{equation}  
since $p=2_*-\eps$.  
Similarly, by  \eqref{eq4.resolvent2},  we have that for some $q=2-\eps'$,
\begin{align}
\|H'_{s,t}\|_{L^2(\bbR^n)} & \lesssim t^{-n(1/q-1/2)} \big\|\B_2\,Q_s F\big\|_{L^q(\bbR^n)}  
\lesssim  t^{-n(1/q-1/2)}\|Q_s F\|_{L^{q^*}(\bbR^n)}     \no \\
&= t^{-n(1/q-1/2)} s \big\|e^{s^2\Delta}\, s\Delta e^{s^2\Delta}F\big\|_{L^{q^*}(\bbR^n)}       \no \\
& \lesssim  t^{-n(1/q-1/2)} s^{1-n(1/2-1/q^*)} \big\|s\Delta e^{s^2\Delta}F\big\|_{L^2(\bbR^n)}    \no \\
&=\left(\frac{s}{t}\right)^{n(1/q-1/2)} \big\|s\Delta e^{s^2\Delta}F\big\|_{L^2(\bbR^n)},
\end{align}
Since $q<2$, this concludes the proof of \eqref{eq4.18}, and hence also that of \eqref{eq4.21}.

It therefore remains only to establish the Carleson measure estimate \eqref{eq4.carl}.  To this end, we invoke a key fact in the proof of the Kato conjecture for $L_a$. 
By \cite{AHLMT02}, 
there exists, for each $Q$, a mapping
$F_{Q} : \mathbb{R}^{n}\rightarrow \mathbb{C}^{n}$ such that 
\begin{equation}
\begin{split}\label{eqA.5} 
\text{$(i)$} & \quad\int_{\mathbb{R}^{n}} d^n x \, |(\nabla F_{Q})(x)|^{2}\leq C|Q|\\
\text{$(ii)$} & \quad\int_{\mathbb{R}^{n}} d^n x \, 
|(L_aF_{Q})(x)|^{2}\leq C\frac{|Q|}{\ell(Q)^{2}}\\
\text{$(iii)$} & \quad\sup_{Q}\frac1{|Q|}\int_{0}^{\ell(Q)}\int_{Q} \frac{d^n xdt}{t} |\vec{\zeta}(x,t)|^{2}\\
 & \quad\quad\leq
C\sup_{Q}\frac1{|Q|}
\int_{0}^{\ell(Q)}\int_{Q} \frac{d^n xdt}{t} |\vec{\zeta}(x,t) E_{t}\nabla F_{Q}(x)|^{2},\end{split}
\end{equation}
 for every function $\vec{\zeta}:\mathbb{R}_{+}^{n+1}\rightarrow\mathbb{C}^{n}$, where
 $\nabla F_{Q}:=(D_{j}(F_{Q})_{k})_{1\leq j,k \leq n}$
 is the Jacobian matrix, and where the product
\begin{equation} 
\vec{\zeta} E_{t}\nabla F_{Q}=\sum_{j=1}^{n}\zeta_{j} E_{t}D_{j}F_{Q}
\end{equation} 
is a vector. 
Moreover, although not stated explicitly in \cite{AHLMT02}, the mapping $F_Q$ also satisfies a 
variant of \eqref{eqA.5}\,$(i)$, namely that for some $q=2+\eps>2$, with $\eps$ depending only on dimension and ellipticity, we also have
\begin{equation}\label{eq4.lq}
\int_{\mathbb{R}^{n}} d^n x \, |(\nabla F_{Q})(x)|^{q} \leq C |Q|
\end{equation}
(see, e.g., \cite{HM02} for a proof of this fact in the case that the kernel of the heat semigroup $e^{-tL}$ enjoys a pointwise Gaussian upper bound, but the proof there may be readily adapted to the general case in which the pointwise kernel bound is replaced by one of ``Gaffney-type").

Following the strategy in \cite[Chapter\ 3]{AT98}, and given the existence of a family of
mappings $F_{Q}$ with these properties, we see that
by \eqref{eqA.5}\,$(iii)$, applied with
$\vec{\zeta}(x,t)=\theta_{t}I_n$, in order to prove \eqref{eq4.carl},
it is enough to show that 
\begin{equation}\label{eq4.30}
\int_{0}^{\ell(Q)}\!\!\!\int_{Q} \frac{d^n xdt}{t} |\left(\theta_{t}I_n\right)(x)\,\left(E_{t}\nabla 
F_{Q}\right)(x)|^{2}
\leq C|Q|\, ,
\end{equation}
uniformly in $Q$.  To this end, we write
\begin{align}\label{eq4.46}
& \left(\theta_{t}I_n\right)\left(E_{t}\nabla F_{Q} \right)   \no \\
& \quad = \left(\theta_{t}I_n\right)\left((E_{t}-E_tP_t)\nabla 
F_{Q}\right) + \left(\left(\theta_{t}I_n\right)E_tP_t
-\theta_t\right) \nabla F_Q\, +\,\theta_t \nabla F_Q   \no \\
& \quad =:R_t^{(1)} +R_t^{(2)} + \theta_t \nabla F_Q, 
\end{align}
and notice that $-R_t^{(2)} = R_t$, as defined in \eqref{eq4.19a}.
Thus, the desired bound for this term follows immediately from 
\eqref{eq4.21} (with $F_Q$ in place of $F$) and \eqref{eqA.5}\,$(i)$.  We claim that a similar bound holds for the contribution of $R_t^{(1)}$,  by \eqref{eqA.5}\,$(i)$ and the well-known square function estimate
\begin{equation}\label{eq4.31}
\iint_{\mathbb{R}^{n+1}_+} \frac{d^n xdt}{t} |(E_t-P_t)h(x)|^2  \leq C\|h\|_{L^2(\bbR^n)}^2.
\end{equation}  
Indeed, since $E_t$ is a projection operator, we have that
\begin{equation} 
E_t-E_tP_t= E_t(E_t-P_t). 
\end{equation} 
Consequently, we may invoke \eqref{eq4.25a}  
to reduce matters to \eqref{eq4.31}, with $h=\nabla F_Q$.

The remaining term is 
\begin{equation} 
\theta_t \nabla F_Q = t\left(I+t^2M_a\right)^{-1}\left(L_a F_Q\right) \,+\,
t\left(I+t^2M_a\right)^{-1} (\vec{B}_1\cdot\nabla F_Q)=: \Phi_t +\Psi_t\,.
\end{equation} 
By the case $p=2$ of \eqref{eq4.resolvent},  and by \eqref{eqA.5} (ii), we have that
\begin{align}
\begin{split} 
\int_{0}^{\ell(Q)}\!\!\!\int_{Q} \frac{d^n xdt}{t} |\Phi_t(x)|^{2}
& \lesssim \int_{0}^{\ell(Q)}\!\!\!\int_{\mathbb{R}^n} d^n x \,t \, dt |L_a F_Q(x)|^2    \\
& \leq C |Q| \frac1{\ell(Q)^2} \int_{0}^{\ell(Q)} t \, dt = C|Q|,
\end{split} 
\end{align}
as desired.   Finally, to treat the contribution of $\Psi_t$, for $q=2+\eps$ as in \eqref{eq4.lq},
we choose $p$ so that 
$p^*=pn/(n-p) = q,$ 
and note that since $q>2$, we therefore have
$p>2_*$.
Then by \eqref{eq4.resolvent}, followed by H\"older's inequality,  
\begin{align}\label{eq4.32}
\int_{0}^{\ell(Q)}\!\!\!\int_{Q} \frac{d^n xdt}{t} |\Psi_t(x)|^{2}
& \lesssim \int_{0}^{\ell(Q)}dt \left(\int_{\mathbb{R}^n} 
\big|\vec{B}_1\cdot \nabla F_Q\big|^p d^n x\right)^{2/p}
t^{1-2n(1/p-1/2)}     \no \\
& \lesssim \int_{0}^{\ell(Q)}dt \, \big\|\vec{B}_1\big\|_{L^n(\bbR^n)^n}^2\,
\|\nabla F_Q\|_{L^{p^*}(\bbR^n)}^2 \,
t^{1-2n(1/p-1/2)}     \no \\
& \lesssim |Q|^{2/p^*} \int_{0}^{\ell(Q)}dt \,
t^{1-2n(1/p-1/2)} \,,
\end{align}
where in the last step we have used \eqref{eq4.lq} with $q=p^*$.
Since $p>2_* = 2n/(n+2)$, we infer 
\begin{equation} 
1-2n(1/p-1/2)\, >\,1-2n\left(\frac{n+2}{2n}-\frac12\right) = 1-(n+2-n) =-1. 
\end{equation} 
Consequently, the last term in \eqref{eq4.32} equals
\begin{equation} 
C \,|Q|^{2/p^* +2/n -2(1/p-1/2)} = C\,|Q|. 
\end{equation} 
This finishes the proof of Theorem \ref{t4.5a}, in the case $n\geq 3$.  Let us now treat
the two-dimensional case.  

In two dimensions, it remains to verify the square function bound \eqref{eq4.16}, but now
with $H_a=M_a=L_a+\B\cdot(\nabla \, \cdot \,)$, where $\dv \big(\B\big) = 0$.
As in the higher-dimensional case, our first step will be to reduce matters to the Carleson measure
estimate \eqref{eq4.carl}.  
The second step (again the main step), will be to prove that \eqref{eq4.carl} holds.
We define $\theta_t$ as above \eqref{eq4.thetadef}, with $\B_1=\B$, but observe that now,
\begin{equation}\label{eq4.updef}
\theta_t\nabla = t(I+t^2M_a)^{-1}M_a = t^{-1}\big(I-(I+t^2M_a)^{-1}\big)\,=: t^{-1} \U_t
\end{equation}
(cf. \eqref{eq4.thetadef2}). We note that by Lemma \ref{l4.4}, the operator $\U_t$ satisfies 
\begin{equation}\label{eq4.l2up}
\sup_{t>0}\|\U_t\|_{\cB(L^2(\mathbb{R}^2))}
 \leq C\,,
\end{equation}   
and also estimate \eqref{eq4.gaffsk2d},
that is, for $0<t\leq 4\ell(Q)$, we have
\begin{equation}\label{eq4.gaffup}
\|\U_t\|_{\cB(L^2(S_k(Q)), L^2(Q))}
 \leq C\,  \left(\frac{t}{2^k\ell(Q)}\right)^4 \,, \quad k \in \bbN.
\end{equation} 
Let us now proceed with

\medskip

\noindent {\bf Step 1}:  
reduction to the Carleson measure estimate \eqref{eq4.carl}.

\begin{proof}[Verification of Step 1]  The goal here is to show that
 the square function bound \eqref{eq4.16}, in two dimensions, follows from \eqref{eq4.carl}.
To this end, we make the same splitting as in \eqref{eq4.19a}, recalling that 
in the present setting, $H_a=M_a$.  Then main term $II$ is handled exactly as before,
and is bounded, without further ado, 
in terms of the Carleson measure estimate \eqref{eq4.carl}.   The ``remainder" term
$I$ is also handled as before:  we need to verify the quasi-orthogonality bound \eqref{eq4.18}.
To carry out the latter task will require a bit of an argument.

In the present setting, by \eqref{eq4.updef}, we have
\begin{equation}\label{eq4.52}
R_t \nabla= \theta_t\nabla -(\theta_t I_2) E_t P_t \nabla= 
t^{-1} \U_t - t^{-1} (\U_t\Phi) E_tP_t\nabla\,,
\end{equation}
where $I_2 =\nabla\Phi$ is the $2\times 2$ identity matrix.
We first establish \eqref{eq4.18} in the case $s\leq t$, and consider the two terms in
\eqref{eq4.52} separately.   The contribution of the term $t^{-1}\U_t =\theta_t \nabla$
is treated exactly like $G_{s,t}$ in \eqref{eq4.43a}; there is no term $H_{s,t}$ in the present context.
The contribution of the other term $(\theta_t I_2) E_t P_t \nabla = 
t^{-1}(\U_t\Phi) E_tP_t\nabla$ may be handled exactly like its higher-dimensional analogue in
\eqref{eq4.27}, except for one minor technical issue which we now describe. 
In estimating \eqref{eq4.27},
we used the
bound \eqref{eq4.25a}, which in turn relied upon
\eqref{eq4.gafftheta} and \eqref{eq4.l2theta}, but in two dimensions it is not clear that $\theta_t$ satisfies
the latter pair of bounds.
Thus,  we cannot simply invoke \cite[Lemma\ 3.11]{AAAHK11}
``off-the-shelf" to obtain \eqref{eq4.25a};  nonetheless, the proof of  \cite[Lemma\ 3.11]{AAAHK11}
goes through, {\it mutatis mutandi}, to establish \eqref{eq4.25a}, 
since $\theta_tI_n = t^{-1}\U_t \Phi$ (cf. \eqref{eq4.updef}),
where $\U_t$ satisfies \eqref{eq4.l2up}, \eqref{eq4.gaffup}, and $\U_t1=0$.  The rest of the argument may now be carried out as before.  We omit the details.

Next, we consider the case $t<s$.  As above, let $Q_s = 
s^2 \Delta e^{2s^2\Delta}$. 
We now claim that for $g\in W^{2,2}(\R^2)$, 
we have
\begin{equation}\label{eq4.54}
\|R_t\nabla g\|_{L^2(\bbR^2)} \,\leq\, C t\, \big\|\nabla^2 g\big\|_{L^2(\bbR^2)}.
\end{equation}
Let us take this claim for granted momentarily, and apply \eqref{eq4.54} with
$g=Q_sF$.  We then have
\begin{align} 
\begin{split} 
\big\|R_t\nabla \big(s^2 \Delta e^{2s^2\Delta} F\big)\big\|_{L^2(\bbR^2)} & \lesssim\frac{t}{s} 
\big\|s^2\nabla^2e^{s^2\Delta} \,s \,e^{s^2\Delta} \Delta F\big\|_{L^2(\bbR^2)}    \\
& \lesssim \frac{t}{s} \big\|s\, e^{s^2\Delta} \Delta F\big\|_{L^2(\bbR^2)},
\end{split} 
\end{align} 
which is \eqref{eq4.18} in the case $t<s$, with $\beta=1$.

We now turn to the proof of \eqref{eq4.54}.  For $t$ fixed, choose $j(t)$ so that
$2^{j(t)}\leq t < 2^{j(t) +1}$, and let
$\mathbb{D}_t$ denote the dyadic grid of cubes of side length $2^{j(t)}$.
For a measurable set $E$, let $\fint_E f:=|E|^{-1}\int_E f$ denote the average of $f$ over $E$. 
For $(x,t)$  fixed, we set 
\begin{equation} 
G_{x,t}(z):= g(z) - g(x) - (z-x)\cdot  \left(\nabla g\right)(x). 
\end{equation} 
Observe that, by \eqref{eq4.52},
\begin{equation} 
R_t\nabla g(y) = \big(\U_t -(\U_t\Phi) E_tP_t\nabla\big)\left(t^{-1}G_{x,t}\right)(y) 
=: \Ut_t\left(t^{-1}G_{x,t}\right)(y), 
\end{equation} 
since  $\U_t 1=0$, and since $(R_t\nabla)\Phi=0$.  We note that
$\Ut_t$  satisfies \eqref{eq4.l2up} and \eqref{eq4.gaffup} (for the term
$(\U_t\Phi) E_tP_t\nabla$, this relies on the fact that, as noted above, estimate
\eqref{eq4.25a} continues to hold in two dimensions).  Thus, setting $S_0(Q):=2Q$, and 
$S_k(Q):= 2^{k+1}Q\backslash  2^kQ$, we have
\begin{align} \lb{eq2.18}  
\|R_t\nabla g\|_{L^2(\bbR^2)} &= \bigg( 
\sum_{Q\in \mathbb{D}_{t}}\int_Q d^2 y \, | \Ut_t\left(t^{-1}G_{x,t}\right)(y)|^2\bigg)^{1/2}   \no \\
& = \bigg(\sum_{Q\in \mathbb{D}_{t}} \fint_Q \int_Q 
d^2 yd^2 x \, |\Ut_t\left(t^{-1}G_{x,t}\right)(y)|^2 \bigg)^{1/2}  \no \\
& \leq \sum_{k=0}^\infty\bigg(\sum_{Q\in \mathbb{D}_{t}} \fint_Q \int_Q 
d^2 yd^2 x \, |\Ut_t\left(t^{-1}G_{x,t}1_{S_k(Q)}\right)(y)|^2 \bigg)^{1/2}    \no \\
& \lesssim \sum_{k=0}^\infty 2^{-k}
\Bigg\{\sum_{Q\in \mathbb{D}_{t}} \int_Q \fint_{|x-y|<C2^k t} d^2 y d^2 x   \no \\
& \hspace*{2cm}
\times \bigg(\frac{|g(y) - g(x) - (y-x)\cdot (\nabla g)(x)|}{2^{2k} t}\bigg)^2 \Bigg\}^{1/2}\,,
\end{align}  
where we have used that $\ell(Q)\approx t$, 
for $Q\in \mathbb{D}_t$, and that $n=2$.  By  
Plancherel's Theorem,
the expression in curly brackets equals
\begin{align}
& \int_{\R^2} \fint_{|h|<C2^kt}d^2 h d^2 x \left(\frac{\left|g(x+h) - g(x) - h\cdot  
\left(\nabla g\right)(x)\right|}{2^{2k} t}\right)^2   \no \\   
& \quad = t^2\, \fint_{|h|<C}\int_{\R^2} d^2 \xi d^2 h 
\bigg(\frac{\big|e^{2\pi i\xi\cdot 2^kth} -1-(2\pi i\xi\cdot 2^kth)\big|}{(2^k  t|\xi|)^2}\bigg)^2
|\xi|^4\, |\hat g(\xi)|^2      \no \\
& \quad \lesssim  t^2 \|\Delta g\|_{L^2(\bbR^2)}^2.
\end{align}
Plugging this bound into \eqref{eq2.18}, we obtain \eqref{eq4.54}, which completes Step 1.
\end{proof}

\noindent {\bf Step 2}:  proof of the Carleson measure estimate \eqref{eq4.carl}.

\medskip

As in the higher-dimensional setting, we use the family of mappings $F_Q$ satisfying
\eqref{eqA.5} and \eqref{eq4.lq}, so that once again matters are reduced to proving
\eqref{eq4.30}.  We make the same splitting, in \eqref{eq4.46}, and observe that the 
the ``main term" $\theta_t\nabla F_Q$, as well as the two ``remainder terms"
$R^{(1)}_t$ and $R^{(2)}_t$, may be handled exactly as before.  To handle the remainder terms, 
we use \eqref{eq4.25a}, which, as we have noted above, is no longer a direct consequence of 
Lemma 3.11 of  \cite{AAAHK11} as stated, 
but nonetheless follows readily from its  proof; we omit the details.
The proof of Theorem \ref{t4.5a} is now complete.
\end{proof}

Next, we remove the smallness condition \eqref{eq4.smallness}, and hence formulate the principal 
result of this section:

\begin{theorem} \lb{t4.5} 
Assume Hypothesis \ref{h4.1} with $n\in\bbN$, $n \geq 3$, and let $L_a$ be defined as in 
\eqref{eq3}, and $M_a$ be defined as in \eqref{eq4.1}. Then  
\begin{equation}\lb{4.59}
\dom(M_a^{1/2}) = W^{1,2}(\bbR^n).
\end{equation}
Moreover, decomposing $\B_j$, $j=1,2$, and $V$ into 
\begin{align}
& \B_j = \B_{j,\varepsilon_0,n} + \B_{j,\varepsilon_0,\infty}, \quad 
\big\| \B_{j,\varepsilon_0,n} \big\|_{L^n(\bbR^n)^n} \leq \varepsilon_0/3,   \no \\
& \quad \big\| \B_{j,\varepsilon_0,\infty} \big\|_{L^{\infty}(\bbR^n)^n} < \infty, \; j=1,2,   \lb{4.59a} \\
& V = V_{\varepsilon_0,n/2} + V_{\varepsilon_0,\infty}, \quad 
\|V_{\varepsilon_0,n/2}\|_{L^{n/2}(\bbR^n)} \leq \varepsilon_0/3, \quad 
\|V_{\varepsilon_0,\infty}\|_{L^{\infty}(\bbR^n)} < \infty,    \no 
\end{align}
choosing $0 < \varepsilon_0$ sufficiently small such that Theorem \ref{t4.2} applies to 
$M_{a,\varepsilon_0}$, where 
\begin{equation} 
M_{a,\varepsilon_0} = L_a + \B_{1,\varepsilon_0,n} \cdot \nabla + 
\dv \big(\B_{2,\varepsilon_0,n} \cdot \big) +V_{\varepsilon_0,n/2},       \lb{4.59b}
\end{equation}
and 
\begin{equation}
M_a = M_{a,\varepsilon_0} + \B_{1,\varepsilon_0,\infty} \cdot \nabla 
+ \dv \big(\B_{2,\varepsilon_0,\infty} \cdot \big) + V_{\varepsilon_0,\infty}. 
\end{equation}
Then for all $\varepsilon > 0$, 
\begin{equation}\label{4.60}
\big\|M_a^{1/2} f \big\|_{L^2(\bbR^n)}^2 \leq [C_1 + \varepsilon C_2] \|\nabla f\|_{L^2(\bbR^n)^n}^2 
+ [C_2/(4 \varepsilon)] \| f\|_{L^2(\bbR^n)}^2,  
\quad f \in W^{1,2}(\bbR^n), 
\end{equation}
where   
\begin{align} 
\begin{split} 
C_1 &= C_1\big(n,a_1,a_2, \big\|\B_{1,\varepsilon_0,n} \big\|_{L^n(\bbR^n)^n}, 
\big\|\B_{2,\varepsilon_0,n} \big\|_{L^n(\bbR^n)^n}, \|V_{\varepsilon_0,n/2}\|_{L^{n/2}(\bbR^n)}\big),     \\
C_2 &= C_2\big(\big\|\B_{1,\varepsilon_0,\infty} \big\|_{L^{\infty}(\bbR^n)^n}, 
\big\|\B_{2,\varepsilon_0,\infty} \big\|_{L^{\infty}(\bbR^n)^n}, 
\|V_{\varepsilon_0,\infty}\|_{L^{\infty}(\bbR^n)}\big).
\end{split}
\end{align}  
\end{theorem} 

\begin{remark} We observe that, in Theorem \ref{t4.5}, 
our quantitative bounds do not depend simply on $\|\B_j\|_n$ and
$\|V\|_{n/2}$, but rather also upon the $L^\infty$ norms of the ``non-small" parts of $\B_j$ and $V$.
Thus, we do not obtain a homogeneous estimate (i.e., one which involves only
$\|\nabla f\|_2^2$) on the right hand side of \eqref{4.60}.  One should contrast this situation with that
of Theorems \ref{t4.2} and \ref{t4.5a}, in which we obtain the homogeneous estimate \eqref{eq4.3},
given the conclusions of Lemmata \ref{l4.4} and \ref{l4.5*} (in particular, in the presence  
of \eqref{eq4.smallness}). 
\end{remark} 

\begin{proof}
We note that the sequilinear forms uniquely associated with $\B_{1,\varepsilon_0,\infty} \cdot \nabla$ 
and $\dv \big(\B_{2,\varepsilon_0,\infty} \cdot \big)$ are of the type 
\begin{equation}
\big({\ol \B_{1,\varepsilon_0,\infty}} f,\nabla g\big)_{L^2(\bbR^n)^n}, \quad f, g \in W^{1,2}(\bbR^n),     \lb{4.63} 
\end{equation}
and 
\begin{equation}
- \big(\nabla f,\B_{2,\varepsilon_0,\infty} g\big)_{L^2(\bbR^n)^n}, \quad f, g \in W^{1,2}(\bbR^n),    \lb{4.64} 
\end{equation}
respectively. Thus, for $b \in L^\infty(\bbR^n)$, the elementary estimate 
\begin{align}
& \big|\big(b f, (-\Delta)^{1/2} f\big)_{L^2(\bbR^n)}\big| \leq \|b\|_{L^{\infty}(\bbR^n)} \Big[\varepsilon 
\big\|(-\Delta)^{1/2} f \big\|_{L^2(\bbR^n)}^2
+ (4 \varepsilon)^{-1} \| f \|_{L^2(\bbR^n)}^2\Big],   \no \\ 
& \hspace*{7.5cm} f \in W^{1,2}(\bbR^n), \; \varepsilon > 0,    
\end{align}
shows that the forms \eqref{4.63} and \eqref{4.64} are infinitesimally 
bounded with respect to the form $(\nabla f, \nabla g)_{L^2(\bbR^n)^n}$, $f, g \in W^{1,2}(\bbR^n)$, 
and hence also with respect to the form uniquely associated with $M_{a,\varepsilon_0}$. 

At this point it remains to apply Theorem \ref{t2.12} repeatedly as follows: We identify $T_0$ with 
$-\Delta$ defined on $W^{2,2}(\bbR^n)$. Then, in a first step, identifying $A= V_{\varepsilon_0,\infty}$, 
$B^* = I_{L^2(\bbR^n)}$,  $T_1 = M_{a,\varepsilon_0}$, and $T = T_1 + V_{\varepsilon_0,\infty}$, it 
is obvious that Hypotheses \ref{h2.3} and \ref{h2.11} are satisfied and hence Theorem \ref{t2.12} applies to 
$T = M_{a,\varepsilon_0} + V_{\varepsilon_0,\infty}$. In the second step we now identify $A = \nabla$, 
$B^* = \B_{1,\varepsilon_0,\infty}$,  $T_1 = M_{a,\varepsilon_0} + V_{\varepsilon_0,\infty}$, and 
$T = T_1 + \B_{1,\varepsilon_0,\infty} \cdot \nabla$, again Hypotheses \ref{h2.3} and \ref{h2.11} are satisfied and hence Theorem \ref{t2.12} applies to 
$T = M_{a,\varepsilon_0} + V_{\varepsilon_0,\infty} + \B_{1,\varepsilon_0,\infty} \cdot \nabla$. 
Indeed, the integrability conditions \eqref{2.10} and \eqref{2.11} are satisfied since there 
exists constants $C_1>0$, $C_2>0$ such that 
\begin{align}
& \big\| (-\Delta)^{1/2} (-\Delta + (\lambda + E) I_{L^2(\bbR^n; d^n x)})^{-1/2}\big\|_{\cB(L^2(\bbR^n))} 
\leq C_1, \quad \lambda \geq R, \; E \geq E_0,      \no \\
& \big\| \big| \B_{1,\varepsilon_0,\infty} \big|  (-\Delta 
+ (\lambda + E) I_{L^2(\bbR^n; d^n x)})^{-1/2}\big\|_{\cB(L^2(\bbR^n))}  
\leq C_2 \lambda^{-1/2}, \quad \lambda \geq R, \; E \geq E_0,     \lb{4.66}
\end{align}
for $0 < R, E_0$ sufficiently large. (Actually, to be precise, one has to apply this step by step to 
each term $\big(\B_{1,\varepsilon_0,\infty}\big)_j (\nabla)_j$, $1 \leq j \leq n$, but the necessary 
estimates are the same.) In the third and final step one identifies $A =  \B_{2,\varepsilon_0,\infty}$, 
$B^* = \dv(\cdot)$, $T_1 = M_{a,\varepsilon_0} + V_{\varepsilon_0,\infty} 
+ \big(\B_{1,\varepsilon_0,\infty} \cdot \nabla \cdot \big)$, and 
$T = M_a = T_1 + \dv \big(\B_{2,\varepsilon_0,\infty} \cdot\big)$ and once more uses the estimates 
in \eqref{4.66}.  
\end{proof}

\begin{remark} \lb{r4.6} 
$(i)$ We note that the results of Theorem \ref{t4.5} continue to hold when $n=2$, provided that 
$\B_j$ permit a decomposition as in \eqref{4.59a} such that $\B_{j,\varepsilon_0,2}$, $j=1,2$, 
are divergence free, and that $V\equiv 0$. \\
$(ii)$ The critical condition $V \in L^1(\bbR^2)$ is not sufficient in the case $n=2$ as the Sobolev 
embedding \eqref{4.28.5} fails in this case. Thus, for $n=2$ one resorts to the condition 
$V \in L^p(\bbR^2)$ for some $p > 1$ as in Hypothesis \ref{h3.2}. \\
$(iii)$ While on the $L^p$-scale the critical index for $V$ in Theorems \ref{t4.2} and \ref{t4.5} 
equals $p=n/2$ for $n \geq 3$, it is possible to go beyond $L^{n/2}(\bbR^n)$. In this context we 
refer to Yagi \cite{Ya87} (see also \cite{Ya84}) who proved stability of square root domains in the 
particular case of three-dimensional Schr\"odinger operators, that is, in the special case $a = I_3$, 
and in the case $\B_j = 0$, $j=1,2$, for potentials $V$ in the Rollnik class $\cR_3$, 
that is, measurable functions $V:\bbR^3 \to \bbC$ satisfying $\|V\|_{\cR_3} < \infty$, where 
\begin{equation}
\|V\|_{\cR_3} := \int_{\bbR^6} d^3x \, d^3 x' \, |V(x)| |x-x'|^{-2} |V(x')|. 
\end{equation}
By Sobolev's inequality, $L^{3/2}(\bbR^3) \subset \cR_3$, in fact, for some $C>0$, 
\begin{equation}
\|V\|_{\cR_3} \leq C \|V\|_{L^{3/2}(\bbR^3)}
\end{equation}
(cf.\ \cite[Section\ I.1]{Si71}). In fact, Yagi \cite{Ya87} also considers fractional powers 
$\alpha \in (0,3/4)$, not necessarily $\alpha = 1/2$, in this context. 
\end{remark}

Finally, we turn to the proofs of Lemmas \ref{l4.4} and \ref{l4.5*}, starting with the former.

\begin{proof}[Proof of Lemma \ref{l4.4}]  
We begin by proving \eqref{eq4.resolvent},  \eqref{eq4.resolventgrad} and  \eqref{eq4.resolvent2}. 
We shall obtain these estimates by establishing appropriate convergence
of a Neumann series expansion of the resolvent.  Suppose first that $n\geq 3$ and set
\begin{equation} 
R:= \big(\B_1 \cdot \nabla \cdot \big) + \dv \big(\B_2 \cdot \big) + V,
\end{equation} 
so that
\begin{align}\label{eq4.neumann}
&(I+t^2M_a)^{-1}     \\
& \quad = (I+t^2L_a)^{-1/2}
\big[I + \big((I+t^2L_a)^{-1/2} \,t^2 R\,(I+t^2L_a)^{-1/2}\big)\big]^{-1}(I+t^2L_a)^{-1/2}     \no \\ 
& \quad =(I+t^2L_a)^{-1/2}\sum_{k=0}^\infty \big[-(I+t^2L_a)^{-1/2}\,t^2R\,(I+t^2L_a)^{-1/2}\big]^k 
(I+t^2L_a)^{-1/2}.    \no 
\end{align}
We shall require the following estimates, which may be gleaned from the results in \cite{Au07}:
\begin{align} 
& \big\|(I+t^2L_a)^{-1/2}g\big\|_{L^q(\bbR^n)} \leq C_{p,q}t^{-n(1/p-1/q)}\,\|g\|_{L^p(\bbR^n)},    \lb{eq4.41a} \\
& \hspace*{3.75cm} 2_*-\eps\leq p\leq q\leq p^*\leq 2^*+\eps,    \no \\
& \big\|t\nabla(I+t^2L_a)^{-1/2}g\big\|_{L^p(\bbR^n)^n} \leq C_{p}\,\|g\|_{L^p(\bbR^n)}, 
\quad 2_*-\eps\leq p\leq 2+\eps,   \lb{eq4.41b} \\
& \big\|t(I+t^2L_a)^{-1/2}\dv\vec{(g)}\big\|_{L^q(\bbR^n)} \leq\, C_{q}\,\|\vec{g}\|_{L^q(\bbR^n)^n},
\quad 2-\eps\leq q\leq 2^*+\eps.     \lb{eq4.41c}
\end{align}
Here, $\eps>0$ is a small number, not necessarily the same in each occurrence in 
\eqref{eq4.41a}--\eqref{eq4.41c}, but
which depends only upon $n$, $a_1$, and $a_2$. 
We now claim that for some $\eps>0$, we have
\begin{equation}\label{eq4.44a}
\big\|(I+t^2L_a)^{-1/2}\,t^2R\,(I+t^2L_a)^{-1/2}\big\|_{\cB(L^r(\bbR^n))} \leq C_r \eps_0, \quad
2-\eps\leq r\leq 2+\eps
\end{equation}
(where of course $C_r$ depends also upon $n$, $a_1$, and $a_2$).
 Let us take this claim for granted momentarily.  
 
 To prove
 \eqref{eq4.resolvent}, we first take $p=q=2$ in \eqref{eq4.41a}, and $r=2$ in \eqref{eq4.44a},
 to obtain that for $\eps_0<C_2$, the Neumann series in \eqref{eq4.neumann} converges in $L^2$.
 We therefore obtain the case $p=2$ of \eqref{eq4.resolvent}.
 The remaining cases of \eqref{eq4.resolvent} then follow by interpolation from the
 endpoint case $p=2_*-\eps$, where we choose $\eps$ so that \eqref{eq4.41a} holds, and such that
 for this $p$, the exponent $r:=p^*=pn/(n-p)<2$ lies within the range specified in \eqref{eq4.44a}.
 Let us then momentarily fix this choice of $p$.
 We then obtain that the Neumann series converges as a mapping from $L^p\to L^2$, by first applying
 \eqref{eq4.41a} with $q=p^*$, then \eqref{eq4.44a} with $r=p^*$, and with $\eps_0<C_{p^*}$, and then
 \eqref{eq4.41a} with $q=2$ and with $p$ replaced by $p^*$.  This concludes the proof
 of \eqref{eq4.resolvent} when $n\geq 3$.
 
Similarly, we may obtain \eqref{eq4.resolventgrad}, 
and \eqref{eq4.resolvent2},  in the former case by invoking
 \eqref{eq4.41a} with $p=q=2$,  \eqref{eq4.44a} with $r=2$ and $\eps_0<C_2$, and
  \eqref{eq4.41b} with $p=2$.    To prove  \eqref{eq4.resolvent2}, we apply \eqref{eq4.41c}
  with $q=2-\eps$, then \eqref{eq4.44a} with $r=q=2-\eps$, and with $\eps_0<C_q$, and then 
  \eqref{eq4.41a} with $p$ replaced by the current $q=2-\eps$, and with $q$ replaced by 2.

To complete the proofs of \eqref{eq4.resolvent}, \eqref{eq4.resolventgrad}, and
\eqref{eq4.resolvent2}, with $n\geq 3$, it remains to verify the claimed bound \eqref{eq4.44a}.
To this end, let $g\in L^r(\bbR^n)$, and write
\begin{align}
& \big\|(I+t^2L_a)^{-1/2}\,t^2R\,(I+t^2L_a)^{-1/2}g\big\|_{L^r(\bbR^n)}    \no \\
& \quad \leq \big\|(I+t^2L_a)^{-1/2}\,t^2\B_1\cdot\nabla\,(I+t^2L_a)^{-1/2}g\big\|_{L^r(\bbR^n)}    \no \\
& \qquad + \big\|(I+t^2L_a)^{-1/2}\,t^2\dv \big(\B_2 \,(I+t^2L_a)^{-1/2}g\big)\big\|_{L^r(\bbR^n)}    \no \\
& \qquad + \big\|(I+t^2L_a)^{-1/2}\,t^2V\,(I+t^2L_a)^{-1/2}g\big\|_{L^r(\bbR^n)}    \no \\ 
& \quad =:\,I+II+III.
\end{align}
By \eqref{eq4.41a} with $q=r$, and $p=r_*:= rn/(n+r)$, followed by
H\"older's inequality and \eqref{eq4.smallness}, and then \eqref{eq4.41b} with $p=r$,
we obtain that 
\begin{equation} 
I\leq \,C_r\,\eps_0\|g\|_{L^r(\bbR^n)}, 
\end{equation} 
as desired.  Similarly, we may handle terms $II$ and $III$
by use of \eqref{eq4.41a} and \eqref{eq4.41c}, with appropriate choices of exponents,
along with H\"older's inequality and \eqref{eq4.smallness}.  We 
omit the details, which are by now familiar.

Let us now verify  \eqref{eq4.resolvent2d},  
\eqref{eq4.resolventgrad}, and the case $p=2$ of \eqref{eq4.resolvent}, when $n=2$. 
In this case, we use the resolvent expansion
\begin{equation}\label{eq4.expansion2d}
(I+t^2M_a)^{-1} = (I+t^2L_a)^{-1}\sum_{k=0}^\infty \big[-t^2\B\cdot \nabla(I+t^2L_a)^{-1}\big]^k\,,
\end{equation}
where now $\dv \big(\B\big) = 0$. We note that $(I+t^2 L_a)^{-1}$ satisfies the following estimates in $\R^2$:
\begin{equation}\label{eq4.resolvent2dL}
\big\|(I +t^2L_a)^{-1}\big\|_{\cB(H^1(\bbR^2),L^2(\bbR^2))} + 
\big\|t\nabla(I +t^2L_a)^{-1}\big\|_{\cB(H^1(\bbR^2),L^2(\bbR^2)^2)} \,\leq \,C\, t^{-1}  \,, 
\end{equation} 
and
\begin{align}\label{eq4.resolvent2dLl2}
& \big\|(I +t^2L_a)^{-1}\big\|_{\cB(L^2(\bbR^2),L^2(\bbR^2))} + 
\big\|t\nabla(I +t^2L_a)^{-1}\big\|_{\cB(L^2(\bbR^2),L^2(\bbR^2)^2)}    \no \\
& \quad + 
\big\|t^2\nabla(I +t^2L_a)^{-1}\dv \big\|_{\cB(L^2(\bbR^2)^2,L^2(\bbR^2)^2)} \,\leq C, 
\end{align} 
Indeed, \eqref{eq4.resolvent2dLl2} is completely standard, and estimate
\eqref{eq4.resolvent2dL}   
then follows by 
two-dimensional Sobolev embedding.  By \eqref{eq2dclms}, 
\eqref{eq4.resolvent2dL}, and a simple induction argument,
we have that
\begin{equation}\label{eq5.9}
\Big\|(I+t^2L_a)^{-1}\big[-t^2\B\cdot \nabla(I+t^2L_a)^{-1}\big]^kf\Big\|_{L^2(\R^2)}\,\leq\, 
t^{-1}C^k\eps_0^k\, \|f\|_{H^1(\R^2)}\,,
\end{equation}
whence  \eqref{eq4.resolvent2d} follows immediately from \eqref{eq4.expansion2d},
for $\eps_0$ small enough.  Similarly, for $k\geq 1$, using
\eqref{eq2dclms} and both \eqref{eq4.resolvent2dL}
and \eqref{eq4.resolvent2dLl2}, we obtain
\begin{align}\label{eq5.10}
& \Big\|(I+t^2L_a)^{-1}\big[-t^2\B\cdot \nabla(I+t^2L_a)^{-1}\big]^kf\Big\|_{L^2(\R^2)}    \no \\
& \qquad + \Big\|t\nabla(I+t^2L_a)^{-1}
\big[-t^2\B\cdot \nabla(I+t^2L_a)^{-1}\big]^kf\Big\|_{L^2(\R^2)^2} \no \\
& \quad \leq C^{k-1}\eps_0^{k-1} \big\|t\B\cdot \nabla(I+t^2L_a)^{-1}f\big\|_{H^1(\R^2)} 
\leq C^{k}\eps_0^{k}\,\|f\|_{L^2(\R^2)}.
\end{align}
By \eqref{eq4.expansion2d}, we then obtain 
\eqref{eq4.resolventgrad}, and the case $p=2$ of \eqref{eq4.resolvent}.  The remaining cases
$1<p<2$, of \eqref{eq4.resolvent} in two dimensions now follow by interpolation.

\smallskip

Next, we turn to the proof of \eqref{eq4.gaff}.
The proof follows that of
\cite[Lemma 2.1]{AHLMT02}, and is essentially the same in all dimensions $n\geq 2$.
To be precise, we give the proof in the case $n\geq3$, and then discuss, at the end of the argument,
the modifications needed to treat the two-dimensional case.

\begin{remark}\label{r4.7}
We note that $M_a$ has the property that its adjoint is of precisely the same class,
thus, any estimates that we establish for $M_a$ can be immediately dualized. 
This property holds even in two dimensions, since in that case we have supposed that $\B$
is divergence free.
\end{remark}
 
We fix two cubes $Q_1, Q_2$, set $d:=\dist(Q_1,Q_2)$,  
let $r:= \min\diam(Q_1,Q_2)$, $R:=\max \diam(Q_1,Q_2)$
and suppose that $0<t\leq 4r $, and that $R\leq 100 d$.  
Let $\eta\in C_0^\infty(\mathbb{R}^n)$, supported on a
$d/2$ neighborhood of $Q_2$. 
Thus, in particular, $\eta\equiv 0$ on $Q_1$. 
We now set $u:= (I+t^2M_a)^{-1} f$, where $\supp \, (f) \subset Q_1$.
We claim that 
\begin{equation}\label{eq4.34}
\int_{\bbR^n} d^n x \, |t\nabla u|^2\eta^2 + \int_{\bbR^n} d^n x \, |u|^2\eta^2\, 
\lesssim \,t^2 \int_{\bbR^n} d^n x \, |u|^2|\nabla \eta|^2\,.
\end{equation}
We verify the claim as follows.
Observe that by definition
\begin{equation} 
\int_{\bbR^n} d^n x \, (I+t^2M_a)u \,\overline{u} \eta^2 =\int_{\bbR^n} d^n x \, f\,\uu \eta^2\,= 0, 
\end{equation} 
since $f$ and $\eta$ have disjoint supports.
Rewriting the left-hand side of this last identity (using the definition of
$M_a$), integrating by parts,
and shuffling terms, we obtain
\begin{align}\label{eq5.12}
& t^2\int_{\bbR^n} d^n x \, a\nabla u\cdot\nabla \uu \eta^2 
+\int_{\bbR^n} d^n x \, |u|^2\eta^2    \no \\
& \quad = -t^2\int_{\bbR^n} d^n x \, \B_1\cdot \nabla u \,\uu\eta^2 
+t^2 \int_{\bbR^n} d^n x \, u \,B_2\cdot \nabla\uu\,\eta^2 
-t^2\int_{\bbR^n} d^n x \, V |u|^2\eta^2   \no \\ 
& \qquad -2t^2 \int_{\bbR^n} d^n x \, a \nabla u \cdot\nabla \eta\, \uu \eta 
+ 2t^2\int_{\bbR^n} d^n x \, B_2|u|^2\nabla \eta\eta   \no \\
& \quad =:I+II+III+IV+V. 
\end{align}
By ellipticity, this implies that
\begin{equation}\label{eq4.35}
\int_{\bbR^n} d^n x \, |t\nabla u|^2\eta^2 
+ \int_{\bbR^n} d^n x \, |u|^2\eta^2\,\lesssim\,|I+II+III+IV+V|\,.
\end{equation}
Using H\"older's inequality (equivalently,  \eqref{eq4.11} and 
the dual version of \eqref{eq4.12}), and then
Cauchy's inequality, we have
\begin{align}
& |I+II+III| \lesssim \eps_0t^2\left(\int_{\bbR^n} d^n x \, |\nabla u|^2\eta^2 
+ \left(\int_{\bbR^n} d^n x \, |u\eta|^{2^*}\right)^{2/2^*}\right),
\no \\ 
& \quad \lesssim \eps_0t^2\left(\int_{\bbR^n} d^n x \, |\nabla u|^2\eta^2 
+\int_{\bbR^n} d^n x \, |\nabla(u\eta)|^{2}\right)  \no \\
& \quad \lesssim\, \eps_0t^2\left(\int_{\bbR^n} d^n x \, |\nabla u|^2\eta^2 
+\int_{\bbR^n} d^n x \, |u\nabla \eta|^{2}\right), 
\end{align}
where in the next-to-last step we have used Sobolev embedding.
Similarly,
\begin{align} 
|V| & \lesssim \eps_0t^2\left(\int_{\bbR^n} d^n x \, | u|^2|\nabla\eta|^2 
+ \left(\int_{\bbR^n} d^n x \, |u\eta|^{2^*}\right)^{2/2^*}\right)   \no \\
& \lesssim\, \eps_0t^2\left(\int_{\bbR^n} d^n x \, |\nabla u|^2\eta^2 
+ \int_{\bbR^n} d^n x \, |u\nabla \eta|^{2}\right). 
\end{align} 
For $\eps_0$ small enough, we may hide the integral involving $\nabla u$ 
(which, by virtue of our previously established global estimate \eqref{eq4.resolventgrad},
is finite) on the left-hand side of \eqref{eq4.35}.  

By Cauchy's inequality with $\eps's$,
we have that
\begin{equation} 
|IV|\lesssim \eps t^2 \int_{\bbR^n} d^n x \, |\nabla u|^2\eta^2 
+ \frac1{\eps} t^2\int_{\bbR^n} d^n x \, |u|^2|\nabla\eta|^2, 
\end{equation} 
where $\eps>0$ is at our disposal.  Taking $\eps$ small enough, depending only upon ellipticity,
we may, as before, hide the integral involving $\nabla u$
on the left-hand side of \eqref{eq4.35}.  Collecting our various estimates, we obtain \eqref{eq4.34}.

Let us now deduce \eqref{eq4.gaff} as a consequence of \eqref{eq4.34}. 
Recall that $\eta$ is supported on a $d/2$ neighborhood of $Q_2$, which we denote by $\Q$, and observe that
$\diam \Q \approx d$ (since, in particular, $d\gtrsim \diam(Q_2)$).  Following essentially verbatim the proof of \cite[Lemma 2.1]{AHLMT02}, we apply \eqref{eq4.34}, with $\eta$ of the form
\begin{equation} 
\eta:= e^{\alpha \vp}-1, 
\end{equation} 
with $\alpha:= c_0 d/t$, where $c_0$ is a sufficiently small constant to be chosen
momentarily, and where
$0\leq \vp\in C^\infty_0(\Q)$, with $\vp\equiv 1$ on $Q_2$, satisfying $\|\nabla \vp\|_\infty \lesssim 1/d$. 
With this choice of $\eta$, and with $c_0$ chosen sufficiently small,
\eqref{eq4.34} yields in particular that
\begin{equation}
\int_{\bbR^n} d^n x \, |u|^2\left(e^{\alpha \vp}-1\right)^2 
\leq \frac14 \int_{\bbR^n} d^n x \, |u|^2|e^{\alpha \vp}|^2\,.
\end{equation}
Thus, by a simple argument involving the triangle inequality, we obtain in turn that
\begin{equation}
\int_{\bbR^n} d^n x \, |u|^2\left(e^{\alpha \vp}\right)^2 
\leq 4 \int_{\mathbb{R}^n} d^n x \, |u|^2 \lesssim 
\int_{Q_1} d^n x \, |f|^2,
\end{equation}
where in the last step we have used the case $p=2$ of \eqref{eq4.resolvent}, and the fact that
$f$ is supported on $Q_1$.  Since $\vp\equiv 1$
on $Q_2$, we then have that
\begin{equation}\label{eq4.42}
e^{2\alpha} \int_{Q_2} d^n x \, |u|^2\, \lesssim \int_{Q_1} d^n x \, |f|^2,
\end{equation}
whence the desired bound \eqref{eq4.gaff} follows for $(I+t^2M_a)^{-1}$. 

Next, let $Q_2'$ be a cube concentric with $Q_2$, having side length
$\ell(Q_2')\approx d$, such that $Q_2\subset Q_2'\subset \Q$, and such that
$\dist (Q_2',\mathbb{R}^n\backslash \Q) \approx d\approx\dist(Q_2,\R^n\backslash  Q_2')$.  
We note that the preceding argument may be applied with $Q_2'$ in place of $Q_2$ 
(indeed, we may construct $\vp$ enjoying the same properties as above, but with $\vp\equiv 1$
on $Q_2'$), so in particular \eqref{eq4.42} holds for $Q_2'$.  Thus,
\begin{equation}\label{eq4.43}
 \int_{Q'_2} d^n x \, |u|^2\, \lesssim \, e^{-cd/t}
\int_{Q_1} d^n x \, |f|^2\,,
\end{equation}
by definition of $\alpha$.  We now choose another cut-off 
$\vpt\in C^\infty_0(Q_2')$, 
such that $\|\nabla\vpt\|_\infty
\lesssim 1/d$, with $0\leq\vpt\leq 1$, and
$\vpt\equiv 1$ on $Q_2$.
We apply \eqref{eq4.34} with $\eta = \vpt$, to obtain in particular that
\begin{align}
\int_{Q_2} d^n x \, |t\nabla u|^2 & \leq \int_{\bbR^n} d^n x \, |t\nabla u|^2(\vpt)^2  
\lesssim t^2 \int_{\bbR^n} d^n x \, |u|^2|\nabla \vpt|^2
\lesssim \left(\frac{t}{d}\right)^2 \int_{Q'_2} d^n x \, |u|^2    \no \\
& \lesssim e^{-cd/t} \int_{Q_1} d^n x \, |f|^2,
\end{align}
where in the last inequality we have used
\eqref{eq4.43} and the fact that $t\lesssim d$. 
We therefore have that  \eqref{eq4.gaff} holds for $t\nabla (I+t^2M_a)^{-1}$, and thus by a duality argument (cf. Remark \ref{r4.7} above),
it holds also for $t(I+t^2M_a)^{-1}\dv(\cdot)$.  This concludes the proof of Lemma \ref{l4.4}, when $n\geq 3$.

The proof in the two-dimensional case is somewhat different:  we shall use a  resolvent expansion
as in the proof of the global estimates treated above.  This approach does not seem to yield the exponential decay that we obtained in higher dimensions, but nonetheless, as noted previously,
polynomial decay of sufficiently high order suffices for our purposes.  Let us first record some 
known estimates for the resolvent $(I+t^2L_a)^{-1}$ in two dimensions, which follow readily from 
two-dimensional estimates for the heat semigroup $e^{-sL}$ proved in \cite{AMT98}, and the formula
$(I+L)^{-1}=\int_0^\infty ds \, e^{-s}e^{-sL}$.   Let $E,F\subset \R^n$ be disjoint sets, with 
\begin{equation} 
d:=\dist(E,F) \geq t>0.
\end{equation}   
We have
\begin{align}\label{eq5.16}
\begin{split} 
& \big\|(I+t^2L_a)^{-1} \big\|_{\cB(L^1(E),L^2(F))} 
+ \big\|t\nabla(I+t^2L_a)^{-1} \big\|_{\cB(L^1(E),L^2(F)^n)}  \\
& \quad \leq Ct^{-1} e^{-cd/t} \leq C\ d^{-1} (t/d)^4,
\end{split} 
\end{align}
and also
\begin{align}\label{eq5.17}
& \big\|(I+t^2L_a)^{-1} \big\|_{\cB(L^2(E),L^2(F))} 
+ \big\|t\nabla(I+t^2L_a)^{-1} \big\|_{\cB(L^2(E),L^2(F)^n)}
\leq Ce^{-cd/t}    \no \\
& \quad \leq C  (t/d)^4,
\end{align}
where as usual the various constants depend only upon $a_1$ and $a_2$. 
We now claim that there is a constant $C_0$ such that for $f$ supported on $E$,
\begin{align}\label{eq5.18}
\begin{split} 
\mu(k) &:= \Big\|(I+t^2L_a)^{-1}\big[-t^2\B\cdot \nabla(I+t^2L_a)^{-1}\big]^k f \Big\|_{L^2(F)}   \\ 
& \leq C_0^k \, \eps_0^k \, (t/d)^4\,\|f\|_{L^2(E)},
\end{split} 
\end{align}
and
\begin{align}\label{eq5.19}
\begin{split} 
\nu(k) & := \Big\|t\nabla(I+t^2L_a)^{-1}\big[-t^2\B\cdot \nabla(I+t^2L_a)^{-1}\big]^k f \Big\|_{L^2(F)^n}    \\
& \leq\, 
C_0^k\,\eps_0^k\, (t/d)^4\,\|f\|_{L^2(E)}.
\end{split} 
\end{align}
Given the claim, and choosing $\eps_0<1/C_0$, we obtain \eqref{eq4.gaff2d} immediately from 
\eqref{eq4.expansion2d} and Remark \ref{r4.7}.  

We now establish this claim by induction.  Observe that the case $k=0$ is simply \eqref{eq5.17}.
Suppose now that \eqref{eq5.18} and \eqref{eq5.19} hold for some $k\geq 0$, for every
set $E$ and $F$ as above, with $d\geq t$, and we wish to verify
the corresponding estimates for $k+1$.   Let $F_1,\,F_2$ denote, respectively,  $d/3$ and $2d/3$
neighborhoods of $F$, and choose smooth a smooth cut-off $\vp$ such that
\begin{equation} 
\vp \equiv 1 \, \text{ on }\, F_1, \quad \supp \, (\vp)\subset F_2,  
\end{equation} 
with $\|\nabla \vp\|_{L^{\infty}(\bbR^n)^n} \lesssim 1/d$.
We may suppose that $t\leq d/3$, since the case $d\geq t>d/3$ follows from the global bound \eqref{eq5.10}.
We then have, for instance, that
\begin{align}
\mu(k+1) &= \Big\|(I+t^2L_a)^{-1}\big(-t^2\B\cdot \nabla(I+t^2L_a)^{-1}\big)   \no \\ 
& \qquad \times \big[-t^2\B\cdot \nabla(I+t^2L_a)^{-1}\big]^kf \Big\|_{L^2(F)}    \no \\   
& \leq \Big\|(I+t^2L_a)^{-1}\big(-t^2\B\cdot \nabla\big(\vp(I+t^2L_a)^{-1}\big)\big)   \no \\ 
& \qquad \times \big[-t^2\B\cdot \nabla(I+t^2L_a)^{-1}\big]^kf \Big\|_{L^2(F)}     \no \\
& \quad + 
\Big\|(I+t^2L_a)^{-1}\big(-t^2\B\cdot \nabla\big((1-\vp)(I+t^2L_a)^{-1}\big)\big)    \no \\ 
& \qquad \quad \times \big[-t^2\B\cdot \nabla(I+t^2L_a)^{-1}\big]^k f \Big\|_{L^2(F)}     \no \\ 
&=:\,I+II.
\end{align}
By \eqref{eq5.16}, since $1-\vp$ is supported on  $E_1:=\R^n\backslash  F_1$,
with $\dist(E_1,F)=d/3$, we have that
\begin{align}
II &\leq C (t/d)^5 \Big\|t\B\cdot \nabla \big((1-\vp)(I+t^2L_a)^{-1}\big)  
\big[-t^2\B\cdot \nabla(I+t^2L_a)^{-1}\big]^k f \Big\|_{L^1(\R^n\backslash  F_1)}   \no \\ 
& \leq C\eps_0 (t/d)^5 \Big\|t \nabla\left((1-\vp)(I+t^2L_a)^{-1}\right)\,
\big[-t^2\B\cdot \nabla(I+t^2L_a)^{-1}\big]^k f \Big\|_{L^2(\R^n)^n}   \no \\
& \leq C\eps_0 (t/d)^6 \Big\|(I+t^2L_a)^{-1}\,
\big[-t^2\B\cdot \nabla(I+t^2L_a)^{-1}\big]^k f \Big\|_{L^2(\R^n)}  \no \\
& \quad + C\eps_0 (t/d)^5 \Big\|t \nabla(I+t^2L_a)^{-1}\,
\big[-t^2\B\cdot \nabla(I+t^2L_a)^{-1}\big]^k f \Big\|_{L^2(\R^n)^n}   \no \\
& \leq C^{k+1} \, \eps_0^{k+1} \, (t/d)^5\|f\|_{L^2(E)},
\end{align}
where in the second inequality we have used H\"older's inequality and \eqref{eq4.smallness},
in the third inequality that $\|\nabla\vp\|_{L^{\infty}(\bbR^n)^n} \lesssim 1/d$,
and in the last inequality we have used the global bound \eqref{eq5.10}.

Turning to term $I$, we have by \eqref{eq4.resolvent2dL} and \eqref{eq2dclms},
since $\supp \, (\vp) \subset F_2$, with $\|\nabla\vp\|_{L^{\infty}(\bbR^n)^n} \lesssim 1/d$, that
\begin{align}
I &\leq C \eps_0 \Big\|t \nabla\big(\vp(I+t^2L_a)^{-1}\big) \, 
\big[-t^2\B\cdot \nabla(I+t^2L_a)^{-1}\big]^k f \Big\|_{L^2(F_2)^n}   \no \\
& \leq C \eps_0 \Big\|t \nabla(I+t^2L_a)^{-1}\,
\big[-t^2\B\cdot \nabla(I+t^2L_a)^{-1}\big]^k f \Big\|_{L^2(F_2)^n}     \no \\
& \quad + C \eps_0\, (t/d)\, \Big\|(I+t^2L_a)^{-1}\,
\big[-t^2\B\cdot \nabla(I+t^2L_a)^{-1}\big]^k f \Big\|_{L^2(F_2)}    \no \\ 
& \leq 2 C \, C_0^{k} \, \eps_0^{k+1} \, (3t/d)^4 \|f\|_{L^2(E)},
\end{align}
by the induction hypothesis, since $\dist(F_2,E)\geq d/3$.  Combining our estimates
for terms $I$ and $II$, we obtain the case $k+1$ of \eqref{eq5.18}, for a suitable choice of $C_0$.
By a very similar argument, one may establish the case $k+1$ of \eqref{eq5.19}.
We omit further details.
\end{proof}

\begin{proof}[Proof of Lemma \ref{l4.5*}]
We recall that $H_a := L_a +\B_1\cdot\nabla$.  Set
\begin{equation} 
u:= - H_a^{-1}\dv \vec{f}\,, 
\end{equation} 
where $\vec{f}\in L^2(\mathbb{R}^n)^n$, i.e., 
$u$ solves the problem
\begin{equation}\label{eq4.125}
\begin{cases}
H_a u = -\dv \vec{f}\,\,\,\, {\rm in}\,\, \mathbb{R}^n, \\[1mm] 
\nabla u \in L^2(\mathbb{R}^n), 
\end{cases}
\end{equation}
in the weak sense.
For $\eps_0$ small enough, depending only on dimension and the ellipticity of $a$, the latter problem
is well-posed, by perturbation of the well known corresponding result for $L_a$.
Moreover
\begin{align}
& \int_{\mathbb{R}^n} d^nx \, |\nabla u|^2 
\lesssim \Re \bigg(\int_{\mathbb{R}^n} d^nx \, \langle A\nabla u, \nabla u \rangle\bigg)   \no \\
& \quad = \Re \bigg(\int_{\mathbb{R}^n} d^nx \, \langle A\nabla u, \nabla u \rangle\bigg) +
\Re \bigg(\int_{\mathbb{R}^n} \B_1\cdot\nabla u \,\overline{u}\bigg) 
- \Re \bigg(\int_{\mathbb{R}^n} d^nx \, \B_1\cdot\nabla u \,\overline{u}\bigg)   \no \\
& \quad = \Re \int_{\mathbb{R}^n} d^nx \, \langle \vec{f},\nabla u\rangle \,+\, 
 \Oh\big(\|\B_1\|_{L^n(\bbR^n)^n} \|\nabla u\|_{L_2(\bbR^n)^n} \|u\|_{L^{2^*}(\bbR^n)}\big) \no \\ 
& \quad \lesssim \eps^{-1} \|\vec{f}\|^2_{L^2(\bbR^n)^n} + \eps \|\nabla u\|_{L^2(\bbR^n)^n}^2 
+ \eps_0 \|\nabla u\|_{L^2(\bbR^n)^n}^2\,,
\end{align}
where in the third line we have used  the definition of weak solution to \eqref{eq4.125}, and \ref{eq4.12a}
with $\B_1$ in place of $B_2$, and in the fourth line we have used Cauchy's inequality
and Sobolev embedding.  Of course, $\eps$ is at our disposal.   We note that the implicit constants depend only on dimension and ellipticity of $a$.  For
$\eps$ and $\eps_0$ small enough, depending only on these implicit constants, we may then
hide the small terms on the left hand side of the inequality to obtain \eqref{eq4.11*}. 
\end{proof}

\appendix

\section{Additional Facts on Kato's Resolvent Approach} \lb{sA} 

In this appendix we further investigate $T$ as defined in \eqref{2.4} and \eqref{2.5} and also make 
the connection between $T$ and the form sum of $T_0$ and $B^* A$ under stronger hypotheses 
on $T_0$, $A$, and $B$ compared to those in Hypothesis \ref{h2.1}. 

For relevant background literature in this connection we refer, for instance, to \cite[Ch.\ IV]{EE89}, 
\cite{GLMZ05}, and \cite{Ka66} (see also \cite{Fo92}, \cite{Sc72}, \cite[Ch.\ II]{Si71} in the context 
of self-adjoint operators). 
 
We start by briefly complementing \eqref{2.4}:
 
\begin{lemma} \lb{lA.1} 
Assume Hypothesis \ref{h2.1}.  If $T$ is defined by \eqref{2.5}, then $T^*$ is given in terms of 
$T_0^*$, $A$, and $B$ by 
\begin{align}
(T^*-zI_{\cH})^{-1} &= (T_0^*-zI_{\cH})^{-1}   \no \\
& \quad -\overline{(T_0^*-zI_{\cH})^{-1}A^*}\big[I_{\cK}-\widetilde{K}(z)\big]^{-1} 
B(T_0^*-zI_{\cH})^{-1},\lb{A.1}\\
&\hspace*{3.4cm}z\in \big\{\zeta\in \rho(T_0^*) \, \big| \, 1\in \rho(\widetilde{K}(\zeta))\big\},\no
\end{align}
where
\begin{equation}\lb{A.2}
\widetilde{K}(z)=-\overline{B(T_0^*-zI_{\cH})^{-1}A^*}, \quad z\in \rho(T_0^*).
\end{equation}
Therefore, $T^*$ is defined via Kato's resolvent equation by replacing $T_0$ by $T_0^*$ and by interchanging the roles of $A$ and $B$.
\end{lemma}

\begin{proof}
Taking adjoints on both sides of \eqref{2.6}, and making use of the identities
\begin{align}
\big[A(T_0-zI_{\cH})^{-1}\big]^*&= \big[(T_0^*-\overline{z}I_{\cH})^{-1}A^*)^*\big]^* 
=\overline{(T_0^*-\overline{z}I_{\cH})^{-1}A^*}\lb{A.3}\\
\big[\overline{(T_0-zI_{\cH})^{-1}B^*}\big]^*&=\big[(T_0-zI_{\cH})^{-1}B^*\big]^*
=B(T_0^*-\overline{z}I_{\cH})^{-1},\lb{A.4}
\end{align}
one obtains the following representation of the resolvent of $T_0^*$:
\begin{align}
&(T^*-\overline{z}I_{\cH})^{-1}\no\\
&=(T_0^*-\overline{z}I_{\cH})^{-1}-\big[A(T_0-zI_{\cH})^{-1}\big]^*
\big([I_{\cK}-K(z)]^{-1} \big)^*\big[\overline{(T_0-zI_{\cH})^{-1}B^*}\big]^*\no\\
&=(T_0^*-\overline{z}I_{\cH})^{-1}-\overline{(T_0^*-\overline{z}I_{\cH})^{-1}
A^*}(I_{\cK}-K(z)^*)^{-1}B(T_0-\overline{z}I_{\cH})^{-1},\lb{A.5}\\
&\hspace*{5.95cm}z\in \{\zeta\in \rho(T_0)|1\in \rho(K(\zeta))\}.\no
\end{align}
Finally, a short computation with adjoints and closures shows
\begin{equation}\lb{A.6}
K(z)^*=\widetilde{K}(\overline{z}), \quad z\in \rho(T_0).
\end{equation}
Thus, \eqref{A.5} and \eqref{A.6} combine to yield \eqref{A.1}.
\end{proof}

Before stating the principal result of this appendix, Theorem \ref{tA.3} below, we need to recall some preparations. The material needed can be found in  \cite[Sect.\ IV.2]{EE89},\cite[Sects.\ 1.2--1.5]{Fa75}, 
and \cite[Sect.\ VI.2.6]{Ka80}, and we here follow the summary presented in \cite[Sect.\ 2]{GM09}, 
except, the latter focused exclusively on the self-adjoint situation, whereas here we extend it to the 
non-self-adjoint (sectorial) context: 
Let $\cV$ be a reflexive Banach space continuously and densely embedded
into $\cH$. Then also $\cH$ embeds continuously and densely into $\cV^*$.
That is,
\begin{equation}
\cV  \hookrightarrow \cH  \hookrightarrow \cV^*.     \lb{B.1}
\end{equation}
Here the continuous embedding $\cH\hookrightarrow \cV^*$ is accomplished via
the identification
\begin{equation}
\cH \ni v \mapsto (\dott,v)_{\cH} \in \cV^*,     \lb{B.2}
\end{equation}
and we use the convention in this manuscript that if $\cX$ denotes a Banach space, 
$\cX^*$ denotes the {\it adjoint space} of continuous  conjugate linear functionals on $X$,
also known as the {\it conjugate dual} of $\cX$.

In particular, if the sesquilinear form
\begin{equation}
{}_{\cV}\langle \dott, \dott \rangle_{\cV^*} \colon \cV \times \cV^* \to \bbC
\end{equation}
denotes the duality pairing between $\cV$ and $\cV^*$, then
\begin{equation}
{}_{\cV}\langle u,v\rangle_{\cV^*} = (u,v)_{\cH}, \quad u\in\cV, \;
v\in\cH\hookrightarrow\cV^*,   \lb{B.3}
\end{equation}
that is, the $\cV, \cV^*$ pairing
${}_{\cV}\langle \dott,\dott \rangle_{\cV^*}$ is compatible with the
scalar product $(\dott,\dott)_{\cH}$ in $\cH$.

Let $W \in\cB(\cV,\cV^*)$. Since $\cV$ is reflexive, $(\cV^*)^* = \cV$, one has
\begin{equation}
W \colon \cV \to \cV^*, \quad  W^* \colon \cV \to \cV^*   \lb{B.4}
\end{equation}
and
\begin{equation}
{}_{\cV}\langle u, Wv \rangle_{\cV^*}
= {}_{\cV^*}\langle W^* u, v\rangle_{(\cV^*)^*}
= {}_{\cV^*}\langle W^* u, v \rangle_{\cV}
= \ol{{}_{\cV}\langle v, W^* u \rangle_{\cV^*}}.
\end{equation}
{\it Self-adjointness} of $W$ is then defined by $W = W^*$, that is,
\begin{equation}
{}_{\cV}\langle u,W v \rangle_{\cV^*}
= {}_{\cV^*}\langle W u, v \rangle_{\cV}
= \ol{{}_{\cV}\langle v, W u \rangle_{\cV^*}}, \quad u, v \in \cV,    \lb{B.5}
\end{equation}
{\it nonnegativity} of $W$ is defined by
\begin{equation}
{}_{\cV}\langle u, W u \rangle_{\cV^*} \geq 0, \quad u \in \cV,    \lb{B.6}
\end{equation}
and {\it boundedness from below of $W$ by $c_W \in\bbR$} is defined by
\begin{equation}
{}_{\cV}\langle u, W u \rangle_{\cV^*} \geq c_W \|u\|^2_{\cH},
\quad u \in \cV.
\lb{B.6a}
\end{equation}
(By \eqref{B.3}, this is equivalent to
${}_{\cV}\langle u, W u \rangle_{\cV^*} \geq c_W \,
{}_{\cV}\langle u, u \rangle_{\cV^*}$, $u \in \cV$.)

Next, let the sesquilinear form $\ga(\dott,\dott)\colon\cV \times \cV \to \bbC$
(antilinear in the first and linear in the second argument) be
{\it $\cV$-bounded}, that is, there exists a $c_a>0$ such that
\begin{equation}
|\ga(u,v)| \le c_a \|u\|_{\cV} \|v\|_{\cV},  \quad u, v \in \cV.
\end{equation}
Then $\wti A$ defined by
\begin{equation}
\wti A \colon \begin{cases} \cV \to \cV^*, \\
\, v \mapsto \wti A v = \ga(\dott,v), \end{cases}    \lb{B.7}
\end{equation}
satisfies
\begin{equation}
\wti A \in\cB(\cV,\cV^*) \, \text{ and } \,
{}_{\cV}\big\langle u, \wti A v \big\rangle_{\cV^*}
= \ga(u,v), \quad  u, v \in \cV.    \lb{B.8}
\end{equation}
Assuming further that $\ga(\dott,\dott)$ is {\it symmetric}, that is,
\begin{equation}
\ga(u,v) = \ol{\ga(v,u)},  \quad u,v\in \cV,    \lb{B.9}
\end{equation}
and that $\ga$ is {\it $\cV$-coercive}, that is, there exists a constant
$C_0>0$ such that
\begin{equation}
\ga(u,u)  \geq C_0 \|u\|^2_{\cV}, \quad u\in\cV,    \lb{B.10}
\end{equation}
respectively, then,
\begin{equation}
\wti A \colon \cV \to \cV^* \, \text{ is bounded, self-adjoint, and boundedly
invertible.}    \lb{B.11}
\end{equation}
Moreover, denoting by $A$ the part of $\wti A$ in $\cH$ defined by
\begin{align}
\dom(A) = \big\{u\in\cV \,|\, \wti A u \in \cH \big\} \subseteq \cH, \quad
A= \wti A\big|_{\dom(A)}\colon \dom(A) \to \cH,   \lb{B.12}
\end{align}
then $A$ is a (possibly unbounded) self-adjoint operator in $\cH$ satisfying
\begin{align}
& A \geq C_0 I_{\cH},   \lb{B.13}  \\
& \dom\big(A^{1/2}\big) = \cV.  \lb{B.14}
\end{align}
In particular,
\begin{equation}
A^{-1} \in\cB(\cH).   \lb{B.15}
\end{equation}
The facts \eqref{B.1}--\eqref{B.15} are a consequence of the Lax--Milgram
theorem and the 2nd representation theorem for symmetric sesquilinear forms.
Details can be found, for instance, in \cite[Sects.\ VI.3, VII.1]{DL00},
\cite[Ch.\ IV]{EE89}, and \cite{Li62}.

Next, we extend the scope of this discussion and consider a {\it sectorial} form 
$\gb(\dott,\dott)\colon \cV\times\cV\to\bbC$ defined by the existence of $c_{\gb} \in \bbR$ and 
$\theta_{\gb} \in [0, \pi/2)$ such that 
\begin{equation}
\{\gb(u,u) \in \bbC \,|\, u \in \cV, \, \|u\|_{\cH} = 1\} \subseteq 
\{\zeta \in \bbC \,|\, |\arg(\zeta - c_{\gb})|\leq \theta_{\gb}\}.     \lb{B.16} 
\end{equation}
Denoting 
\begin{equation}
\gb_{\gR}(u,v) = \big[\gb(u,v) + \ol{\gb(v,u)}\big]/2, \quad u, v \in \cV,     \lb{B.17}
\end{equation} 
this yields 
\begin{equation}
\gb_{\gR}(u,u) \geq c_{\gb} \|u\|_{\cH}^2, \quad u\in\cV.  \lb{B.19}
\end{equation}
Introducing the scalar product
$(\dott,\dott)_{\cV_{\gb}^{}} \colon \cV\times\cV\to\bbC$
(and the associated norm $\|\cdot\|_{\cV_{\gb}^{}} $) by
\begin{equation}
(u,v)_{\cV_{\gb}^{}}  = \gb_{\gR}(u,v) + (1- c_{\gb})(u,v)_{\cH}, \quad u,v \in \cV,  \lb{B.20}
\end{equation}
turns $\cV$ into a pre-Hilbert space $(\cV; (\dott,\dott)_{\cV_{\gb}^{}} )$,
which we denote by
$\cV_{\gb}^{}$. The form $\gb$ is called {\it closed} in $\cH$ if $\cV_{\gb}^{}$ is actually
complete, and hence a Hilbert space. The form $\gb$ is called {\it closable}
in $\cH$ if it has a closed extension in $\cH$. The estimates, 
\begin{equation}
|\gb_{\gR}(u,v) + (1- c_{\gb})(u,v)_{\cH}| \le \|u\|_{\cV_{\gb}^{}}  \|v\|_{\cV_{\gb}^{}} ,
\quad u,v\in \cV,
\lb{B.21}
\end{equation}
and
\begin{equation}
\big|\gb_{\gR}(u,u) + (1 - c_{\gb})\|u\|_{\cH}^2\big| = \|u\|_{\cV_{\gb}^{}} ^2, \quad u \in \cV,
\lb{B.22}
\end{equation}
show that if $\gb$ is closed in $\cH$, then the form $\gb_{\gR}(\dott,\dott)+(1 - c_{\gb})(\dott,\dott)_{\cH}$ 
is a symmetric, $\cV$-bounded, and $\cV$-coercive sesquilinear form and hence the analog of 
\eqref{B.7}--\eqref{B.15} applies. Similarly, one has the bound 
\begin{equation}
|\gb(u,v) + (1- c_{\gb})(u,v)_{\cH}| \le [1 + \tan(\theta_{\gb})] \|u\|_{\cV_{\gb}^{}}  \|v\|_{\cV_{\gb}^{}} ,
\quad u,v\in \cV.      \lb{B.22a}
\end{equation}
Hence, if $\gb$ is closed, the 1st representation theorem for forms (cf.\ \cite[Theorem\ IV.2.4]{EE89}, 
\cite[Theorem\ VI.2.1]{Ka80}) yields the existence of a linear map
\begin{equation}
\wti B_{c_{\gb}} \colon \begin{cases} \cV_{\gb}^{} \to \cV_{\gb}^*, \\ 
\, v \mapsto \wti B_{c_{\gb}} v = \gb(\dott,v) +(1 - c_{\gb})(\dott,v)_{\cH},
\end{cases}
\lb{B.23}
\end{equation}
with
\begin{equation}
\wti B_{c_{\gb}} \in\cB(\cV_{\gb},\cV_{\gb}^*) \, \text{ and } \,
{}_{\cV_{\gb}^{}} \big\langle u, \wti B_{c_{\gb}} v \big\rangle_{\cV_{\gb}^*}
= \gb(u,v)+(1 -c_{\gb})(u,v)_{\cH}, \quad  u, v \in \cV.    \lb{B.24}
\end{equation}
Introducing the linear map
\begin{equation}
\wti B = \wti B_{c_{\gb}} + (c_{\gb} - 1)\wti I \colon \cV_{\gb}\to\cV_{\gb}^*,
\lb{B.24a}
\end{equation}
where $\wti I\colon \cV_{\gb}\hookrightarrow\cV_{\gb}^*$ denotes
the continuous inclusion (embedding) map of $\cV_{\gb}^{}$ into $\cV_{\gb}^*$, one
obtains an m-sectorial operator $B$ in $\cH$ by restricting $\wti B$ to $\cH$,
\begin{align}
\dom(B) = \big\{u\in\cV \,\big|\, \wti B u \in \cH \big\} \subseteq \cH, \quad
B= \wti B\big|_{\dom(B)}\colon \dom(B) \to \cH,   \lb{B.25}
\end{align}
satisfying the following properties:
\begin{align}
& \gb(u,v) = {}_{\cV_{\gb}^{}} \big\langle u, \wti B v \big\rangle_{\cV_{\gb}^*},
\quad u, v \in \cV, \lb{B.28a} \\
& \gb(u,v) = (u, Bv)_{\cH}, \quad  u\in \cV, \; v \in\dom(B),  \lb{B.29} \\
& \dom(B) = \{v\in\cV\,|\, \text{there exists an $f_v\in\cH$ such that}  \no \\
& \hspace*{3.05cm} \gb(w,v)=(w,f_v)_{\cH} \text{ for all $w\in\cV$}\},
\lb{B.30} \\
& Bu = f_u, \quad u\in\dom(B),  \no \\
& \dom(B) \, \text{ is dense in $\cH$ and in $\cV_{\gb}^{}$}.  \lb{B.31}
\end{align}
Properties \eqref{B.30} and \eqref{B.31} uniquely determine $B$; the operator $B$ is called 
{\it the operator associated with the form $\gb$}; conversely, the form $\gb$ is called {\it the form 
associated with $B$}.

Finally, we also recall the following perturbation theoretic fact:
Suppose that $\ga$ is a sectorial and closed form with domain $\cV \subseteq \cH$, and let $\gb$ 
be a sesquilinear form bounded with respect to $\ga$ with bound less than one, that is,
$\dom(\gb)\supseteq \cV$, and that there exist $0\le \alpha < 1$ and
$\beta\ge 0$ such that
\begin{equation}
|\gb(u,u)| \le \alpha |\ga(u,u)| + \beta \|u\|_{\cH}^2, \quad u\in \cV.   \lb{B.45}
\end{equation}
Then
\begin{equation}
(\ga + \gb)\colon \begin{cases} \cV\times\cV \to \bbC, \\
\hspace*{.12cm} (u,v) \mapsto (\ga + \gb)(u,v) = \ga(u,v) + \gb(u,v) \end{cases}
\lb{B.46}
\end{equation}
defines a sectorial and closed form (cf., e.g., \cite[Sect.\,VI.1.6]{Ka80}). In the special 
case where $\alpha$ can be chosen arbitrarily small, the form $\gb$ is called {\it infinitesimally
form bounded with respect to $\ga$}. In the case where $T_{\ga}$ is the m-sectorial operator 
uniquely associated with the form $\ga$, and also $\gb$ is sectorial with uniquely associated 
m-sectorial operator $T_{\gb}$, the m-sectorial operator uniquely associated with the form sum 
$\ga + \gb$ in \eqref{B.46} will be denoted by $T_{\ga + \gb}$ as well as by 
$T_{\ga} +_{\gq} T_{\gb}$, that is, we agree to use the notation 
$T_{\ga + \gb} = T_{\ga} +_{\gq} T_{\gb}$ in this situation.

Given these preliminaries, we now turn to the connection between Kato's 
resolvent equation \eqref{2.4} defining $T$ and the form sum of $T_0$ and an extension of 
$B^*A$ under somewhat stronger hypotheses on $T_0$, $A$, and $B$ than those in 
Hypothesis \ref{h2.1}. 

We start with the following basic assumptions:

\begin{hypothesis} \lb{hA.2}
$(i)$ Suppose $T_0$ is m-sectorial and 
\begin{equation}\lb{A.7}
\dom\big(T_0^{1/2}\big)=\dom\big((T_0^*)^{1/2}\big).
\end{equation}
$(ii)$ Assume that $A:\dom(A)\rightarrow \cK$, $\dom(A)\subseteq \cH$, is a closed, linear operator 
from $\cH$ to $\cK$, and $B:\dom(B)\rightarrow \cK$, $\dom(B)\subseteq \cH$, is a closed, linear 
operator from $\cH$ to $\cK$ such that 
\begin{equation}\lb{A.8}
\dom(A)\supseteq \dom\big(T_0^{1/2}\big), \quad \dom(B)\supseteq \dom\big(T_0^{1/2}\big),
\end{equation}
and that for some $0 \leq a < 1$, $b \geq 0$, 
\begin{align}
\begin{split} 
& \|A f\|_{\cK}^2 \leq a \Re[\gq^{}_{T_0}(f,f)] + b \|f\|_{\cH}^2, \quad f \in \dom\big(T_0^{1/2}\big),    \\
& \|B f\|_{\cK}^2 \leq a \Re[\gq^{}_{T_0}(f,f)] + b \|f\|_{\cH}^2, \quad f \in \dom\big(T_0^{1/2}\big).  
\lb{A.9}
\end{split} 
\end{align}
$(iii)$ Suppose that 
\begin{equation}
\lim_{E \uparrow \infty} \|K(-E)\|_{\cB(\cK)} = 0,    \lb{A.44} 
\end{equation} 
where
\begin{align}
\begin{split} 
K(z) &= - \overline{A(T_0-zI_{\cH})^{-1}B^*}    \\
&= - \big[A (T_0-zI_{\cH})^{-1/2}\big] 
\Big[\ol{\big[B (T_0^* - \ol{z} I_{\cH})^{-1/2}\big]^*}\Big] \in \cB(\cK),  \quad z \in \rho(T_0).  \lb{A.45} 
\end{split} 
\end{align}
\end{hypothesis}

In the following we denote by $\gq^{}_{T_0}$ the sectorial form associated with $T_0$ and note 
that according to Remark \ref{r2.14}, assumption \eqref{A.7} yields that 
\begin{equation}
\dom\big(T_0^{1/2}\big) = \dom(\gq^{}_{T_0}) = \dom\big((T_0^*)^{1/2}\big).   \lb{A.46}
\end{equation} 
In addition, introducing the form 
\begin{equation}
\gq^{}_W (f,g) = (Bf, Ag)_{\cK}, \quad f, g \in \dom\big(T_0^{1/2}\big)     \lb{A.47}
\end{equation}
(formally corresponding to the operator $B^* A$), it is clear from assumption \eqref{A.9} that 
$\gq^{}_W$ is relatively bounded with respect to $\gq^{}_{T_0}$ with bound strictly less than one 
(cf.\ \eqref{B.45}). Thus, we may introduce the sectorial form 
\begin{align}
\begin{split} 
& \gq^{}_{T_0 +_{\gq} W}(f,g) = \gq^{}_{T_0}(f,g) + \gq^{}_W (f,g), \\
& f,g \in \dom(\gq^{}_{T_0 +_{\gq} W}) = \dom(\gq^{}_{T_0}) = \dom\big(T_0^{1/2}\big),   \lb{A.48}
\end{split} 
\end{align}
and hence denote the m-sectorial operator uniquely associated with $\gq^{}_{T_0 +_{\gq} W}$ 
by $T_0 +_{\gq} W$ in the following. 

The principal result of this appendix, establishing equality between $T$ defined according to 
Kato's method \eqref{2.4}, \eqref{2.5} on one hand, and the form sum $T_0 +_{\gq} W$ on 
the other, a result of interest in its own right, is proved next:
 
\begin{theorem} \lb{tA.3}
Assume Hypothesis \ref{hA.2}. Then $T$ defined as in 
\eqref{2.4} and \eqref{2.5} coincides with the form sum of $T_0 +_{\gq} W$ in $\cH$, 
\begin{equation}
T = T_0 +_{\gq} W.     \lb{A.49} 
\end{equation}
\end{theorem}
\begin{proof}
The strategy of proof (closely following the one in \cite[Theorem\ II.34]{Si71} in connection with Rollnik 
potentials on $\bbR^3$) is to show that the resolvents of $T$ and $T_0 +_{\gq} W$ both coincide 
with the right-hand side of \eqref{2.4}.  

For this purpose we introduce $\cV_{T_0}^{}$ and $\cV_{T_0}^*$ in analogy to $\cV_{\gb}^{}$ 
and $\cV_{\gb}^*$ with the form $\gb$ replaced by $\gq^{}_{T_0}$ in \eqref{B.16}--\eqref{B.22a}, 
and similarly, the extension $\wti{T_0} : \cV_{T_0}^{} \to \cV_{T_0}^*$ of $T_0 : \dom(T_0) \to \cH$ 
as in \eqref{B.24a}. Moreover, since for some fixed $C > 0$, 
\begin{equation}
|\gq^{}_W (f,g)| \leq \|Bf\|_{\cK} \|Ag\|_{\cK} 
\leq C \|f\|_{\cV_{T_0}^{}} \|g\|_{\cV_{T_0}^{}}, \quad f, g \in \dom\big(T_0^{1/2}\big),    \lb{A.50}
\end{equation}
there is a natural map 
\begin{align}
\begin{split} 
& \wti W \in \cB(\cV_{T_0}^{}, \cV_{T_0}^*), \\ 
& \big(\wti W g\big) (f) = {}_{\cV_{T_0}^{}} \big\langle f, \wti W g \big\rangle_{\cV_{T_0}^*}
= (Bf, Ag)_{\cK}, \quad f, g \in \cV_{T_0}^{},    \lb{A.51}
\end{split} 
\end{align}  
extending $B^* A$, such that 
\begin{align}
& \wti{T_0 +_{\gq} W} = \wti{T_0} + \wti W \in \cB(\cV_{T_0}^{}, \cV_{T_0}^*),     \no \\ 
& \big(\big(\wti{T_0 +_{\gq} W}\big) g\big) (f) = \big(\wti{T_0} g\big) (f) + \big(\wti W g\big) (f) 
= {}_{\cV_{T_0}^{}} \big\langle f, \wti{T_0} g \big\rangle_{\cV_{T_0}^*} +
{}_{\cV_{T_0}^{}} \big\langle f, \wti W g \big\rangle_{\cV_{T_0}^*}     \no  \\ 
& \quad = \gq^{}_{T_0}(f,g) + (Bf, Ag)_{\cK}, \quad f, g \in \cV_{T_0}^{}.    \lb{A.52}
\end{align} 
In addition, we note that $B^* : \dom(B^*) \to \cH$, $\dom(B^*) \subseteq \cK$, extends to 
\begin{align}
& \wti{B^*} \in \cB(\cK, \cV_{T_0}^*),    \no \\ 
& \big(\wti{B^*} k\big) (f) = {}_{\cV_{T_0}^{}} \big\langle f, \wti{B^*} k \big\rangle_{\cV_{T_0}^*} 
= \big((T_0 + c I_{\cH})^{1/2} f, [B (T_0 +c I_{\cH})^{-1/2}]^* k\big)_{\cH},    \lb{A.53} \\
& \hspace*{8.55cm}  f \in \cV_{T_0}^{}, \, k \in \cK.    \no 
\end{align}  
This implies 
\begin{equation}
\wti W = \wti{B^*} A,     \lb{A.53a} 
\end{equation}
and also 
\begin{align}
K(z) &= - \overline{A(T_0 - z I_{\cH})^{-1}B^*} = - A \big(\wti{T_0} - z \wti I \, \big)^{-1} \wti{B^*} 
\in \cB(\cK), \quad z \in \rho(T_0),     \lb{A.54}
\end{align} 
where $A : \dom(A) \to \cK$, $\dom(A) \subseteq \cH$ and 
\begin{equation}
\big(\wti{T_0} - z \wti I \, \big)^{-1} \in \cB(\cV_{T_0}^*, \cV_{T_0}^{}), \quad z \in \rho(T_0).    \lb{A.55}
\end{equation} 
Next, we introduce the operator
\begin{align}
\begin{split} 
& \hatt R(z) = (T_0 - z I_{\cH})^{-1} - \big(\wti{T_0} - z \wti I \, \big)^{-1} \wti{B^*} [I_{\cK} - K(z)]^{-1} 
A (T_0 - z I_{\cH})^{-1},     \\  
& \hspace*{6.05cm}  z\in \{\zeta \in \rho(T_0) \,|\,1 \in \rho(K(\zeta))\},    \lb{A.56} 
\end{split} 
\end{align}
and hence conclude that 
\begin{align} 
\hatt R(-E) &= (T_0 + E I_{\cH})^{-1}     \lb{A.58} \\ 
& \quad - \big(\wti{T_0} + E \wti I \, \big)^{-1} \wti{B^*} \sum_{k=0}^\infty 
(-1)^k \Big[A \big(\wti{T_0} + E \wti I \, \big)^{-1} \wti{B^*}\Big]^k A (T_0 + E I_{\cH})^{-1},   \no 
\end{align}
for $E >0$ sufficiently large such that 
\begin{equation}
\|K(-E)\|_{\cB(\cK)} < 1.    \lb{A.59}
\end{equation}
We note that the infinite sum in \eqref{A.58} is convergent in $\cB(\cK)$ and hence the mapping 
properties in \eqref{A.51}--\eqref{A.55} yield $\hatt R(-E) \in \cB(\cH)$ for $E >0$ sufficiently large. One 
then computes, taking into account telescoping of series, that
\begin{align}
& (T_0 + E I_{\cH})^{-1} \big(\wti{T_0 +_{\gq} W} + E \wti I \,\big) \bigg[(T_0 + E I_{\cH})^{-1}   \no \\
& \qquad - \big(\wti{T_0} + E \wti I \,\big)^{-1} \wti{B^*} \sum_{k=0}^N (-1)^k 
\Big[A \big(\wti{T_0} + E \wti I \, \big)^{-1} \wti{B^*}\Big]^k A (T_0 + E I_{\cH})^{-1}\bigg]    \no \\
& \quad =  (T_0 + E I_{\cH})^{-1} \big(\wti{T_0} + \wti{B^*} A + E \wti I \,\big) \bigg[(T_0 + E I_{\cH})^{-1}   \no \\
& \qquad - \big(\wti{T_0} + E \wti I \,\big)^{-1} \wti{B^*} \sum_{k=0}^N (-1)^k 
\Big[A \big(\wti{T_0} + E \wti I \, \big)^{-1} \wti{B^*}\Big]^k A (T_0 + E I_{\cH})^{-1}\bigg]  \no \\
& \quad = (T_0 + E I_{\cH})^{-1}    \no \\
& \qquad + (-1)^{N+1} \big(\wti{T_0} + E \wti I \,\big)^{-1} \wti{B^*} 
\Big[A \big(\wti{T_0} + E \wti I \, \big)^{-1} \wti{B^*}\Big]^{N+1} A (T_0 + E I_{\cH})^{-1}  \no \\
& \quad \underset{N \uparrow \infty}{\longrightarrow} (T_0 + E I_{\cH})^{-1} \, \text{ in $\cB(\cH)$-norm.}
\end{align} 
Thus,
\begin{align}
& (T_0 + E I_{\cH})^{-1} \big(\wti{T_0} + \wti{B^*} A + E \wti I \,\big) R(-E) = (T_0 + E I_{\cH})^{-1}  \no \\
& \quad = (T_0 + E I_{\cH})^{-1} \big(\wti{T_0} + \wti{B^*} A + E \wti I \,\big) 
\big(\wti{T_0} + \wti{B^*} A + E \wti I \,\big)^{-1},
\end{align}
implying 
\begin{equation}
\hatt R(-E) = \big(\wti{T_0} + \wti{B^*} A + E \wti I \,\big)^{-1} \, \text{ for $E > 0$ sufficiently large,} 
\end{equation}
since $\ker\big(\wti{T_0} + E \wti I \, \big) = \ker\big(\wti{T_0} + \wti{B^*} A + E \wti I \, \big) = \{0\}$ for 
for $E > 0$ sufficiently large. Analytic continuation with respect to $E$ then proves
\begin{equation}
\hatt R(z) = \big(\wti{T_0} + \wti{B^*} A - z \wti I \,\big)^{-1}, \quad 
z\in \{\zeta \in \rho(T_0) \,|\,1 \in \rho(K(\zeta))\}.
\end{equation}
A comparison of the right-hand sides of \eqref{2.4} and \eqref{A.56} then yields $R(z) = \hatt R(z)$ 
and hence \eqref{A.49}. 
\end{proof}

\noindent
{\bf Acknowledgments.} We are indebted to Alan McIntosh for helpful discussions and 
grateful to the anonymous referee for very helpful comments and a critical reading of our 
manuscript. 


\end{document}